\documentclass[oneside]{amsart}
\usepackage{amsmath, amssymb,pb-diagram}
\usepackage{amsthm,color}
\newtheorem{theorem}{Theorem}[section]
\newtheorem{proposition}[theorem]{Proposition}
\newtheorem{lemma}[theorem]{Lemma}
\newtheorem{corollary}[theorem]{Corollary}

\newtheorem{example}{Example}
\newtheorem{question}{Question}
\newtheorem{problem}{Problem}

\usepackage[initials]{amsrefs}
\usepackage{graphicx}
\graphicspath{{../img/}{../../img/}}

\def\bp{\partial_{+} }

\def\C{\mathbb{C} }

\def\R{\mathbb{R} }
\def\Z{\mathbb{Z} }

\def\nbd{neighborhood }
\def\nbds{neighborhoods }

\def\Sv{\mathop{\mathrm{Sing}}(v)}
\def\Pv{\mathop{\mathrm{Per}}(v)}
\def\Cv{\mathop{\mathrm{Cl}}(v)}

\title[Quotient spaces and topological invariants of flows]{Quotient spaces and topological invariants of flows}
\author{Tomoo Yokoyama}
\date{\today}

\address{Department of Mathematics, Kyoto University of Education,
1 Fujinomori, Fukakusa, Fushimi-ku Kyoto, 612-8522, Japan \\
}
\email{tomoo@kyokyo-u.ac.jp}
\thanks{The author is partially supported by the JST PRESTO  Grant Number JPMJPR16E and JSPS Kakenhi Grant Number 20K03583
}

\begin{document}

\maketitle

\begin{abstract}
We construct topological invariants, called abstract weak orbit spaces, of flows and homeomorphisms on topological spaces, to describe both gradient dynamics and recurrent dynamics. In particular, the abstract weak orbit spaces of flows on topological spaces are generalizations of both Morse graphs of flows on compact metric spaces and Reeb graphs of Hamiltonian flows with finitely many singular points on surfaces. Moreover, we show that the abstract weak orbit spaces are complete and finite for several kinds of flows on manifolds, and we state several examples whose Morse graphs are singletons but whose abstract weak orbit spaces are not singletons. In addition, we consider when the time-one map reconstructs the topology of the original flow. Therefore we show that the orbit space of a Hamiltonian flow with finitely many singular points on a compact surface is homeomorphic to the abstract weak orbit space of the time-one map by taking an arbitrarily small reparametrization, and that the abstract weak orbit spaces of a Morse flow on a compact manifold and the time-one map are homeomorphic. Furthermore, the abstract weak orbit space of a Morse flow on a closed manifold is a refinement of the CW decomposition which consists of the unstable manifolds of singular points. Though the CW decomposition of a Morse flow on a closed manifold is finite, the intersection of the unstable manifold and the stable manifold of saddles of a Morse-Smale flow on a closed manifold need not consist of finitely many connected components (or equivalently need not consist of finitely many abstract weak orbits). Therefore we study the finiteness of abstract weak orbit spaces of Morse(-Smale) flow on compact manifolds.
\end{abstract}

\section{Introduction}

In 1927, Birkhoff introduced the concepts of non-wandering points and recurrent points \cite{Birkhoff}.
Using these concepts, we can describe and capture sustained or stationary dynamical behaviors and conservative dynamics.
Moreover, Conley \cite{Conley1988} defined a weak form of recurrence, called chain recurrence, for a flow on a compact metric space.
The Conley theory states that dynamical systems on compact metric spaces can be decomposed into blocks,  each of which is a chain recurrent one or a gradient one.
Then this decomposition implies a directed graph, called a Morse graph, which can capture the gradient behaviors.
The vertices, called Morse sets, correspond to recurrent parts, and the edges correspond to gradient parts.
Moreover, the Conley indices of the Morse sets are used to capture information about the local dynamics near the Morse sets.
Since the Conley-Morse graphs (i.e. Morse graph with Conley indices) represent topological behaviors of dynamical systems, applying graph algorithms, one can analyze the dynamics automatically if the Morse graph is finite.
The finite representability of Conley-Morse graphs is crucial for graph algorithms because only a finite date can be computed.
In fact, the Conley-Morse graphs are implemented as a computer software CHomP (Computational Homology Project software) \cite{AKKMOP2009}, and there are several relative works for analyzing dynamical systems using algorithms \cite{MSW2015,HMMN2014,MMW2016}.
The finite realizability of the orbit class spaces is studied \cite{BHS2011,BHSV2011}.
Therefore the question of how to reduce dynamical systems into finite topological invariants is essential from a theoretical and application point of view.
On the other hand, Hamiltonian flows, which are typical examples of recurrent dynamics, with finitely many singular points on compact surfaces can be classified by Reeb graphs of Hamiltonians with information about critical levels, because Reeb graphs are the dual graphs of the multi-saddle connection diagrams (i.e. the union of multi-saddles and orbits from or to multi-saddles) and the complement of the multi-saddle connection diagrams consist of periodic annuli.
Here a periodic annulus is an annulus consisting of periodic orbits.
In fact, Hamiltonian flows with finitely many singular points on closed surfaces are classified by a molecular, which is a finite labelled graph \cite{bolsinov2004integrable}.
The set of structurally stable Hamiltonian flows on compact surfaces is open dense in the set of Hamiltonian flows, and such Hamiltonian flows are characterized as Hamiltonian flows with nondegenerate singular points and with self-connected saddle connections and are classified by the multi-saddle connection diagrams,  which are finite invariants \cite{ma2005geometric}.
Similar results hold for the non-compact surfaces \cite{yokoyama2013word}.

It is known that finite topological spaces are in a one-to-one correspondence with finite pre-ordered sets via the specialization orders.
Here the specialization order $\leq$ is defined by $x \leq y$ if $x \in \overline{\{ y \}}$, where the closure of a subset $A$ is denoted by $\overline{A}$.
In particular, finite $T_0$ spaces are in a one-to-one correspondence with finite posets via the specialization orders.
McCord showed that there is a correspondence between finite topological spaces and finite simplicial complex up to weak homotopy equivalence \cite{McCord1966}.
Thus one can analyze topologies of dynamical systems using finite quotient spaces.
For instance, the relation between Hasse diagrams and orbit class spaces was studied \cite{BHSV2011}.
Note that a Morse graph is a quotient space of the orbit space with the directed structure.
In other words, the orbit space is a covering space of the underlying space of a Morse graph.
To analyze both gradient parts and recurrent parts of dynamical systems and describe global dynamics in detail, we define a directed structure on the orbit space, and construct topological invariants of flows on topological spaces which are refinements of Morse graphs of flows on metric spaces, CW decompositions of Morse flows on closed manifolds, and Reeb graphs of Hamiltonian flows on surfaces.
In fact, the Morse graphs of chaotic flows in the sense of Devaney, Hamiltonian flows on compact surfaces, pseudo-Anosov homeomorphisms, and non-identical non-pointwise periodic non-minimal volume-preserving homeomorphisms on compact connected manifolds are singletons.
However, the topological invariants of them are not singletons as shown in the final section.

On the other hand, we consider the following problems.

\begin{problem}\label{prob:01}
Find an equivalence relation $\sim_{w_1}$ generated by the time-one map $w_1$ of a flow $w$ and a large class $\mathcal{C}$ of flows such that, for any flow $v \in \mathcal{C}$ on a topological space $X$, the orbit space $X/v$ is homeomorphic to the quotient space $X/\sim_{v_1}$, where the orbit space $X/v$ is the quotient space collapsing orbits into singletons.
\end{problem}
\begin{problem}\label{prob:02}
Find a small equivalence relation $\approx_w$ generated by dynamical systems $w$ and a large class $\mathcal{C}$ of flows such that, for any flow $v \in \mathcal{C}$ on a topological space $X$,  the orbit space $X/\approx_v$ is homeomorphic to the quotient space $X/\approx_{v_1}$, where $v_1 \colon X \to X$ is the time-one map of $v$.
\end{problem}

\noindent
In general, the time-one map of a flow loses topological information of the original flow.
However, in some cases, the time-one map has the topological information of the original flow and so we obtain a reconstructing method.
In fact, the equivalence relation canonically induced by the abstract weak orbit space is one of the desired relations.
For instance, on Problem~\ref{prob:01}, we show that the orbit space of a Hamiltonian flows with finitely many singular points on a compact surface is homeomorphic to the abstract weak orbit space of the time-one map up to an arbitrarily small reparametrization (see Theorem~\ref{th:Ham_equiv}).
On Problem~\ref{prob:02}, we show that the abstract weak orbit spaces of a Morse flow on a compact manifold and the time-one map are homeomorphic (see Theorem~\ref{th:MS_equiv}).

The present paper consists of twelve sections.
In the next section, as preliminaries, we introduce the notions of combinatorics, topology, dynamical systems, and decomposition theory.
In particular, we introduce topological invariants of flows, called abstract weak orbit space and abstract orbit space.
In \S 3, fundamental properties of pre-orders and equivalence classes for flows on topological spaces are described.
In \S 4, fundamental properties for flows on Hausdorff spaces are described. In particular, we characterize recurrences and properness, and reductions of quotient spaces are described.
In \S 5, we show one of the main results that the Morse graph is a reduction of the abstract (weak) orbit space.
In \S 6, properties of binary relations for flows on compact Hausdorff spaces are described.
In \S 7, we show that CW decompositions of unstable manifolds of singular points of Morse flows on closed manifolds are quotient spaces of the abstract weak orbit spaces. 
Moreover, finiteness of abstract weak orbit spaces of Morse-Smale flows is described.  
In \S 8, we describe properties of gradient flows and Hamiltonian flows on compact surfaces.
In particular, the abstract weak orbit spaces of Hamiltonian surfaces flows with finitely many singular points can be reduced into the Reeb graphs.
In \S 9, we describe properties of flows of finite type on compact surfaces.
In \S 10, we reconstruct orbit spaces and abstract weak orbit spaces of flows from the time-one maps.
In \S 11, we show that the abstract weak orbit space is a complete invariant both of  Morse flows on compact surfaces and of flows generated by structurally stable Hamiltonian vector fields on compact surfaces.
In the final section, we state several examples.
In particular,  the Morse graphs in Examples~\ref{ex:rot}--\ref{ex:chaotic} are singletons but the abstract weak orbit spaces are not singletons.
Moreover, non-completeness of abstract weak orbit spaces is stated using pairs of examples.

\section{Preliminaries}

\subsection{Combinatorial and topological notions}

\subsubsection{Decomposition}
By a decomposition, we mean a family $\mathcal{F}$ of pairwise disjoint nonempty subsets of a set $X$ such that $X = \bigsqcup \mathcal{F}$, where $\bigsqcup$ denotes a disjoint union.
Note that the sets of orbits of flows are decompositions of connected elements.
Since connectivity is not required, the sets of orbits of homeomorphisms are also decompositions.

\subsubsection{Elements of topological notation}
For a subset $A$ of a topological space, denote by $\overline{A}$ the closure of $A$.
Then $\partial A := \overline{A} - \mathrm{int} A$ is the boundary of $A$, $\delta A = \partial_- A := A -  \mathrm{int} A$ is the border of $A$, and $\bp A := \overline{A} - A$ is the coborder of $A$.
Here $B - C$ is used instead of $B \setminus C$ when $C \subseteq B$.

\subsubsection{Stratification}
A finite filtration $\emptyset = M_{-1} \subseteq M_0 \subseteq \cdots \subseteq M_n = M$ of closed subsets of an $n$-dimensional manifold $M$ is a stratification if the difference $S_i := M_i - M_{i-1}$ is either the emptyset or an $i$-dimensional submanifold with $\overline{S_i} - S_i \subseteq \bigsqcup_{j <i} S_j$ for any $i \geq 0$.

\subsubsection{Local homeomorphism}
A continuous mapping $f \colon X \to Y$ is a local homeomorphism if for any point $x \in X$ there is an open \nbd $U$ of $x$ such that the image $f(U)$ is open and the restriction $f|_U  \colon  U \to f(U)$ is a homeomorphism.

\subsubsection{Separation axiom}
A point $x$ of a topological space $X$ is $T_0$ (or Kolmogorov) if for any point $y \neq x \in X$ there is an open subset $U$ of $X$ such that $|\{x, y \} \cap U| =1$, where $|A|$ is the cardinality of a subset $A$.
A point is $T_1$ if its singleton is closed.
Here a singleton is a set consisting of a point.
A topological space is $T_0$ (resp. $T_1$) if each point is $T_0$ (resp. $T_1$).
A topological space is KC if any compact subset is closed.
It is known that $KC$ implies $T_1$, and that Hausdorff separation axiom implies $KC$.


\subsubsection{Fr\'echet-Urysohn space}
A point $x \in X$ is a limit of a sequence $(x_n)_{n \in \Z_{\geq 0}}$ if for any \nbd $U$ of $x$, there is an integer $N$ such that  $\{ x_n \mid n > N \} \subseteq U$.
For a subset $A \subseteq X$, the sequential closure of $A$ is the set of limits of sequences of points in $A$.
A topological space $X$ is Fr\'echet-Urysohn if the closure of a subset of $X$ corresponds to its sequential closure.
It is known that a first-countable space is Fr\'echet-Urysohn.


\subsubsection{Graphs}
An ordered pair $G := (V, D)$ is an abstract directed graph (or a directed graph) if $V$ is a set and $D \subseteq V \times V$.
An ordered triple $G := (V, E, r)$ is an abstract multi-graph (or a multi-graph) if $V$ and $E$ are sets and $r : E \to \{ \{ x,y \} \mid x, y \in V \}$.

\subsubsection{Orders}
A binary relation $\leq$ on a set $P$ is a pre-order if it is reflexive (i.e. $a \leq a$ for any $a \in P$) and transitive (i.e. $a \leq c$ for any $a, b, c \in P$ with $a \leq b$ and $b \leq c$).
For a pre-order $\leq$, the inequality $a<b$ means both $a \leq b$ and $a \neq b$.
A pre-order $\leq$ on $X$ is a partial order if it is antisymmetric (i.e. $a = b$ for any $a,b \in P$ with $a \leq b$ and $b \leq a$).
A poset is a set with a partial order.
A pre-order order $\leq$ is a total order (or linear order) if either $a < b$ or $b < a$ for any  points $a \neq b$.
Note that a poset is totally ordered if and only if there is no pair of two incomparable points.
A chain is a totally ordered subset of a pre-ordered set with respect to the induced order.

\subsubsection{Heights of points}
Let $(X, \leq)$ be a pre-ordered set.
For a point $x \in X$, define the upset $\mathop{\uparrow}_{\leq} x = \mathop{\uparrow} x  := \{ y \in X \mid x \leq y \}$, the downset $\mathop{\downarrow}_{\leq} x = \mathop{\downarrow} x  := \{ y \in X \mid y \leq x \}$, and the class $\hat{x} := \mathop{\downarrow}  x \cap \mathop{\uparrow}  x$.
Denote by $\min X$ the set of minimal points.
For a subset $A \subseteq X$, define the upset $\mathop{\uparrow}_{\leq} A = \mathop{\uparrow} A := \bigcup_{x \in A} \mathop{\uparrow}  x$, the downset $\mathop{\downarrow}_{\leq} A = \mathop{\downarrow} A := \bigcup_{x \in A} \mathop{\downarrow}  x$, and the class $\hat{A} := \bigcup_{x \in A} \hat{x}$.
Define the height $\mathop{\mathrm{ht}} x$ of $x$ by
$\mathop{\mathrm{ht}}_{\leq} x = \mathop{\mathrm{ht}} x := \sup \{ |C| - 1 \mid C :\text{chain containing }x \text{ as the maximal point}\}$.
Define the height of the empty set is $-1$.
The height $\mathop{\mathrm{ht}} A$ of a nonempty subset $A \subseteq X$ is defined by
$\mathop{\mathrm{ht}}_{\leq} A = \mathop{\mathrm{ht}} A := \sup_{x \in A} \mathop{\mathrm{ht}} x$.

\subsubsection{Multi-graphs as posets}
A poset $P$ is said to be multi-graph-like if the height of $P$ is at most one and  $| \mathop{\downarrow} x | \leq 3$ for any element $x \in P$.
For a multi-graph-like poset $P$, each element of $P_0$ is called a vertex and each element of $P_1$ is called an edge.
Then an abstract multi-graph $G$ can be considered as a multi-graph-like poset $(P, \leq_G)$ with $V = P_{0}$ and $E = P_1$ as follows: $P = V \sqcup \mathrm{E}$ and $x <_G e$ if $x \in r(e)$, where $\sqcup$ denotes a disjoint union.
Conversely, a multi-graph-like poset $P$ can be considered as an abstract multi-graph with $V = P_{0}$, $E = P_1$, and $r: P_1 \to \{ \{ x,y \} \mid x, y \in V \}$ defined by $r(e) := \mathop{\downarrow}e - \{ e \}$.
Therefore we identify a multi-graph-like poset with an abstract multi-graph.


\subsubsection{Specialization order}
For a subset $A$ and a point $x$ of a topological space $(X, \tau)$, an abbreviated form of the singleton $\{ x \}$ (resp. the difference $A - \{ x\}$, the point closure $\overline{\{ x \}}$) will be $x$ (resp. $A - x$, $\overline{x}$).
The specialization order $\leq_\tau$ on a topological space $(X, \tau)$ is defined as follows: $ x \leq_\tau y $ if  $ x \in \overline{y}$.
Note that $\mathop{\downarrow} x  = \overline{x}$ and that the set $\min X$ of minimal points in $X$ is the set of points whose classes are closed.
Then $\hat{x} = \{ y \in X \mid \overline{x} = \overline{y} \}$ for any point $x \in X$.

\subsubsection{$T_0$-tification of a topological space}
Let $X$ a topological space with the specialization order.
Then the set $\hat{X} := \{ \hat{x} \mid x \in X \} = \{ \{ y \in X \mid \overline{x} = \overline{y} \} \mid x \in X \}$ of classes is a decomposition of $X$ and is a $T_0$ space as a quotient space, which is called the $T_0$-tification (or Kolmogorov quotient) of $X$.
Recall the following observation.

\begin{lemma}\label{lem:quotient_top}
Let $\pi \colon (X, \mathcal{O}_X) \to (\hat{X}, \mathcal{O}_{\hat{X}})$ be the quotient map.
Then induced mapping $\pi^* \colon \mathcal{O}_{\hat{X}} \to \mathcal{O}_X$ by $\pi^*(\mathcal{U}) := \pi^{-1}(\mathcal{U}) = \bigsqcup_{\hat{x} \in \mathcal{U}} \hat{x}$ is bijective.
Moreover, we have $\mathcal{O}_{\hat{X}} = \pi(\mathcal{O}_X) = \{ \{ \hat{x} \mid x \in U \} \mid U \in \mathcal{O}_X \} = \{ \mathcal{U} \subseteq \hat{X} \mid \bigsqcup_{\hat{x} \in \mathcal{U}} \hat{x} \in \mathcal{O}_X \}$.
\end{lemma}

\begin{proof}
By definition of quotient topology, the image $\pi^*(\mathcal{U}) = \pi^{-1}(\mathcal{U}) = \bigsqcup_{\hat{x} \in \mathcal{U}} \hat{x}$ for any element $\mathcal{U} \in \mathcal{O}_{\hat{X}}$ is an open subset of $X$ (i.e. $\pi^{-1}(\mathcal{U}) \in \mathcal{O}_X$).
Thus $\pi^*$ is well-defined.
Since $\pi^*(\mathcal{U}) = \bigsqcup_{\hat{x} \in \mathcal{U}} \hat{x}$ for any element $\mathcal{U} \in \mathcal{O}_{\hat{X}}$, the induced mapping $\pi^*$ is injective.
%
Fix an open subset $U \in \mathcal{O}_X$.
To show the surjectivity of $\pi^*$, it suffices to show that $U = \pi^{-1}(\pi(U))$.
Indeed, fix a point $y \notin U$.
Since $y \notin U$, we have $\overline{y} \cap U = \emptyset$ and so $\overline{y}   \neq \overline{z}$ for any $z \in U$.
This means that $\hat{y} \cap \hat{U} = \emptyset$ and so $U = \hat{U}$.
By definition of $\hat{X}$, we have $\pi(U) = \{ \hat{x} \mid x \in U \}$ and so $\pi^{-1}(\pi(U)) = \bigsqcup_{\hat{x} \in \pi(U)} \hat{x} = \bigsqcup_{x \in U} \hat{x} = \hat{U} = U$.
\end{proof}

\subsubsection{Alexandroff spaces}

A topological space $(X, \tau)$ is Alexandroff if $\tau$ is the set of upsets with respect to the specialization order $\leq_\tau$ (i.e $\tau = \{ \mathop{\uparrow} A \mid A \subseteq X \}$).
In other words, a topological space $(X, \tau)$ is Alexandroff if and only if the intersection of any family of open subsets is open.
By definition of Alexandroff spaces, notice that there is a one-to-one correspondence between Alexandroff spaces and pre-ordered sets, and that  there is a one-to-one correspondence between $T_0$ Alexandroff spaces and posets.

\subsubsection{Poset-stratification}
A continuous mapping $X \to P$ from a topological space $X$ to a topological space $P$ is a poset-stratification of $X$ if $P$ is $T_0$ and Alexandroff.

\subsubsection{Reeb graph of a function on a topological space}
For a function $f \colon  X \to \R$ on a topological space $X$, the Reeb graph of a function $f \colon  X \to \R$ on a topological space $X$ is a quotient space $X/\sim_{\mathrm{Reeb}}$ defined by $x \sim_{\mathrm{Reeb}} y$ if there are a number $c \in \R$ and a connected component of $f^{-1}(c)$ which contains $x$ and $y$.

\subsection{Fundamental notion of dynamical systems}

A mapping $v:\mathbb{R} \times X \to X$ on a topological space $X$ is an $\R$-action on $X$ if the restriction $v(t,\cdot) \colon X \to X$ by $v(t,\cdot)(x) := v(t,x)$ for any $t \in \R$ is homeomorphic such that the restriction $v(0,\cdot)$ is an identity mapping on $X$ and $v(t,v(s,x)) = v(t+s,x)$ for any $t,s \in \R$.
%
A flow is a continuous $\R$-action on a topological space.
For a point $x$ of $X$, we denote by $O(x)$ (or $O_v(x)$) the orbit of $x$ (i.e. $O(x) := \{ v_t(x) \mid t \in \R \}$), $O^+(x)$ the positive orbit (i.e. $O^+(x) := \{ v_t(x) \mid t > 0 \}$), and $O^-(x)$ the negative orbit (i.e. $O^-(x) := \{ v_t(x) \mid t < 0 \}$).
A flow $w$ is a reparametrization of a flow $v$ if $O_v(x) = O_w(x)$ for any $x \in X$.
Let $v  \colon  \R \times X \to X$ be a flow on a topological space $X$.
For $t \in \R$, define $v_t  \colon  X \to X$ by $v_t := v(t, \cdot )$.
The homeomorphism $v_1 \colon X \to X$ by $v_1(x) = v(1, x)$ is called the time-one map of $v$.
A subset of $X$ is invariant (or saturated) if it is a union of orbits.
For a subset $A \subseteq X$, the saturation $v(A)$ of $A$ is the union $\bigcup_{x \in A} O(x)$.
Recall that a point $x$ of $X$ is singular if $x = v_t(x)$ for any $t \in \R$ and is periodic if there is a positive number $T > 0$ such that $x = v_T(x)$ and $x \neq v_t(x)$ for any $t \in (0, T)$.
An orbit is singular (resp. periodic) if it contains a singular (resp. periodic) point.
An orbit is closed if it is singular or periodic.
Note that any closed orbit of a flow on a Hausdorff space is closed (see Lemma~\ref{lem:kc_closed}), and that an orbit of a flow on a compact Hausdorff space is closed if and only if it is a closed subset (see Lemma~\ref{lem:closed}).
Denote by $\mathop{\mathrm{Sing}}(v)$ the set of singular points and by $\mathop{\mathrm{Per}}(v)$ (resp. $\mathop{\mathrm{Cl}}(v)$) the union of periodic (resp. closed) orbits.
A point is wandering if there are its neighborhood $U$
and a positive number $N$ such that $v_t(U) \cap U = \emptyset$ for any $t > N$.
Then such a neighborhood $U$ is called a wandering domain.
A point is non-wandering if it is not wandering (i.e. for any its neighborhood $U$ and for any positive number $N$, there is a number $t \in \mathbb{R}$ with $|t| > N$ such that $v_t(U) \cap U \neq \emptyset$).
Denote by $\Omega (v)$ the set of non-wandering points, called the non-wandering set.

\subsubsection{Chain recurrence}
Let $w \colon  \R \times M \to M$ be a flow on a metric space $(M,d)$.
For any $\varepsilon > 0$ and $T>0$, a pair $\{ (x_i)_{i=0}^{k+1}, (t_i)_{i=0}^{k} \}$ is an $(\varepsilon, T)$-chain from a point $x \in M$ to a point $y \in M$ if $x_0 = x$, $x_{p+1} = y$, $t_i > T$ and $d(w_{t_i}(x_i),x_{i+1}) < \varepsilon$ for any $i = 0, \ldots , k$.
Define a binary relation $\sim_{\mathop{CR}}$ on $M$ by $x \sim_{\mathop{CR}} y$ if for any $\varepsilon > 0$ and $T>0$ there is an $(\varepsilon, T)$-chain from $x$ to $y$.
A point $x \in M$ is chain recurrent \cite{conley1978isolated} if $x \sim_{\mathop{CR}} x$.
Denote by $\mathop{CR} (w)$ the set of chain recurrent points, called the chain recurrent set.
It is known that the chain recurrent set $\mathop{CR} (w)$ is closed and invariant and contains the non-wandering set $\Omega (w)$ \cite[Theorem~3.3B]{Conley1988}, and that connected components of $\mathop{CR} (w)$ are equivalence classes of the relation $\approx_{\mathop{CR}}$ on $\mathop{CR} (w)$ \cite[Theorem~3.3C]{Conley1988}, where $x \approx_{\mathop{CR}} y$ if $x \sim_{\mathop{CR}} y$ and $y \sim_{\mathop{CR}} x$.
\\

From now on, we assume that $v$ is a flow on a topological space $X$ unless otherwise stated.

\subsubsection{$\alpha$-limit sets, $\omega$-limit sets, and recurrent orbits}

Recall that the $\omega$-limit set of a point $x \in X$ is $\omega(x) := \bigcap_{n\in \mathbb{R}}\overline{\{v_t(x) \mid t > n\}}$, and that the $\alpha$-limit set of $x$ is $\alpha(x) := \bigcap_{n\in \mathbb{R}}\overline{\{v_t(x) \mid t < n\}}$.
By definitions,  the $\alpha$-limit set and the $\omega$-limit set of $x$ are closed and invariant.
For an orbit $O$, define $\omega(O) := \omega(x)$ and
$\alpha(O) := \alpha(x)$ for some point $x \in O$.
Note that an $\omega$-limit (resp. $\alpha$-limit) set of an orbit is independent of the choice of a point in the orbit.
Moreover, $\alpha(O) = \bigcup_{x \in O} \alpha(x)$ and $\omega(O) = \bigcup_{x \in O} \omega(x)$ for any orbit $O$.
Therefore we define $\alpha(A) := \bigcup_{x \in A} \alpha(x)$ and $\omega(A) := \bigcup_{x \in A} \omega(x)$ for a subset $A$ of $X$.
A separatrix is a non-singular orbit whose $\alpha$-limit or $\omega$-limit set is a singular point.
A point $x$ is recurrent if $x \in \alpha(x) \cup \omega(x)$.
An orbit is recurrent if it contains a recurrent point.
Denote by $\mathcal{R}(v)$ the set of recurrent points.
By definition, we have that $\mathcal{R}(v) \subseteq \Omega(v)$ and that $ \Omega(v) \subseteq \mathop{CR}(v)$ if the base space $X$ is a compact metric space.
Moreover, denote by $\mathrm{R}(v)$ (resp. $\mathrm{P}(v)$) the union of non-closed recurrent orbits (resp. non-recurrent orbits).
By definition, for a flow $v$ on a topological space, we have a decomposition $X = \mathcal{R}(v) \sqcup \mathrm{P}(v) = \mathop{\mathrm{Sing}}(v) \sqcup \mathop{\mathrm{Per}}(v) \sqcup \mathrm{P}(v) \sqcup \mathrm{R}(v)$.

\subsubsection{$\alpha'$-limit sets and $\omega'$-limit sets}
Define $\alpha'(x)$ (resp. $\omega'(x))$ for a point $x \in X$ as follows \cite{markus1954global,buendia2018markus}:
$$\alpha'(x) := \alpha(x) \setminus O(x)$$
$$\omega'(x) := \omega(x) \setminus O(x)$$
Similarly, define $\alpha'(O) := \alpha(O) \setminus O$ and $\omega'(O) := \omega(O) \setminus O$ for an orbit $O$.
We call $\alpha'(x)$ (resp. $\alpha'(O)$) the $\alpha'$-limit set and $\omega'(x)$ (resp. $\omega'(O)$) the $\omega'$-limit set.
The definition of $\alpha'(x)$ (resp. $\omega'(x)$) by Buend{\'\i}a and L\'opes is slightly different from the original one, cf.~\cite{markus1954global,neumann1975classification,neumann1976global} but these definitions correspond for proper orbits.

\subsubsection{Topological properties of orbits}
An orbit is proper if it is embedded.
In other words, an orbit is proper if and only if it is homeomorphic to either a singleton, the unit circle $\{ z \in \C \mid |z| = 1 \}$, or the real line $\R$.
A point is proper if so is its orbit.
Note that an orbit of a flow on a Hausdorff space is proper if and only if it either is non-recurrent or is a closed subset (see Lemma~\ref{cor:proper}), and that an orbit of a flow on a Hausdorff space is recurrent if and only if it either is either non-proper or closed (see Lemma~\ref{cor:recurrent}).
%
Note that $\mathrm{P}(v)$ for a flow on a Hausdorff space is the union of non-closed proper orbits.

\subsection{Quotient spaces and binary relations for flows}

\subsubsection{Morse graph}
Recall that a directed graph is a pair of a set $V$ and a subset $D \subseteq V \times V$.
For a flow $w$ on a compact metric space $M$ with a set $\mathcal{M} = \{ M_i \}_{i \in \Lambda}$ of pairwise disjoint compact invariant subsets, a directed graph $(V, D)$ with the vertex set $V := \{ M_i \mid i \in \Lambda \}$ and with the directed edge set $D := \{ (M_j, M_k) \mid D_{j,k} \neq \emptyset  \}$ is a Morse graph of $\mathcal{M}$ if $M - \bigsqcup_i M_i = \bigsqcup D_{j,k}$, where $D_{j,k} := \{ x \in M \mid \alpha(x) \subseteq M_j, \omega(x) \subseteq M_k \} = W^u(M_j) \cap W^s(M_k)$ for any distinct indices $j \neq k$.
Then such a graph is denoted by $G_{\mathcal{M}}$, and $D_{j,k}$ is called a connecting orbit set from $M_j$ to $M_k$.
We also call that the vertex $M_k$ is a connecting orbit set from $M_k$ to $M_k$. 
%
If $\mathcal{M}$ is the set of connected components of the chain recurrent set $\mathop{CR}(w)$, then the graph $G_{\mathcal{M}}$ is called the Morse graph of the flow $w$ and denoted by $G_w$, and vertices are called Morse sets.
Note that
a Morse graph $\mathcal{M}_{w}$ of a flow $w$ is a quotient space of the orbit space $M/w$ with a directed structure (see Theorem~\ref{th:Morse_reduction} and Corollary~\ref{cor:Morse_reduction_weak} for details).
To describe dynamics in detail, we will introduce an intermediate quotient spaces, called abstract weak orbit space and weak orbit class space, of the orbit space which are refinements of the Morse graph.
In other words, the Morse graph is a quotient space of such quotient spaces.

\subsubsection{Orbit classes}
For any point $x \in X$, define the orbit class $\hat{O}(x)$ as follows:
$$\hat{v}(x) = \hat{O}(x) := \{ y \in X \mid \overline{O(x)} = \overline{O(y)} \} \subseteq \overline{O(x)}$$
We call $\hat{O}(x)$ the orbit class of $x$.
Note that $\hat{O}(x) = \hat{O}(y)$ for any point $y \in \hat{O}(x)$.
It's known that the following conditions are equivalent for an orbit $O$ on a paracompact manifold: $(1)$ The orbit $O$ is proper; $(2)$ $O = \hat{O}$ \cite[Corollary 3.4]{yokoyama2019properness}.

\subsubsection{Orbit spaces and orbit class spaces}

For a flow $v$ on a topological space $X$, the orbit space $X/v$ of $X$ is a quotient space $X/\sim_v$ defined by $x \sim_v y$ if $O(x) = O(y)$.
Similarly, the orbit class space $X/\hat{v}$ of $X$ is a quotient space $X/\sim_{\hat{v}}$ defined by $x \sim_{\hat{v}} y$ if $\overline{O(x)} = \overline{O(y)}$.
Note that the orbit class space $X/{\hat{v}}$ is the $T_0$-tification of the orbit space $X/v$.
The orbit class space is also called the quasi-orbit space in \cite{BHS2011}.
For a saturated subset $T$ of $X$, the subset $\pi_v(T)$ is denoted by $T/v$, where $\pi_v \colon X \to X/v$ is the quotient map.
Notice that an orbit space $T/v$ is the set of orbits contained in $T$ as a set.

\subsubsection{Weak orbit classes and weak orbit class spaces}

For any point $x \in X$, define the saturated subset $\check{O}(x)$ as follows:
$$\check{v}(x) = \check{O}(x) := \{ y \in X \mid \overline{O(x)} = \overline{O(y)}, \alpha(x) = \alpha(y), \omega(x) = \omega(y) \} \subseteq \hat{O}(x)$$
We call that $\check{O}(x)$ is the weak orbit class of $x$.
By definition, we have the following observation.

\begin{lemma}
The following properties hold for a flow $v$ on a topological space $X$ and for any points $x,y \in X$:
\\
$(1)$ $O(x) \subseteq \check{O}(x) \subseteq \hat{O}(x) \subseteq \overline{O(x)}$.
\\
$(2)$ $\overline{O(x)} = \overline{\hat{O}(x)} = \overline{\check{O}(x)}$.
\\
$(3)$  If $y \in \check{O}(x)$, then $\check{O}(x) = \check{O}(y)$.
\end{lemma}

Define the weak orbit class space $X/\check{v}$ as a quotient space $X/\sim_{\check{v}}$ defined by $x \sim_{\check{v}} y$ if $\check{O}(x) = \check{O}(y)$.
By definition, the weak orbit class space $X/\check{v}$ is a quotient space of the orbit space $X/v$ and the orbit class space $X/\hat{v}$ is a quotient space of the weak orbit class space $X/\check{v}$.

\subsubsection{The specialization orders of flows}

The specialization order $\leq_v$ on $X/v$ of the flow $v$ is defined as follows:
$$O_1 \leq_v O_2 \text{ if } O_1 \subseteq \overline{O_2} \text{ as subsets of } X$$
By abuse of terminology, define the specialization order $\leq_v$ on $X$ of the flow $v$ defined as follows:
$$x \leq_v y \text{ if } O(x) \subseteq \overline{O(y)}$$
Moreover, the specialization (partial) order $\leq_v$ on $X/\hat{v}$ is induced as follows:
$$\hat{O}_1 \leq_v \hat{O}_2 \text{ if } \hat{O}_1 \subseteq \overline{\hat{O}_2} \text{ as subsets of } X$$
For any orbit $O$, since $O \subseteq \hat{O} \subseteq \overline{O}$, we have that $\overline{\hat{O}} = \overline{O}$, and so that $\hat{O}_1 \leq_v \hat{O}_2$ if and only if $O_1 \subseteq \overline{O_2}$.
Then we have the following observation.

\begin{lemma}
The following conditions are equivalent: \\
$(1)$ $\overline{O(x)} \subseteq \overline{O(y)}$
\\
$(2)$ $x \leq_v y$
\\
$(3)$ $O(x) \leq_v O(y)$
\\
$(4)$ $\hat{O}(x) \leq_v \hat{O}(y)$
\end{lemma}

Let $\tau_v$ be the quotient topology on the orbit space $X/v$ and $\tau_{\hat{v}}$ the quotient topology on the orbit space $X/\hat{v}$.
Then the binary relation $\leq_v$ on the orbit space $X/v$ (resp. orbit class space $X/\hat{v}$) corresponds to the specialization order $\leq_{\tau_v}$ (resp. partial order $\leq_{\tau_{\hat{v}}}$).

\subsubsection{Height of an orbit space}

For a point $x \in X$, define the height of $x$ by the height of the element $O(x)$ in the orbit space $X/v$ and write $\mathop{\mathrm{ht}}(x) := \mathop{\mathrm{ht}}_{\tau_v}(O(x))$, where $\mathop{\mathrm{ht}}_{\tau_v}$ is the height with respect to the specialization order $\leq_{\tau_v}$ on the orbit space $X/v$ induced by $v$.
Since the orbit class space $X/{\hat{v}}$ is the $T_0$-tification of the orbit space $X/v$, we have $\mathop{\mathrm{ht}}(x) = \mathop{\mathrm{ht}}_{\tau_v}(O(x)) = \mathop{\mathrm{ht}}_{\tau_{\hat{v}}}(\hat{O}(x))$.
Moreover, we have the following observation.
\begin{lemma}
The following conditions are equivalent for a point $x \in X$: \\
$(1)$ $\mathop{\mathrm{ht}}(x) \geq k$.
\\
$(2)$ There is a sequence $\overline{O_0} \subsetneq \overline{O_1} \subsetneq  \cdots \subsetneq \overline{O_{k-1}} \subsetneq \overline{O_k} = \overline{O(x)}$ as subsets of $X$.
\\
$(3)$ There is a sequence $\overline{\hat{O}_0} \subsetneq \overline{\hat{O}_1} \subsetneq  \cdots \subsetneq \overline{\hat{O}_{k-1}} \subsetneq \overline{\hat{O}_k} = \overline{\hat{O}(x)}$ as subsets of $X$.
\end{lemma}

The height $\mathop{\mathrm{ht}} A$ of a subset $A \subseteq X$ is defined by $\mathop{\mathrm{ht}} A := \sup_{x \in A} \mathop{\mathrm{ht}}(x)$.
The height of the flow $v$ is defined by $\mathop{\mathrm{ht}}(v) := \mathop{\mathrm{ht}}(X)$.


\subsubsection{Order structures of orbit spaces of flows}

Define pre-orders $\leq_{\alpha}$ and $\leq_{\omega}$
on $X$ as follows:
\\
$x \leq_{\alpha} y$ if  either $O(x) = O(y)$ or $x \in \alpha (y)$ (i.e. $x \in \alpha (y) \cup O(y)$)
\\
$x \leq_{\omega} y$ if either $O(x) = O(y)$ or $x \in \omega (y)$ (i.e. $x \in \omega (y) \cup O(y)$)
\\
\\
These orders induce the following pre-orders on the orbit space $X/v$ of the flow $v$, by abuse of terminology, which are denoted by the same symbols:
\\
$O_1 \leq_{\alpha} O_2$ if either $O_1 = O_2$ or $O_1 \subseteq \alpha (O_2)$ (i.e. $O_1 \subseteq \alpha (O_2) \cup O_2$)
\\
$O_1 \leq_{\omega} O_2$ if either $O_1 = O_2$ or $O_1 \subseteq \omega (O_2)$ (i.e. $O_1 \subseteq \omega (O_2) \cup O_2$)
\\

\subsubsection{Equivalence classes of orbits}

For a point $x \in X$, the connected component of the subset $\{ y \in X \mid \alpha' (x) = \alpha' (y), \omega'(x) = \omega' (y) \}$ containing $x$ is denoted by $[x]'$.
Note that $O(x) \subseteq [x]'$ and that $[x]' = [y]'$ for any $y \in [x]'$.
%
Define a saturated subset $[x]'' \subseteq [x]'$ by the connected component of $[x]' \cap \Sv$ (resp. $[x]' \cap \Pv$, $[x]' \cap \mathrm{P}(v)$, $[x]' \cap  \mathrm{R}(v)$) containing $x$ if $x \in \Sv$ (resp. $x \in \Pv$, $x \in \mathrm{P}(v)$, $x \in \mathrm{R}(v)$).
Note that $O(x) \subseteq [x]''$ and that $[x]'' = [y]''$ for any $y \in [x]''$.
Define $X/[v]''$ by a quotient space $X/\sim_{[v]''}$ defined by $x \sim_{[v]''} y$ if $[x]'' = [y]''$.

\subsubsection{Abstract weak orbits and abstract orbits}
Define a saturated subset $[x]$ of $X$ as follows:
\[
  [x] := \begin{cases}
    [x]'' & \text{if } x \in \mathop{\mathrm{Cl}}(v) \sqcup \mathrm{P}(v) = X - \mathrm{R}(v)  \\
    \check{O}(x) & \text{if } x \in \mathrm{R}(v)
  \end{cases}
\]
We call that $[x]$ is the abstract weak orbit of $x$.
Note that $[x] = [y]$ for any point $y \in [x]$.
In the case that $X$ is Hausdorff, we have a simple explicit characterization of abstract weak orbit as in Corollary~\ref{cor:ch}.
In general, the closure of an abstract weak orbit is saturated but is not a union of abstract weak orbits.
The abstract weak orbit of a subset $A$ is defined by $[A]:= \bigcup_{a \in A} [a]$.
We have the following observation.

\begin{lemma}\label{lem:same_limit_set}
For any point $x \in X - \mathop{\mathrm{Cl}}(v)$ and any point $y \in [x]$, we have $\alpha(x) = \alpha(y)$ and $\omega(x) = \omega(y)$.
\end{lemma}

\begin{proof}
Fix a point $x \in X - \mathop{\mathrm{Cl}}(v)$ and a point $y \in [x]$.
Suppose that $x \in \mathrm{R}(v)$.
Then $[x] = \check{O}(x)$.
Definition of $\check{O}(x)$ implies the assertion.
Thus we may assume that $x \in \mathrm{P}(v) = X -(\mathop{\mathrm{Cl}}(v) \sqcup \mathrm{R}(v))$.
Then $[x] = [x]'' \subseteq \mathrm{P}(v)$ and so $y \in \mathrm{P}(v)$.
The non-recurrence of $x$ (resp. $y$) implies that $O(x) \cap (\alpha(x) \cup \omega(x)) = \emptyset$ (resp. $O(y) \cap (\alpha(y) \cup \omega(y)) = \emptyset$).
This means that $\alpha(x) = \alpha'(x)$, $\omega(x) = \omega'(x)$, $\alpha(y) = \alpha'(y)$, and $\omega(y) = \omega'(y)$.
Then $\alpha(x) = \alpha'(x) = \alpha'(y) = \alpha(y)$ and $\omega(x) = \omega'(x) = \omega'(y) = \omega(y)$.
\end{proof}

Define a saturated subset $\langle x \rangle$ of $X$ as follows:
\[
  \langle x \rangle := \begin{cases}
    [x]'' & \text{if } x \in \mathop{\mathrm{Cl}}(v) \sqcup \mathrm{P}(v) = X - \mathrm{R}(v)  \\
    \hat{O}(x) & \text{if } x \in \mathrm{R}(v)
  \end{cases}
\]
We call that $\langle x \rangle$ is the abstract orbit of $x$.
Note that $\langle x \rangle = \langle y \rangle$ for any point $y \in \langle x \rangle$.
In the caee that $X$ is Hausdorff, we obtain a simple explicit characterization of abstract orbit as in Corollary~\ref{cor:ch01}.
In general, the closure of an abstract orbit is saturated but is not a union of abstract orbits.
By definition, we have the following observation.

\begin{lemma}
The following properties hold for a flow $v$ on a topological space $X$ and for any point $x \in X$:
\\
$(1)$ $O(x) \subseteq [x] \subseteq \langle x \rangle$.
\\
$(2)$ $\overline{[x]} = \overline{\langle x \rangle}$.
\end{lemma}

\begin{proof}
Fix any points $x \in X$.
Since $\check{O}(x) \subseteq \hat{O}(x)$, we have $O(x) \subseteq [x] \subseteq \langle x \rangle$.
If $x \in \mathop{\mathrm{Cl}}(v) \sqcup \mathrm{P}(v) = X - \mathrm{R}(v)$, then $\langle x \rangle = [x]''= [x]$ and so $\overline{[x]} = \overline{\langle x \rangle}$.
Thus we may assume that $x \in \mathrm{R}(v)$.
Then $\overline{[x]} = \overline{\check{O}(x)} = \overline{O(x)} = \overline{\hat{O}(x)} = \overline{\langle x \rangle}$.
\end{proof}

The abstract orbit of  a subset $A$ is defined by $\langle A \rangle:= \bigcup_{a\in A} \langle a \rangle$.

\subsubsection{Abstract weak orbit space and abstract orbit space of a flow}

Define the abstract weak orbit space $X/[v]$ as a quotient space $X/\sim_{[v]}$ defined by $x \sim_{[v]} y$ if $[x] = [y]$.
Similarly, define the abstract orbit space $X/\langle v \rangle$ as a quotient space $X/\sim_{\langle v \rangle}$ defined by $x \sim_{\langle v \rangle} y$ if $\langle x \rangle = \langle y \rangle$.
Since $O(x) \subseteq [x] \subseteq \langle x \rangle$ for any $x \in X$, the abstract weak orbit space is a quotient space of the orbit space, and the abstract orbit space is a quotient space of the abstract weak orbit space.
We have the following observation.

\begin{lemma}\label{lem:w_s}
Let $v$ be a flow on a topological space $X$ with $\mathrm{R}(v) = \emptyset$.
Then the abstract orbit and the abstract weak orbit of a point coincide with each other.
Moreover, we obtain $X/[v] = X/\langle v \rangle$.
\end{lemma}

\subsubsection{Transverse pre-orders for abstract {\rm(}weak {\rm)} orbits}

The binary relations $\leq_v$, $\leq_{\alpha}$, and $\leq_{\omega}$ on $X/v$ of a  flow $v$ induces following binary relations on the abstract weak orbit space $X/[v]$ and on the abstract orbit space $X/\langle v \rangle$, by abuse of terminology, which are denoted by the same symbols:
$$[x] \leq_v [y] \text{ if either } [x] = [y] \text{ or } \overline{O(x_1)} \subseteq \overline{O(y_1)} \text{ for some } x_1 \in [x] \text{ and } y_1 \in [y]$$
$$[x] \leq_\alpha [y] \text{ if either } [x] = [y] \text{ or } x_1 \in \alpha (y_1) \text{ for some } x_1 \in [x] \text{ and } y_1 \in [y]$$
$$[x] \leq_\omega [y] \text{ if either } [x] = [y] \text{ or } x_1 \in \omega (y_1) \text{ for some } x_1 \in [x] \text{ and } y_1 \in [y]$$
$$\langle x \rangle \leq_v \langle y \rangle \text{ if either } \langle x \rangle = \langle y \rangle \text{ or } \overline{O(x_1)} \subseteq \overline{O(y_1)} \text{ for some } x_1 \in \langle x \rangle \text{ and } y_1 \in \langle y \rangle$$
$$\langle x \rangle \leq_\alpha \langle y \rangle \text{ if either } \langle x \rangle = \langle y \rangle \text{ or } x_1 \in \alpha (y_1) \text{ for some } x_1 \in \langle x \rangle \text{ and } y_1 \in \langle y \rangle$$
$$\langle x \rangle \leq_\omega \langle y \rangle \text{ if either } \langle x \rangle = \langle y \rangle \text{ or } x_1 \in \omega (y_1) \text{ for some } x_1 \in \langle x \rangle \text{ and } y_1 \in \langle y \rangle$$

By definition, we have the following observation.

\begin{lemma}
The following properties hold for points $x,y \in X$: \\
$(1)$ $[x] \leq_\alpha [y]$ if and only if  $[x] = [y]$ or $[x] \cap \alpha([y]) \neq \emptyset$.
\\
$(2)$ $[x] \leq_\omega [y]$ if and only if  $[x] = [y]$ or $[x] \cap \omega([y]) \neq \emptyset$.
\\
$(3)$ $\langle x \rangle \leq_\alpha \langle y \rangle$ if and only if $\langle x \rangle = \langle y \rangle$ or $\langle x \rangle \cap \alpha(\langle y \rangle) \neq \emptyset$.
\\
$(4)$ $\langle x \rangle \leq_\omega \langle y \rangle$ if and only if $\langle x \rangle = \langle y \rangle$ or $\langle x \rangle \cap \omega(\langle y \rangle) \neq \emptyset$.
\end{lemma}

We will show that these binary relations $\leq_\alpha$ and $\leq_\omega$ on the abstract weak orbit space $X/[v]$ and on the abstract orbit space $X/\langle v \rangle$ are pre-orders (see Lemma~\ref{lem:preorders} for details).
If $X$ is Hausdorff, then the binary relation $\leq_v$ is also a pre-order on such spaces (see Lemma~\ref{lem:preorder} and Lemma~\ref{lem:preorders02} for details).

\subsubsection{Horizontal boundaries and transverse boundaries of a saturated subset}

For a saturated subset $A$, define the horizontal boundary $\partial_\perp A$ and the transverse boundary $\partial_\pitchfork$ as follows:
$$\partial_\perp A := \left( \bigcup_{z \in A} \alpha(z) \cup \omega(z) \right) \setminus A = (\alpha(A) \cup \omega(A)) \setminus A$$
$$\partial_\pitchfork A := (\overline{A} - A) \setminus \left( \bigcup_{z \in A} (\alpha(z) \cup \omega(z)) \right) = (\overline{A} - A) \setminus (\alpha(A) \cup \omega(A))$$
Note that $\partial_\perp A \sqcup \partial_\pitchfork A = \overline{A} - A = \bp A$ and so that $\partial A = \delta A \sqcup \partial_\perp A \sqcup \partial_\pitchfork A$, where $\delta A = A - \mathrm{int}A$ is the border of $A$.

\subsubsection{Specialization orders and transverse binary relations}

We define binary relations $\leq_\partial$, $\leq_\perp$, and $\leq_\pitchfork$ on $X$ as follows:
$$x \leq_\partial y \text{ if } x \in \overline{[y]}$$
$$x \leq_\perp y \text{ if either } [x] = [y] \text{ or } x \in \bigcup_{z \in [y]} (\alpha(z) \cup \omega(z)) = \alpha([y]) \cup \omega([y])$$
$$x \leq_\pitchfork y \text{ if either } [x] = [y] \text{ or } x \in \overline{[y]} - \left( \bigcup_{z \in [y]} (\alpha(z) \cup \omega(z)) \right) =\overline{[y]} - (\alpha([y]) \cup \omega([y]))$$

By definition, the binary relation $\leq_\partial $ is the union of binary relations $\leq_\perp$ and $\leq_\pitchfork$ as a direct product, and it contains the pre-order $\leq_v$ as a direct product (i.e. for any points $x$ and $y$ in $X$, the relation $x \leq_v y$ implies that $x \leq_\partial y$).
Since $\overline{[y]} = \overline{\langle y \rangle}$, we have that $x \leq_\partial y$ if and only if $x \in \overline{\langle y \rangle}$.
We have the following observations.

\begin{lemma}
The following properties hold for points $x,y \in X$: \\
$(1)$ $x \leq_\partial y$ if and only if $[x] = [y]$ or $x \in \bp [y]$.
\\
$(2)$ $x \leq_\perp y$ if and only if $[x] = [y]$ or $x \in \partial_\perp [y]$.
\\
$(3)$ $x \leq_\pitchfork y$ if and only if $[x] = [y]$ or $x \in \partial_\pitchfork [y]$.
\end{lemma}

\begin{lemma}
The following properties hold for points $x,y \in X$: \\
$(1)$ $x <_\partial y$ if and only if $x \in \bp [y]$.
\\
$(2)$ $x <_\perp y$ if and only if $x \in \partial_\perp [y]$.
\\
$(3)$ $x <_\pitchfork y$ if and only if $x \in \partial_\pitchfork [y]$.
\end{lemma}
For any points $x$ and $y$ in a Hausdorff space $X$, since $\overline{O(x)} = O(x) \cup \alpha(x) \cup \omega(x)$ (see Lemma~\ref{lem:decomp_limit}), we have that $x \leq_\pitchfork y$ if and only if either $[x] = [y]$ or $x \in \overline{[y]} - \left( \bigcup_{z \in [y]} \overline{O(z)} \right) = \overline{[y]} - \mathop{\downarrow}_{\leq_v} [y]$.
Note that the binary relations $\leq_\partial $ and $\leq_\pitchfork$ are neither antisymmetric nor transitive in general.
In fact, there is a flow whose binary relations $\leq_\partial $ and $\leq_\pitchfork$ are not antisymmetric (see Figure~\ref{non_symmetric}) and there is a flow whose binary relations $\leq_\partial $ and $\leq_\pitchfork$ are not transitive (see Figure~\ref{non_transitive}).
\begin{figure}
\begin{center}
\includegraphics[scale=0.33]{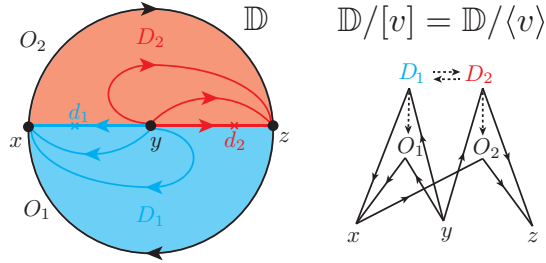}
\end{center}
\caption{A flow $v$ on a closed disk $D$ whose binary relations $\leq_\partial $ and $\leq_\pitchfork$ are not antisymmetric such that $D$ consists of seven abstract (weak) orbits $x, y, z, O_1, O_2, D_1, D_2$ with $D = \{x, y, z \} \sqcup O_1 \sqcup O_2 \sqcup D_1 \sqcup D_2$, and that the quotients $D_1/v$ and $D_2/v$ are half-open intervals, and there are points $d_1 \in D_1 \cap (\overline{D_2} - D_2)$ and $d_2 \in D_2 \cap (\overline{D_1} - D_1)$.
Then $d_1 \leq_\partial d_2 \leq_\partial d_1$ and $d_1 \leq_\pitchfork d_2 \leq_\pitchfork d_1$ but $d_1 \neq d_2$.
In the diagram on the right, solid lines form the Hessian diagram of the order $<_v$ whose down (resp. up) arrows correspond to $<_{\omega}$ (resp. $<_{\alpha}$), and dotted directed lines correspond to $<_\pitchfork$. Moreover, the binary relation $<_\partial$ is the union of $<_v$ and $<_\pitchfork$ in this case.}
\label{non_symmetric}
\end{figure}
\begin{figure}
\begin{center}
\includegraphics[scale=0.33]{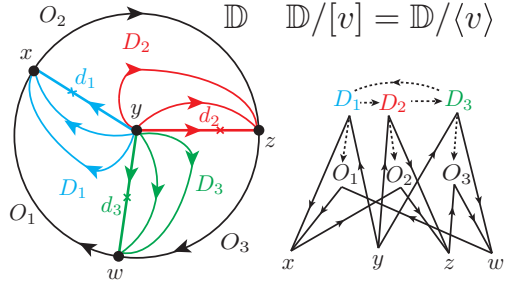}
\end{center}
\caption{A flow $v$ on a closed disk $\mathbb{D}$ whose binary relations $\leq_\partial $ and $\leq_\pitchfork$ are not transitive such that $\mathbb{D}$ consists of ten abstract (weak) orbits $\langle w \rangle = [w]=\{w\}$, $\langle x \rangle = [x]=\{ x \}$, $\langle y \rangle = [y]=\{ y \}$, $\langle z \rangle = [z] = \{ z \}$, $O_1, O_2, O_3$, $\langle d_1 \rangle = [d_1] = D_1$, $\langle d_2 \rangle = [d_2] = D_2$, $\langle d_3 \rangle = [d_3] = D_3$ with $\mathbb{D} = \{w, x, y, z \} \sqcup O_1 \sqcup O_2 \sqcup O_3 \sqcup D_1 \sqcup D_2 \sqcup D_3$, and that the quotients $D_1/v$, $D_2/v$ and $D_3/v$ are half-open intervals, and there are points $d_1 \in D_1 \cap (\overline{D_2} - D_2) \subseteq \overline{[d_2]} - [d_2]$, $d_2 \in D_2 \cap (\overline{D_3} - D_3) \subseteq \overline{[d_3]} - [d_3]$ and $d_3 \in D_3 \cap (\overline{D_1} - D_1) \subseteq \overline{[d_1]} - [d_1]$. Then both $d_1 \leq_\partial d_2 \leq_\partial d_3 \leq_\partial d_1$  and $d_1 \leq_\pitchfork d_2 \leq_\pitchfork d_3  \leq_\pitchfork d_1$ but $d_1 \not\leq_\partial d_3$ and $d_1 \not\leq_\pitchfork d_3$. In the diagram on the right, solid lines form the Hessian diagram of the order $<_{v}$ whose down (resp. up) arrows correspond to $<_{\omega}$ (resp. $<_{\alpha}$), and dotted directed lines correspond to $<_\pitchfork$. Moreover, the binary relation $<_\partial$ is the union of $<_v$ and $<_\pitchfork$ in this case.}
\label{non_transitive}
\end{figure}
Moreover, the binary relations $\leq_\partial $ and $\leq_\pitchfork$ on $X$ canonically imply binary relations, also denoted by $\leq_\partial $ and $\leq_\pitchfork$, on the orbit space $X/v$ as follows:
$$O(x) \leq_\partial O(y) \text{ if } x \leq_\partial y$$
$$O(x) \leq_\perp O(y) \text{ if } x \leq_\perp y$$
$$O(x) \leq_\pitchfork O(y) \text{ if } x \leq_\pitchfork y$$
Note that the binary relations $\leq_\partial$, $\leq_\perp$, and $\leq_\pitchfork$ are independent of the choice of points in $O(x)$ and $O(y)$.
In addition, we define binary relations $\leq_\partial$, $\leq_\perp$, and $\leq_\pitchfork$ on the abstract orbit weak class space $X/[v]$ as follows:
$$[x] \leq_\partial [y] \text{ if either } [x] = [y] \text{ or } x_1 \leq_\partial y \text{ for some } x_1 \in [x]$$
$$[x] \leq_\perp [y] \text{ if either } [x] = [y] \text{ or } x_1 \leq_\perp y \text{ for some } x_1 \in [x]$$
$$[x] \leq_\pitchfork [y] \text{ if either } [x] = [y] \text{ or } x_1 \leq_\pitchfork y \text{ for some } x_1 \in [x]$$
Note that the binary relations $\leq_\partial$, $\leq_\perp$, and $\leq_\pitchfork$ are independent of the choice of points in $[x]$ and $[y]$.
By definition, we have the following observation.
\begin{lemma}
The following properties hold for points $x,y \in X$: \\
$(1)$ $[x] \leq_\partial [y]$ if and only if $[x] = [y]$ or $[x] \cap \bp [y] \neq \emptyset$.
\\
$(2)$ $[x] \leq_\perp [y]$ if and only if $[x] = [y]$ or $[x] \cap  \partial_\perp [y] \neq \emptyset$.
\\
$(3)$ $[x] \leq_\pitchfork [y]$ if and only if $[x] = [y]$ or $[x] \cap  \partial_\pitchfork [y] \neq \emptyset$.
\end{lemma}

Similarly, we define binary relations $\leq_\partial$, $\leq_\perp$, and $\leq_\pitchfork$ on the abstract orbit space $X/\langle v \rangle$ as follows:
$$\langle x \rangle \leq_\partial \langle y \rangle \text{ if either } \langle x \rangle = \langle y \rangle \text{ or } x_1 \in \overline{\langle y \rangle} \text{ for some } x_1 \in \langle x \rangle$$
$$\langle x \rangle \leq_\perp \langle y \rangle \text{ if either } \langle x \rangle = \langle y \rangle \text{ or } x_1 \in \bigcup_{z \in \langle y \rangle} (\alpha(z) \cup \omega(z)) \text{ for some } x_1 \in \langle x \rangle$$
$$\langle x \rangle \leq_\pitchfork \langle y \rangle \text{ if either } \langle x \rangle = \langle y \rangle \text{ or } x_1 \in \overline{\langle y \rangle} - \left( \bigcup_{z \in \langle y \rangle} (\alpha(z) \cup \omega(z)) \right) \text{ for some } x_1 \in \langle x \rangle$$
Note that the binary relations $\leq_\partial$, $\leq_\perp$, and $\leq_\pitchfork$ are independent of the choice of points in $\langle x \rangle$ and $\langle y \rangle$.
By definition, we have the following observation.
\begin{lemma}
The following properties hold for points $x,y \in X$: \\
$(1)$ $\langle x \rangle \leq_\partial \langle y \rangle$ if and only if $\langle x \rangle = \langle y \rangle$ or $\langle x \rangle \cap \bp \langle y \rangle \neq \emptyset$.
\\
$(2)$ $\langle x \rangle \leq_\perp \langle y \rangle$ if and only if $\langle x \rangle = \langle y \rangle$ or $\langle x \rangle \cap  \partial_\perp \langle y \rangle \neq \emptyset$.
\\
$(3)$ $\langle x \rangle \leq_\pitchfork \langle y \rangle$ if and only if $\langle x \rangle = \langle y \rangle$ or $\langle x \rangle \cap  \partial_\pitchfork \langle y \rangle \neq \emptyset$.
\end{lemma}
We equip the abstract (weak) orbit spaces with either three relations $\leq_\alpha, \leq_\omega$, and $\leq_\pitchfork$.
Recall that the abstract weak orbit $[y]$ need not be contained in the closure $\overline{[x]}$ even if $y \in \overline{[x]}$, and that the abstract orbit $\langle y \rangle$ need not be contained in the closure $\overline{\langle x \rangle}$ even if $y \in \overline{\langle x \rangle}$.
Moreover, $\pi_{[v]}(\overline{[y]}) \subsetneq \overline{[y]}^{\tau_{[v]}}$ and $\pi_{\langle v \rangle}(\overline{\langle y \rangle}) \subsetneq \overline{\langle y \rangle}^{\tau_{\langle v \rangle}}$ in general, where $\pi_{[v]} \colon X \to X/[v]$  and $\pi_{\langle v \rangle} \colon X \to X/\langle v \rangle$ are the quotient maps, $\overline{\mathcal{A}}^{\tau_{[v]}}$ is the closure of a subset $\mathcal{A} \subseteq X/[v]$ with respect to the quotient topology $\tau_{[v]}$ on $X/[v]$, and $\overline{\mathcal{B}}^{\tau_{\langle v \rangle}}$ is the closure of a subset $\mathcal{B} \subseteq X/\langle v \rangle$ with respect to the quotient topologies $\tau_{\langle v \rangle}$ on $X/\langle v \rangle$.
Indeed, in the example in Figure~\ref{non_transitive}, we have $\pi_{\langle v \rangle}(\overline{\langle d_2 \rangle}) = \pi_{[v]}(\overline{[d_2]}) = \{ \{ x \}, \{ y \}, \{ z \}, O_2, [d_1], [d_2] \} \subsetneq \mathbb{D}/[v] = \overline{[d_2]}^{\tau_{[v]}} = \overline{\langle d_2 \rangle}^{\tau_{\langle v \rangle}}$ in the abstract (weak) orbit space $\mathbb{D}/\langle v \rangle = \mathbb{D}/[v] = \{ [w], [x], [y], [z], O_1, O_2, O_3, [d_1], [d_2], [d_3] \}$.

\subsubsection{Pre-orders for orbit classes}
Pre-orders for orbits class can be defined as weak orbit classes.
For instance, we define $\leq_\alpha, \leq_\omega, \leq_\pitchfork$ on the orbit class space $X/\hat{v}$ as follows:
\\
$\hat{O} \leq_\alpha \hat{O}' \text{ if either } \hat{O} = \hat{O}' \text{ or } x \in \alpha (y) \text{ for some } x \in \hat{O} \text{ and } y \in \hat{O}'$
\\
$\hat{O} \leq_\omega \hat{O}' \text{ if either } \hat{O} = \hat{O}' \text{ or } x_1 \in \omega (y_1) \text{ for some } x \in \hat{O} \text{ and } y \in \hat{O}'$
\\
$\hat{O} \leq_\pitchfork \hat{O}' \text{ if either } \hat{O} = \hat{O}' \text{ or } x \leq_\pitchfork y \text{ for some } x \in \hat{O} \text{ and } y \in \hat{O}'$
\\
We obtain an observation.
\begin{lemma}
The following statements hold for any orbit class $\hat{O} \neq \hat{O}'$:
\\
$(1)$ $\hat{O} <_\alpha \hat{O}'$ if and only if $\hat{O} \cap \bigcup_{y \in \hat{O}'}\alpha(y) \neq \emptyset$ {\rm(i.e.} $\hat{O} \cap \alpha(\hat{O}') \neq \emptyset$ {\rm)}.
\\
$(2)$ $\hat{O} <_\omega \hat{O}'$ if and only if $\hat{O} \cap \bigcup_{y \in \hat{O}'}\omega(y) \neq \emptyset$ {\rm(i.e.} $\hat{O} \cap \omega(\hat{O}') \neq \emptyset$ {\rm)}.
\end{lemma}

\subsubsection{Higher equivalence classes of orbits}
Let $X$ be a topological space with either a flow.
Put $[x]''_1 := [x]''$ for any point $x \in X$.
Define a saturated subset $[x]''_{k+1}$ of a point $x$ of $X$ for any $k \in \Z_{>0}$ as follows:
\[
  [x]''_{k+1} := \begin{cases}
    [x]'' & \text{if } x \in X - \mathrm{P}(v)  \\
    \left\{ y \in \mathrm{P}(v) \mid [\alpha(y)]''_k = [\alpha(x)]''_k, [\omega(y)]''_k = [\omega(x)]''_k \right\} & \text{if } x \in \mathrm{P}(v)
  \end{cases}
\]
where $[A]''_k := \bigcup_{a \in A} [a]''_k$.
For any $k \in \Z_{>0}$ and any point $x \in X$, define a $k$-th abstract weak orbit $[x]_k$ and $k$-th abstract orbit $\langle x \rangle_k$ of $X$ as follows:
\[
  [x]_k := \begin{cases}
    [x]''_k & \text{if } x \in \mathop{\mathrm{Cl}}(v) \sqcup \mathrm{P}(v) = X - \mathrm{R}(v)  \\
    \check{O}(x) & \text{if } x \in \mathrm{R}(v)
  \end{cases}
\]
\[
  \langle x \rangle_k := \begin{cases}
    [x]''_k & \text{if } x \in \mathop{\mathrm{Cl}}(v) \sqcup \mathrm{P}(v) = X - \mathrm{R}(v)  \\
    \hat{O}(x) & \text{if } x \in \mathrm{R}(v)
  \end{cases}
\]
Define the $k$-th abstract weak orbit space $X/[v]_k$ as a quotient space $X/\sim_{[v]_k}$ defined by $x \sim_{[v]_k} y$ if $[x]_k = [y]_k$.
Similarly, define the abstract orbit space $X/\langle v \rangle_k$ as a quotient space $X/\sim_{\langle v \rangle_k}$ defined by $x \sim_{\langle v \rangle_k} y$ if $\langle x \rangle_k = \langle y \rangle_k$.
For a subset $A$, the $k$-th abstract weak orbit of $A$ is defined by $[A]_k:= \bigcup_{a \in A} [a]_k$ and the abstract orbit of $A$ is defined by $\langle A \rangle_k := \bigcup_{a \in A} \langle a \rangle_k$.
By definition, we have the following observations.

\begin{lemma}\label{lem:k_th_weak}
Let $v$ be a flow on a topological space $X$.
If $[\alpha(x)] = \alpha(x)$ and $[\omega(x)] = \omega(x)$ for any point $x$, then the abstract weak orbit space coincides with the $k$-th abstract weak orbit space {\rm (i.e.} $X/[v] =X/[v]_k$ {\rm)} for any $k \in \Z_{>0}$.
\end{lemma}

\begin{lemma}\label{lem:k_th}
Let $v$ be a flow on a topological space.
If $\langle \alpha(x) \rangle = \alpha(x)$ and $\langle \omega(x) \rangle = \omega(x)$ for any point $x$, then the abstract orbit space coincides with the $k$-th abstract weak orbit space {\rm (i.e.}  $X/\langle v \rangle =X/\langle v \rangle_k$ {\rm)} for any $k \in \Z_{>0}$.
\end{lemma}

\begin{corollary}\label{cor:k_th}
Let $v$ be a flow on a topological space.
If the $\alpha$-limit set and $\omega$-limit set of any point are abstract weak orbits, then the abstract weak orbit space coincides with the $k$-th  abstract weak orbit space {\rm (i.e.} $X/[v] =X/[v]_k$ {\rm)} and the abstract orbit space coincides with the $k$-th abstract orbit space  {\rm (i.e.}  $X/\langle v \rangle =X/\langle v \rangle_k$ {\rm)}  for any $k \in \Z_{>0}$.
\end{corollary}

\begin{proof}
By Lemma~\ref{lem:k_th}, it suffices to show that the $\alpha$-limit set and $\omega$-limit set of any point are abstract orbits.
Indeed, let $\gamma$ be either the $\alpha$-limit set or the $\omega$-limit set of a point $x$.
Then $\gamma$ is closed.
Fix a point $x \in \gamma$ with $\gamma = [x]$.
If $x \in X - \mathrm{R}(v)$, then $\gamma = [x] = [x]'' = \langle x \rangle$ by definitions of abstract weak orbit and abstract orbit.
If $x \in \mathrm{R}(v)$, then $\gamma = [x] = \overline{O(x)} = \langle x \rangle$ because $\gamma = [x]$ is closed and $O(x) \subseteq [x] \subseteq \langle x \rangle \subseteq \overline{O(x)}$.
\end{proof}

\subsection{Quotient spaces and fundamental notion of homeomorphisms}

\subsubsection{Suspension flows of homeomorphisms}

For a homeomorphism $f \colon X \to X$ on a topological space $X$, consider a quotient space $X_f := (X \times \R)/\sim_{\mathrm{susp}}$ defined by $(x,t) \sim_{\mathrm{susp}} ( f^n(x), t - n)$ for any $t \in \R$ and $n \in \Z$, and define a flow $v_f \colon \R \times X \to \R \times X$ by $v_f(t, (x,s)/\sim_{\mathrm{susp}}) =(x,t+s)/\sim_{\mathrm{susp}}$, where $(x,s)/\sim_{\mathrm{susp}}$ is the equivalence class of a point $(x,s) \in X \times \R$.
Then $v_f$ is called the suspension flow of $f$ and $X_f$ is called the mapping torus of $f$.
We have the following observation.
\begin{lemma}
Let $f \colon X \to X$ be a homeomorphism  on a topological space $X$ and $v_f$ the suspension flow of $f$ on the mapping torus $X_f$.
The orbit spaces $X/f$ and $X_f/v_f$ are homeomorphic.
\end{lemma}
\begin{proof}
Let $\widetilde{v_f}$ be a flow on $X \times \R$ by $\widetilde{v_f}(t,(x,s)) :=  (x, t+s)$ and $F \colon X \times \R \to X \times \R$ a homeomorphism defined by $F(x,t) = (f(x), t-1)$.
Identify $X$ with the orbit space $(X \times \R)/\widetilde{v_f} = (X \times \R)/(x,s) \sim (x,t)$.
The mapping torus $X_f = (X \times \R)/(x,t) \sim ( f^n(x), t - n)$ equals to the orbit space $(X \times \R)/F$.
Since actions $F$ and $\widetilde{v_f}$ are commutative, we obtain $X_f/v_f = ((X \times \R)/F)/v_f \approx (X \times \R)/\langle \widetilde{v_f}, F\rangle \approx ((X \times \R)/ \widetilde{v_f})/ f \approx X/f$, where $\langle \widetilde{v_f}, F \rangle \cong \R \times \Z$ is the abelian group generated by abelian groups $v_f \cong \R$ and $F \cong \Z$.
\end{proof}


\subsubsection{Notion of homeomorphisms}
Let $f$ be a homeomorphism on a topological space $X$.
An orbit $O(x)$ of $x$ is $\{ f^n(x) \mid n \in \Z \}$.
A point $x \in X$ is fixed if $f(x) = x$, and is periodic if there is a positive integer $k$ such that $f^k(x) = x$.
The $\omega$-limit set of a point $x \in X$ is $\omega(x) := \bigcap_{n\in \mathbb{Z}}\overline{\{f^m(x) \mid m > n\}}$, and that the $\alpha$-limit set of $x$ is $\alpha(x) := \bigcap_{n\in \mathbb{Z}}\overline{\{f^m(x) \mid m < n\}}$.
Moreover, put $\alpha'(x) := \alpha(x) \setminus O(x)$  and $\omega'(x) := \omega(x) \setminus O(x)$.
A point $x$ is recurrent if $x \in \alpha(x) \cup \omega(x)$.
Denote by $\mathop{\mathrm{Fix}}(f)$ (resp. $\mathop{\mathrm{Per}}(f)$, $\mathrm{R}(f)$, $\mathrm{P}(f)$) the set of fixed points (resp. periodic points, non-periodic recurrent points, non-recurrent points).
Then $X = \mathop{\mathrm{Per}}(f) \sqcup \mathrm{R}(f) \sqcup \mathrm{P}(f)$.
For any point $x \in X$, define the orbit class $\hat{O}(x)$ as follows:
$\hat{O}(x) := \{ y \in X \mid \overline{O(x)} = \overline{O(y)} \} \subseteq \overline{O}(x)$.
For any point $x \in X$, define the weak orbit class $\check{O}(x)$ as follows:
$\check{O}(x) := \{ y \in X \mid \overline{O(x)} = \overline{O(y)}, \alpha(x) = \alpha(y), \omega(x) = \omega(y) \} \subseteq \hat{O}(x)$.
By definition, we have $O(x) \subseteq \check{O}(x) \subseteq \hat{O}(x) \subseteq \overline{O(x)}$.
For a homeomorphism $f$, define an abstract orbit, an abstract weak orbit, and so on using the suspension flow $v_f$ via the canonical homeomorphism $h_f \colon X_f/v_f \to X/f$.
Indeed, let $\pi_f \colon X \to X/f$ be the quotient map and $\pi_{v_f} \colon X_f \to X_f/v_f$ be the quotient map.
The abstract weak orbit $[x]_f$ of a point $x \in X$ by $f$ is defined by $[x]_f := \pi_f^{-1}(h_f \circ \pi_{v_f}([(x,0)/\sim_{\mathrm{susp}}]_{v_f}))$, where $[(x,0)/\sim_{\mathrm{susp}}]_{v_f}$ is the abstract weak orbit of the point $(x,0)/\sim_{\mathrm{susp}} \in X_f$ by the suspension flow $v_f$.
Similarly, the abstract orbit $\langle x \rangle_f$ of a point $x \in X$ by $f$ is defined by $\langle x \rangle_f := \pi_f^{-1}(h_f \circ \pi_{v_f} (\langle(x,0)/\sim_{\mathrm{susp}}\rangle_{v_f}))$, where $\langle (x,0)/\sim_{\mathrm{susp}} \rangle_{v_f}$ is the abstract orbit of the point $(x,0)/\sim_{\mathrm{susp}} \in X_f$ by the suspension flow $v_f$.
The abstract weak orbit space $X/[f]$ (resp. abstract orbit space $X/\langle f \rangle$) of a homeomorphism is the decomposition $\{ [x]_f \mid x \in X \}$ (resp. $\{ \langle x \rangle_f \mid x \in X \}$) with the quotient topology.
By definitions, the abstract weak orbit spaces $X/[f]$ and $X_f/[v_f]$ are homeomorphic, and so are the abstract orbit spaces $X/\langle f \rangle$ and $X_f/\langle v_f \rangle$.
Moreover, the $k$-th abstract weak orbit $[x]_{f,k}$ of a point $x \in X$ by $f$ is defined by $[x]_{f,k} := \pi_f^{-1}(h_f \circ \pi_{v_f} ([(x,0)/\sim_{\mathrm{susp}}]_{v_f,k}))$, and
the $k$-th abstract orbit $\langle x \rangle_{f,k}$ of a point $x \in X$ by $f$ is defined by $\langle x \rangle_{f,k} := \pi_f^{-1}(h_f \circ \pi_{v_f} (\langle(x,0)/\sim_{\mathrm{susp}}\rangle_{v_f,k}))$.
The $k$-th abstract weak orbit space $X/[f]_k$ (resp. $k$-th abstract orbit space $X/\langle f \rangle_k$) of a homeomorphism is the decomposition $\{ [x]_{f,k} \mid x \in X \}$ (resp. $\{ \langle x \rangle_{f,k} \mid x \in X \}$) with the quotient topology.

\section{Fundamental properties of flows on topological spaces}
Let $v$ a flow on a topological space $X$.

\subsection{Transitivity and inclusions of binary relations}

The binary relations $\leq_{\alpha}$ and $\leq_{\omega}$ are pre-orders on the abstract weak orbit space $X/[v]$ and the abstract orbit space $X/\langle v \rangle$.

\begin{lemma}\label{lem:preorders}
The binary relations $\leq_{\alpha}$ and $\leq_{\omega}$ on the abstract weak orbit space $X/[v]$ of a flow $v$ on a topological space $X$ are pre-orders.
\end{lemma}

\begin{proof}
By definitions, reflexivity holds for the relations.
Therefore it suffices to show transitivity.
%
Fix points $x,y,z \in X$ with $[x] \leq_{\alpha} [y]$, $[y] \leq_{\alpha} [x]$, $[x] \neq [y]$ and $[y] \neq [z]$.
Then $O(x) \neq O(y)$ and $O(y) \neq O(z)$, and there are points $x_1 \in [x]$, $y_1, y_2 \in [y]$ and $z_2 \in [z]$ such that $x_1 \in \alpha(y_1)$ and $y_2 \in \alpha(z_2)$.
Since $O(x_1) \neq O(y_1)$, we have $x_1 \in \alpha(y_1) - O(y_1) = \alpha'(y_1)$.
Since $\alpha'(y_1) = \alpha'(y_2)$ if $y \notin \mathrm{R}(v)$, and since $\alpha(y_1) = \alpha(y_2)$ if $y \in \mathrm{R}(v)$, we obtain $x_1 \in \alpha'(y_1) \subseteq \alpha(y_2)$.
By the time reversion, symmetry implies that the binary relation $\leq_{\omega}$ is a pre-order.
\end{proof}

\begin{lemma}\label{lem:preorders02}
The binary relations $\leq_{\alpha}$ and $\leq_{\omega}$ on the abstract orbit space $X/\langle v \rangle$ of a flow $v$ on a topological space $X$  are pre-orders.
\end{lemma}

\begin{proof}
By definitions, reflexivity holds for the relations.
Therefore it suffices to show transitivity.
Fix points $x,y,z \in X$ with $\langle x \rangle \leq_{\alpha} \langle y \rangle$, $\langle y \rangle \leq_{\alpha} \langle x \rangle$, $\langle x \rangle \neq \langle y \rangle$ and $\langle y \rangle \neq \langle z \rangle$.
Then $O(x) \neq O(y)$ and $O(y) \neq O(z)$, and there are points $x_1 \in \langle x \rangle$, $y_1, y_2 \in \langle y \rangle$ and $z_2 \in \langle z \rangle$ such that $x_1 \in \alpha(y_1)$ and $y_2 \in \alpha(z_2)$.
Then $\alpha(y_2) \subseteq \overline{O(y_2)} \subseteq \alpha(z_2)$.
Suppose that $y \notin \mathrm{R}(v)$.
Then $\langle y \rangle = [y]''$ and so $\alpha'(y_1) = \alpha'(y) = \alpha'(y_2)$.
Since $O(x_1) \neq O(y_1)$, we have $x_1 \in \alpha(y_1) - O(y_1) = \alpha'(y_1)= \alpha'(y_2) \subseteq \alpha(z_2)$.
Suppose that $y \in \mathrm{R}(v)$.
Then $\langle y \rangle = \hat{O}(y)$ and so
$x_1 \in \alpha(y_1) \subseteq \overline{O(y_1)} = \overline{O(y)} = \overline{O(y_2)}   \subseteq \alpha(z_2)$.
By the time reversion, symmetry implies that the binary relation $\leq_{\omega}$ is a pre-order.
\end{proof}

\subsection{Properties of orbits, points, and limit sets}

We have the following properties

\begin{lemma}\label{lem:properness}
Any non-closed proper orbit of $v$ is not recurrent.
\end{lemma}

\begin{proof}
Let $x$ be a point whose orbit is non-closed proper.
Define a mapping $v_x  \colon  \R \to X$ by $v_x :=v( \cdot, x)$.
Then the restriction $v_x$ of $v$ is continuous.
Then there is an open \nbd $U_x$ of $x$ such that $\overline{O(x)} \cap U_x = O(x) \cap U_x$ and that $v_x^{-1}(U_x)$ is the open interval $(-1,1)$ in $\R$.
We have that $U_x \cap v(\R - (-1,1), x) = \emptyset$ and so that
$x \notin \overline{v(\R_{<-1}, x)}$ and $x \notin \overline{v(\R_{>1}, x)}$.
By definitions of $\alpha$-limit set and $\omega$-limit set, we obtain $x \notin \alpha(x) \cup \omega(x)$.
This means that $x$ is not recurrent.
\end{proof}

\begin{lemma}\label{lem:kc}
Suppose that $X$ is KC.
Then a non-recurrent orbit $O$ of $v$ has an open neighborhood in which the orbit is closed.
Moreover, the coborder $\overline{O} - O$ is closed.
\end{lemma}

\begin{proof}
Let $O$ be an orbit.
Fix a point $x \in O$.
Define a continuous mapping $v_x  \colon  \R \to X$ by $v_x :=v( \cdot, x)$.
Suppose that $O$ is non-recurrent.
Then $x \notin \alpha(x) \cup \omega(x)$.
There is a number $T > 0$ such that $x \notin \overline{v_x(\R_{\leq -T})}$ and $x \notin \overline{v_x(\R_{\geq T})}$.
Moreover, there is an open \nbd $U_x$ of $x$ such that $U_x \cap v_x(\R - [-T,T]) = \emptyset$.
The compactness of the closed interval $[-T,T]$ implies that the image $v_x([-T,T])$ is compact.
Since $X$ is KC, the image $v_x([-T,T])$ is closed.
Then we have $\overline{O} \cap U_x = \overline{v_x(\R - [-T,T]) \cup v_x([-T,T])} \cap U_x = (\overline{v_x(\R - [-T,T])} \cup v_x([-T,T])) \cap U = v_x([-T,T]) \cap U_x = O \cap U_x$.
Therefore the union $U := \bigcup_{x \in O} U_x$ is an open \nbd of $O$ such that $\overline{O} \cap U = \bigcup_{x \in O} (\overline{O} \cap U_x) = \bigcup_{x \in O} (O \cap U_x) = O \cap U = O$.
Since $\overline{O} \cap U = O$, the difference $\overline{O} \setminus U = \overline{O} - O$ is closed.
\end{proof}

The converse does not hold if $X$ is not KC.
In other words, there is a flow on a compact $T_1$ space with a recurrent orbit which has a neighborhood in which the orbit is closed.
In fact, an $\R$-action $v \colon  \R \times X \to X$ by $v(t,x) := t + x$ with respect to the cointerval topology is a flow (see Lemma~\ref{lem:cointerval}) such that $X$ consists of one orbit $X$ such that $X = \R = \alpha(x) = \omega(x)$ for any $x \in X$.
The author would like to know whether the converse for a KC space holds or not.
In other words, is any orbit $O$ of $v$ of a flow on a KC space which has an open neighborhood in which the orbit is closed, non-recurrent?
%
Recall the following connectivity of $\alpha$-limit sets and $\omega$-limit sets.

\begin{lemma}\label{lem:connectivity}
The following properties hold for a flow $v$ on a topological space $X$ and for a point $x \in X$:
\\
$(1)$ The $\alpha$-limit set $\alpha(x)$ and the $\omega$-limit set $\omega(x)$ is closed and invariant.
\\
$(2)$ If $X$ is compact, then $\alpha(x) \neq \emptyset$ and $\omega(x) \neq \emptyset$.
\\
$(3)$ If $X$ is either compact KC or sequentially compact, then the limit sets $\alpha(x)$ and $\omega(x)$ are connected.
\end{lemma}

\begin{proof}
By definitions, the $\alpha$-limit set $\alpha(x)$ and the $\omega$-limit set $\omega(x)$ is closed and invariant.
Since the orbit $O(x)$ is a continuous image of $\R$, the orbit $O(x)$ is connected.
If $X$ is compact, then the finite intersection property of a compact space implies that the $\alpha$-limit set and the $\omega$-limit set of a point is nonempty.
Suppose that $O(x)$ is singular or periodic.
Since the closure of a connected subset is connected, the closure $\overline{O(x)} = \bigcap_{t \in \R} \overline{v(\R_{<t}, x)} = \alpha(x)$ is connected.
Thus we may assume that $O(x)$ is neither singular nor periodic.
Then the continuous mapping $v_x  \colon  \R \to X$ defined by $v_x :=v( \cdot, x)$
is injective.
Suppose that $x \in \alpha(x)$.
The invariance of $\alpha(x)$ implies that $O(x) \subseteq \alpha(x)$.
The closedness of $\alpha(x)$ implies that $\overline{O(x)} = \alpha(x)$.
This implies that $\alpha(x) = \overline{O(x)}$ is connected.
Thus we may assume that $x \notin \alpha(x)$.
For any $y \in O(x)$, there is a negative number $T_y$ such that $y \notin \overline{v_x(\R_{<T_y})}$, and so there is an open \nbd $U_y$ of $y$ such that $U_y \cap v_x(\R_{<T_y}) = \emptyset$.
Assume that $\alpha (x)$ is disconnected.
The disconnectivity implies that there are disjoint open subsets $U$ and $V$ such that $\alpha(x) \subseteq U \sqcup V$, $\alpha(x) \cap U \neq \emptyset$, and $\alpha(x) \cap V \neq \emptyset$.
The connectivity of $v_x(\R_{<t})$ for any $t \in \R$ implies that $v_x(\R_{<t}) \setminus U \sqcup V \neq \emptyset$ for any $t \in \R$.
There is an unbounded decreasing sequence $(t_n)_{n \in \Z_{>0}}$ such that $v_x(t_{n}) \in X - (U \sqcup V)$.
Then the difference $K := O^{-0}(x) \setminus (U \sqcup V)$ is a subset and contains $\{ v_x(t_{n}) \}_{n \in \Z_{>0}}$.
Here $O^{-0}(x) := O^{-}(x) \sqcup \{ x\} = v_x(\R_{\leq 0})$.
Suppose that $X$ is sequentially compact.
Then there is a convergence sequence $(t_{n_k})$ which is a subsequence of $(t_n)_{n \in \Z_{>0}}$.
Since $v_x(t_n) \in K \subseteq X - (U \sqcup V)$, the closedness of $X - (U \sqcup V)$ implies that $\lim v_x(t_{n_k}) \in (X - (U \sqcup V)) \cap \alpha(x) = \alpha(x) \setminus (U \sqcup V) = \emptyset$, which is a contradiction.
Suppose that $X$ is compact and KC.
We claim that $K = \overline{O^{-0}(x)} \setminus (U \sqcup V)$.
Indeed, assume that there is a point $y \in (\overline{O^{-0}(x)} \setminus (U \sqcup V)) - K = \overline{O^{-0}(x)} \setminus (O^{-0}(x) \cup U \cup V)$.
Then any \nbd of $y$ intersects $O^{-0}(x)$.
For any $t \in \R_{<0}$, the closed interval $[t,0]$ is compact in $\R$ and so the image $v_x([t,0])$ is compact.
KC property implies that $v_x([t,0])$ is closed and so the complement $U_t := X - v_x([t,0])$ of the image $v_x([t,0])$ is an open \nbd of $y$.
For any $t \in \R_{<0}$ and any \nbd $W$ of $y$, we have $v_x(\R_{<t}) \cap W = O^{-0}(x) \cap (U_t \cap W) \neq \emptyset$.
This means that $y \in \bigcap_{t \in \R_{<0}} \overline{v_x(\R_{<t})} = \bigcap_{t \in \R} \overline{v_x(\R_{<t})} = \alpha(x) \subseteq U \sqcup V$, which contradicts $y \notin U \sqcup V$.
Therefore the nonempty subset $K$ is closed and so compact.
By $K \cap \alpha(x) = \emptyset$ and $K \subset O^{-0}(x)$, for any point $y \in K$, there are an negative number $T_y < 0$ and an open \nbd $U_y$ of $y$ such that $U_y \cap v(\R_{<T_y}, y) = \emptyset$.
The compactness of $K$ implies that there are finitely many $x_1, x_2, \ldots , x_k \in K$ such that $K \subseteq \bigcup_{i=1}^k U_{x_i}$.
Since $K \subset O^{-0}(x)$, there are negative numbers $s_1, s_2, \ldots , s_k < 0$ such that $x_i = v(s_i, x)$.
Put $T: = \min_i \{ T_{x_i} \} + \min_i \{ s_i \} < 0$.
Since $v(\R_{<t}, x_i) = v(\R_{<t + s_i}, x)$ for any $t \in \R$ and since $v(\R_{<t}, y) \subseteq v(\R_{<t'}, y)$ for any $t<t'$, we have that  $U_{x_i} \cap v(\R_{<T}, x) \subseteq U_{x_i} \cap v(\R_{<T_{x_i}+s_i}, x) = U_{x_i} \cap v(\R_{<T_{x_i}}, x_i) = \emptyset$ for any $i = 1,2, \ldots , k$.
Therefore $K \cap v(\R_{<T}, x) \subseteq \bigcup_{i=1}^k U_{x_i} \cap v(\R_{<T}, x) = \emptyset$, which contradicts the existence of the unbounded decreasing sequence $(t_n)_{n \in \Z_{>0}}$ with $v_x(t_{n}) \in K$.
This implies that  the $\alpha$-limit set $\alpha(x)$ is connected.
By symmetry, the $\omega$-limit set $\omega(x)$ is connected.
\end{proof}

The author would like to know an answer to the following question.
\begin{question}
Are the $\alpha$-limit set and the $\omega$-limit set for a flow of any point on of a compact space connected?
\end{question}

\subsection{Properties of flows on typical topological spaces with low separation axioms}

We have the following continuity of $\R$-actions on discrete spaces.

\begin{lemma}
Any $\R$-actions on indiscrete spaces are continuous.
In other words, any $\R$-actions on indiscrete spaces are flows.
\end{lemma}

\begin{proof}
Let $X$ be an indiscrete topological space and $v \colon \R \times X \to X$ an $\R$-action.
Then the open subset of $X$ is either the emptyset or the whole space $X$.
Therefore the inverse images $v^{-1}(\emptyset) = \emptyset$ and $v^{-1}(X) = \R \times X$ are open subsets of the product topological space $\R \times X$.
\end{proof}

The product of flows on an indiscrete space and a topological space is also a flow.

\begin{corollary}\label{cor:indiscrete}
Let $v_1$ be a flow on a topological space $X_1$, $v_2$ a flow on a topological space $X_2$, $X := X_1 \times X_2$ the product topological space.
If either $X_1$ or $X_2$ is indiscrete, then the induced $\R$-action $v := (v_1, v_2) \colon \R \times X \to X$ by $v(t, (x_1, x_2)) := (v_1(t,x_1), v_2(t, x_2))$ is continuous.
\end{corollary}

\begin{proof}
By symmetry, we may assume that $X_2$ is indiscrete.
Let $\mathcal{O}_i$ be the topologies of $X_i$ and $\mathcal{O}$ the product topology of $X$.
Since $X_2$ is indiscrete, by definition of product topology, we have that $\mathcal{O} = \{ U \times X_2 \mid U \in \mathcal{O}_1 \}$.
For any open subset $U \times X_2$ of $X$, since $v^{-1}_1(U)$ is an open subset of $\R \times X_1$, the inverse image $v^{-1}(U \times X_2) = v^{-1}_1(U) \times  X_2$ is an open subset of $\R \times X_1 \times X_2$.
\end{proof}

We have the following triviality of flows on a discrete topological space.

\begin{lemma}\label{lem:discrete}
Any flow on a discrete topological space is identical.
\end{lemma}

\begin{proof}
Let $X$ be a discrete topological space and $v$ a flow on $X$.
By definition of product topology, the direct products of elements of bases of $\R$ and $X$ form a base for the product topology of $\R \times X$.
This means that the family $\mathcal{B} := \{ I \times \{ x \} \mid I \subset \R : \text{open interval}, x \in X \}$ is a base for the product topology of $\R \times X$.
Fix a point $p = (0,x) \in \R \times X$.
Since the singleton $\{ x \}$ is an open subset of $X$, the inverse image $v^{-1}(x)$ is an open subset of $\R \times X$.
Then there is an open subset $U \in \mathcal{B}$ with $p \in U \subseteq v^{-1}(x)$.
Moreover, there is a positive number $\delta >0$ with $(-\delta, \delta) \times \{ x \} \subseteq U$.
This means that $p = (0,x) = v(t, 0)$ for any $t \in (-\delta, \delta)$ and so that $p$ is a singular point.
\end{proof}

We construct a flow on a non-$T_1$ space from a flow on a topological space as follows.

\begin{lemma}\label{lem:double}
Let $v$ be a flow on a topological space $(X, \mathcal {O}_X)$.
Define a topological space $X_2 := X \times \{ -, + \}$ whose topology is generated by a base $\mathcal{B}_2 := \{ U \times \{ + \}, U \times \{ -, + \} \mid U \in \mathcal{O}_X \}$.
Then the induced $\R$-action $v_2$ on $X_2$ defined by $v_2 (t, (x, \sigma)) := (v(t,x), \sigma)$ is continuous {\rm (i.e.} a flow {\rm)}.
\end{lemma}

\begin{proof}
Let $\mathcal{O}_{\R}$ be the Euclidean topology on $\R$.
By definition of product topology, the product topology $\mathcal{O}_{\R \times X}$ on a direct product $\R \times X$ is generated by a base $\mathcal{B}_{\R \times X} := \{ I \times U \mid I \in \mathcal{O}_{\R}, U \in \mathcal{O}_X \}$, and the product topology $\mathcal{O}_{\R \times X_2}$ on a direct product $\R \times X_2$ is generated by a base $\mathcal{B}_{\R \times X_2} := \{ I \times V \mid I \in \mathcal{O}_{\R}, V \in \mathcal{B}_2 \} = \{  I \times U \times \{ + \}, I \times U \times \{ -, + \} \mid I \in \mathcal{O}_{\R}, U \in \mathcal{O}_X \}$.
This means that the topology $\mathcal{O}_{\R \times X_2}$ on $\R \times X_2 = \R \times X \times  \{ -, + \}$ is generated by a base $\{ W \times \{ + \}, W \times \{ -, + \} \mid W \in \mathcal{O}_{\R \times X}\}$.
Fix an open subset $V \in \mathcal{B}_2$ of $X_2$.
By definition of $\mathcal{B}_2$, there are $U \in \mathcal{O}_X$ and $T \subseteq \{ -, + \}$ with $V = U \times T$.
Then $T$ is either $ \{ + \}$ or $ \{ -, + \}$.
The continuity of $v$ implies that the inverse image $v^{-1}(U)$ is an open subset of $\R \times X$ (i.e. $v^{-1}(U) \in \mathcal{O}_{\R \times X}$).
By definition of $v_2$, the inverse image $v_2^{-1}(V) = v_2^{-1}(U \times T) = v^{-1}(U) \times T$ is an open subset of $\R \times X_2$.
This implies that the $\R$-action $v_2$ is continuous.
\end{proof}

Recall that the cofinite topology $\mathcal{O}_{\mathrm{cofin}}$ of a set $X$ is $\{ X - F \mid F:\text{finite}  \} \cup \{ \emptyset \}$. The pair of a set and the cofinite topology is called a cofinite topological space.
We have the following triviality of flows on a cofinite topological space.

\begin{lemma}
Any flow on a cofinite topological space is identical.
\end{lemma}

\begin{proof}
Let $X$ be a cofinite topological space and $v$ a flow on $X$.
Lemma~\ref{lem:discrete} implies that $X$ is infinite.
By definition of product topology, the direct products of elements of bases of $\R$ and $X$ form a base for the product topology of $\R \times X$.
This means that the family $\mathcal{B} := \{ I \times (X - F) \mid I \subset \R : \text{open interval}, F \subset X : \text{finite} \}$ is a base for the product topology of $\R \times X$.
Fix a point $p = (0,x) \in \R \times X$.
We claim that for any closed subset $E$ of $\R \times X$ with $E \cap (\{ 0 \} \times X) = \{ p \}$, there is a positive number $\varepsilon >0$ such that $E \cap ([- \varepsilon, \varepsilon] \times X) \subseteq [- \varepsilon,  \varepsilon] \times \{ x \}$.
Indeed, the complement $U := (\R \times X) - E$ is open.
Since the complement $X - \{ x \}$ is compact and so is the product $\{ 0 \} \times (X - \{ x \})$, there are open subsets $U_1$, \ldots , $U_k \in \mathcal{B}$ with $\{ 0 \} \times (X - \{ x \}) \subset \bigcup_{i =1}^k U_i \subseteq U$.
By $U_i \in \mathcal{B}$ for any $i = 1, \ldots , k$, there are a positive number $\delta_i >0$ and a finite subset $F_i \subset X$ with $y \notin F_i$ such that $(-\delta_i, \delta_i) \times (X - F_i) \subseteq U_i$.
Fix a positive number $\varepsilon < \min_{i} \delta_i$.
Then $[- \varepsilon, \varepsilon]  \times (X - \bigcap_{i =1}^k F_i) \subseteq U$.
Since $E \cap (\{ 0 \} \times X) = \{ p \}$ and $U = (\R \times X) - E$, we have $([- \varepsilon, \varepsilon]  \times (X - \bigcap_{i =1}^k F_i)) \cap (\{ 0 \} \times X) \subseteq U \cap (\{ 0 \} \times X) = \{ 0 \} \times (X - \{ p \})$.
This means that $\{ x \} = \bigcap_{i =1}^k F_i$.
Therefore $[- \varepsilon, \varepsilon]  \times (X - \{ x \}) = [- \varepsilon, \varepsilon]  \times (X - \bigcap_{i =1}^k F_i) \subseteq U \cap ([- \varepsilon, \varepsilon] \times X) = ([- \varepsilon, \varepsilon] \times X) \setminus E$ and so $E \cap ([- \varepsilon, \varepsilon] \times X) \subseteq [- \varepsilon, \varepsilon]  \times \{ x \}$.
Since the singleton of any point in $X$ is a closed subset, the inverse image $v^{-1}(x)$ is a closed subset with $v^{-1}(x) \cap (\{ 0 \} \times X) = \{ p \}$.
Claim implies that $p = (0,x) = v(t, x)$ for any $t \in [- \varepsilon, \varepsilon]$ and so that $p$ is a singular point.
\end{proof}

Recall that the lower topology of a pre-ordered set $X$ is the topology generated by a base $\{ X - \mathop{\downarrow} F \mid F \subseteq X : \text{finite} \}$.
For instance, the lower topology of the real line $\R$ with the canonical total order is $\{ \R_{\leq x} \mid x \in \R \} \sqcup \{ \emptyset, \R \}$.
Then we have the following observation.

\begin{lemma}
A flow on the  one-dimensional Euclidean space $\R$ is also a flow on the real line with the lower topology.
\end{lemma}

\begin{proof}
Let $v$ be a flow on the one-dimensional Euclidean space $\R$ and $R$ the real line with the lower topology.
Fix a closed subset $F$ with respect to the lower topology of $R$.
If $F = \emptyset$ (resp. $F = \R$), then the inverse image $v^{-1}(F)$ is empty (resp. $\R \times X$) which is open.
To show the continuity of $v$, we may assume that $F = \R_{\leq y}$ for some point $y \in R$.
Then the inverse image $v^{-1}(F)$ is $v^{-1}(\R_{\leq y}) = \bigcup_{t \in \R} \{t \} \times \R_{\leq v(t,y)}$.
Fix a point $p = (t,x) \notin v^{-1}(F)$.
Since $v$ is continuous with respect to the Euclidean topology, there is a small positive number $\delta > 0$ such that $\max_{s \in [-\delta, \delta]} v(t+s,y) < x$.
Put $d := x - \max_{s \in [-\delta, \delta]} v(t+s,y) > 0$.
This implies that the open subset $(t -\delta, t+\delta) \times  \R_{> x-d}$ of the product topological space $\R \times R$ does not intersect $v^{-1}(F)$.
Therefore $v^{-1}(F)$ is a closed subset of $\R \times R$.
\end{proof}

Recall that the cointerval topology of a pre-ordered set $X$ is the topology generated by a base $\{ X - [a,b] \mid a < b \} \sqcup \{ X, \emptyset \}$, where $[a,b] := \mathop{\uparrow} a \cap \mathop{\downarrow} b$.
Then we have the following observation.

\begin{lemma}\label{lem:cointerval}
A flow on the one-dimensional Euclidean space $\R$ is also a flow on the real line with the cointerval topology.
\end{lemma}

\begin{proof}
Let $v$ be a flow on the one-dimensional Euclidean space $\R$ and $R$ the real line with the lower topology.
Fix a closed subset $F$ with respect to the cointerval topology of $R$.
If $F = \emptyset$ (resp. $F = \R$), then the inverse image $v^{-1}(F)$ is empty (resp. $\R \times R$) which is open.
To show the continuity of $v$, we may assume that $F = [a,b]$ for some pair $a \leq b \in R$.
Then the inverse image $v^{-1}(F)$ is $v^{-1}([a,b]) = \bigcup_{t \in \R} \{t \} \times [v(t,a), v(t,b)]$.
Fix a point $p = (t,x) \notin v^{-1}(F)$.
Then $x \notin [v(t,a), v(t,b)]$.
Suppose that $v(t,b) < x$ (resp. $x < v(t,a)$).
Since $v$ is continuous with respect to the Euclidean topology, there is a small positive number $\delta > 0$ such that $\max_{s \in [-\delta, \delta]} v(t+s,b) < x$ (resp. $x < \min_{s \in [-\delta, \delta]} v(t+s,a)$).
Put $d := \max \{ \max_{s \in [-\delta, \delta]} v(t+s,b) -b, a - \min_{s \in [-\delta, \delta]} v(t+s,a)\} \geq 0$.
This implies that the open subset $(t - \delta, t+\delta) \times (R- [a-d, b+d])$ of the product topological space $\R \times R$ does not intersect $v^{-1}(F)$.
Therefore $v^{-1}(F)$ is a closed subset of $\R \times R$.
This means that $v$ is continuous.
\end{proof}

%

\section{Fundamental properties of flows on Hausdorff spaces}
From now on, we assume that the topological space $X$ is Hausdorff unless otherwise stated.
We have the following property.

\begin{lemma}\label{lem:decomp_limit}
For any point $x$, we have $\overline{O(x)} = \alpha(x) \cup O(x) \cup \omega(x)$.
\end{lemma}

\begin{proof}
By definition of $\alpha$-limit set and $\omega$-limit set, the orbit closure $\overline{O(x)}$ contains $\alpha(x) \cup O(x) \cup \omega(x)$.
Assume that there is a point $y \in \overline{O(x)} - (\alpha(x) \cup O(x) \cup \omega(x))$.
Define a continuous mapping $v_x  \colon \R \to X$ by $v_x :=v( \cdot, x)$.
Since a closed interval $I \subset \R$ is compact, the image $v_x(I)$ is compact.
The Hausdorff separation axiom of $X$ implies that the image of a closed interval by $v_x$ is closed.
By $y \notin \alpha(x) \cup \omega(x)$, there is a number $T > 0$, such that $y \notin \overline{v_x(\R_{<-T})}$ and $y \notin \overline{v_x(\R_{>T})}$.
Then $y \notin \overline{v_x(\R - [-T,T])}$ and so there is an open \nbd $U$ of $y$ such that $U \cap v_x(\R - [-T,T]) = \emptyset$.
Since the image $v_x([-T,T]) \subseteq O(x)$ of the interval $[-T,T]$ is closed and $y \notin O(x)$, the difference $V := U \setminus v_x([-T,T])$ is an open \nbd of $y$.
Then $V \cap O(x) = (U \setminus v_x([-T,T])) \cap O(x) = U \cap (O(x) - v_x([-T,T])) = U \cap v_x(\R - [-T,T]) = \emptyset$ and so $y \notin \overline{O(x)}$, which contradicts $y \in \overline{O(x)}$.
Thus $\overline{O(x)} = \alpha(x) \cup O(x) \cup \omega(x)$.
\end{proof}

The Hausdorff separation axiom is necessary.
In fact, there is a flow on a topological space with an orbit $O$ such that $\overline{O(x)} \supsetneq \alpha(x) \cup O(x) \cup \omega(x)$.
Indeed, let $w_1$ be a flow defined by $w_1(t,x) := x + t(1- x^2)$ on a closed interval $I_1 := [-1, 1]$ with the relative topology $\mathcal{O}_1$ of the Euclidean topology.
Then $I_1$ has exactly three orbits $\{ -1 \}$, $(-1, 1)$, and $\{ 1 \}$ with $\overline{\{ -1 \}} = \{ -1 \}$, $\overline{(-1,1)} = [-1,1]$, and $\overline{\{ 1 \}} = \{ 1 \}$.
Define a topological space $I_2 := I_1 \times \{ -, + \}$ whose topology is  generated by a base $\mathcal{B}_2 := \{ U \times \{ + \}, U \times \{ -, + \} \mid U \in \mathcal{O}_1 \}$, and a flow $w_2$ by $w_2 (t, (x, \sigma)) := (w_1(t,x), \sigma)$ because of Lemma~\ref{lem:double}.
Then the closure of the orbit $O_m := (-1, 1) \times \{ + \} = O((0,+))$ is the whole space $I_2$, and $\alpha(O_m) = \{-1\} \times \{ -, + \}$ and $\omega(O_m) = \{ 1\} \times \{ -, + \}$.
Therefore $\overline{O_m} - (\alpha(O_m) \cup O_m \cup \omega(O_m)) = (-1, 1) \times \{ - \} = O((0,-))$.
The previous lemma implies the following observation.

\begin{lemma}\label{lem:tr_bdry}
The following statements hold for a saturated subset $A \subseteq X$: \\
$(1)$ $\partial_\perp A = \left( \bigcup_{z \in A} \overline{O(z)} \right) - A =  \mathop{\downarrow}_{\leq_v}A - A$.
\\
$(2)$ $\partial_\pitchfork A = \overline{A} - \bigcup_{z \in A} \overline{O(z)} = \overline{A} - \mathop{\downarrow}_{\leq_v}A$.
\end{lemma}

\begin{proof}
By Lemma~\ref{lem:decomp_limit}, for any point $z$, we have $\overline{O(z)} = \alpha(z) \cup O(z) \cup \omega(z)$.
Therefore $\partial_\perp A = \left( \bigcup_{z \in A} \alpha(z) \cup \omega(z) \right) \setminus A = \left( \bigcup_{z \in A} \alpha(z) \cup O(z) \cup \omega(z) \right) \setminus A = \left( \bigcup_{z \in A} \overline{O(z)} \right) - A =  \mathop{\downarrow}_{\leq_v}A - A$.
Moreover, we obtain $\partial_\pitchfork A = (\overline{A} - A) \setminus \left( \bigcup_{z \in A} (\alpha(z) \cup \omega(z)) \right) = \overline{A} -  \left( \bigcup_{z \in A} (\alpha(z) \cup O(z) \cup \omega(z)) \right) = \overline{A} - \bigcup_{z \in A} \overline{O(z)} = \overline{A} - \mathop{\downarrow}_{\leq_v}A $.
\end{proof}

We have the following properties.

\begin{lemma}\label{lem:kc_closed}
Any closed orbit is proper and is a closed subset.
\end{lemma}

\begin{proof}
Let $O$ be a closed orbit and $x \in O$ a point.
Suppose that $O$ is singular.
Since any singleton $Y$ has only one topology $\{ Y, \emptyset \}$, the orbit $O$ is embedded and so proper.
Since $X$ is Hausdorff and so $T_1$, the singleton $O$ is closed.
Thus we may assume that $O$ is periodic.
Define a continuous mapping $v_x  \colon \R \to X$ by $v_x :=v( \cdot, x)$.
Then there is a number $T > 0$, called the period of $O$, such that $v_x([0,T]) = O$ and $v_x((0,T)) \subsetneq O$.
Since the closed interval $[0,T]$ is compact, the image $v_x([0,T]) = O$ is compact.
The Hausdorff separation axiom implies that $O$ is closed.
Since the restriction $v_x  \colon  [0,T) \to O$ is bijective, there is an induced mapping $v_x/T \colon \R/T\Z \to O$, where $\R/T\Z$ is a quotient space $\R/\sim$ of an equivalence relation $\sim$ by $x \sim y$ if $(x - y)/T \in \Z$.
Then the induced mapping $v_x/T \colon \R/T\Z \to O$ is a continuous bijection from the compact space $\R/T\Z$ to the Hausdorff space $O$ and so is a homeomorphism.
This means that $O$ is embedded and so proper.
\end{proof}

\begin{lemma}\label{lem:kc_proper}
The following conditions are equivalent for a point $x \in X - \Cv$: \\
$(1)$ The orbit $O(x)$ is proper.
\\
$(2)$ The orbit $O(x)$ is not recurrent.
\\
In any case, the orbit $O(x)$ has an open neighborhood in which the orbit is a closed subset and the coborder $\overline{O(x)} - O(x)$ is closed.
\end{lemma}

\begin{proof}
Fix a point $x \in X - \Cv$.
Lemma~\ref{lem:properness} implies that properness implies non-recurrence.
Lemma~\ref{lem:kc} implies that non-recurrence implies that the orbit $O(x)$ has an open neighborhood in which the orbit is a closed subset.
Suppose that $O(x)$ is not recurrent.
Then $x \notin \alpha(x) \cup \omega(x)$ and $x \in \mathrm{P}(v)$.
Define the continuous mapping $v_x  \colon  \R \to X$ by $v_x :=v( \cdot, x)$.
The non-recurrence implies that $v_x$ is injective.
There is a number $T > 0$ such that $x \notin \overline{v_x(\R_{<-T})}$ and $x \notin \overline{v_x(\R_{>T})}$.
Moreover, there is an open \nbd $U$ of $x$ such that $v_x(\R - [-T,T]) \cap U = \emptyset$.
Then $\overline{v_x(\R - [-T,T])} \cap U = \emptyset$ and $O(x) \cap U = v_x([-T,T]) \cap U$.
The compactness of the closed interval $[-T,T]$ implies that the image $v_x([-T,T])$ is compact.
Then the restriction $v_x|_{[-T,T]} \colon  [-T,T] \to v_x([-T,T])$ is a continuous bijection from a compact space to a Hausdorff space.
This implies that the restriction $v_x|_{[-T,T]} \colon  [-T,T] \to v_x([-T,T])$ is a homeomorphism and so $\overline{O(x)} \cap U = \overline{v_x(\R - [-T,T]) \cup v_x([-T,T])} \cap U = (\overline{v_x(\R - [-T,T])} \cup v_x([-T,T])) \cap U = v_x([-T,T]) \cap U = O(x) \cap U$.
We claim that the restriction $v_x  \colon  \R \to O(x)$ is a local homeomorphism.
Indeed, since $v_x([-T,T])$ is a simple closed arc and so is compact, there is an open flow box $B \subset U$ containing $x$ such that the intersection $C_x := B \cap v_x([-T,T])$ is an open orbit arc contained in $v_x([-T,T])$.
Since the restriction $v_x|_{[-T,T]} \colon  [-T,T] \to v_x([-T,T])$ is a homeomorphism, so is the restriction $v_x|_{v_x^{-1}(C_x)} \colon v_x^{-1}(C_x) \to C_x$.
By $\overline{O(x)} \cap U = v_x([-T,T]) \cap U$, we have that $\overline{O(x)} \cap B = O(x) \cap B = v_x([-T,T]) \cap B = C_x$.
This means that $C_x$ is open in $O(x)$.
Since the restriction $v_t|_{O(x)} \colon O(x) \to O(x)$ for any $t \in \R$ is a homeomorphism, the restriction $v_x  \colon  \R \to O(x)$ is a local homeomorphism.
Therefore $v_x  \colon  \R \to O(x)$ is a bijective local homeomorphism.
Since a local homeomorphism is an open mapping, the restriction $v_x  \colon  \R \to O(x)$ is  a homeomorphism.
This means that $O(x)$ is embedded and so proper.
\end{proof}

Note that the existence of a neighborhood as in the previous lemma does not imply recurrence.
Indeed, there is a flow on a Hausdorff space such that there is a recurrent orbit which has a neighborhood in which the orbit is a closed subset.
Indeed, consider an irrational rotation $w$ on a torus and fix an orbit $O$.
Then the subspace $X:= O$ is Hausdorff and the restriction $v:= w|_O$ is a flow on $X$ which consists of one recurrent orbit.
Since any closed orbit is recurrent, Lemma~\ref{lem:kc_proper} implies  characterizations of the saturated subsets $\mathrm{P}(v)$ and $\mathrm{R}(v)$,  and a characterization of recurrence.

\begin{corollary}\label{cor:t2_proper}
Let $v$ be a flow on a Hausdorff space $X$.
Then the following conditions are equivalent for an orbit $O$: \\
$(1)$ The orbit $O$ is non-recurrent  {\rm(i.e.} $O \subseteq \mathrm{P}(v)$ {\rm)}.
\\
$(2)$ The orbit $O$ is a non-closed proper orbit.
\end{corollary}


\begin{corollary}\label{cor:t2_nonclosed_recurrent}
Let $v$ be a flow on a Hausdorff space $X$.
Then the following conditions are equivalent for an orbit $O$: \\
$(1)$ The orbit $O$ is a non-closed recurrent orbit {\rm(i.e.} $O \subseteq \mathrm{R}(v)$ {\rm)}.
\\
$(2)$ The orbit $O$ is a non-closed non-proper orbit.
\end{corollary}


\begin{corollary}\label{cor:recurrent}
Let $v$ be a flow on a Hausdorff space $X$.
Then the following conditions are equivalent for an orbit $O$: \\
$(1)$ The orbit $O$ is recurrent.
\\
$(2)$ The orbit $O$ either is a closed orbit or is non-proper.
\\
$(3)$ $O \subseteq \Cv \sqcup \mathrm{R}(v) = X - \mathrm{P}(v)$.
\end{corollary}

Lemma~\ref{lem:kc_closed} and Lemma~\ref{lem:kc_proper} imply the following characterization of properness.

\begin{corollary}\label{cor:proper}
Let $v$ be a flow on a Hausdorff space $X$.
Then the following conditions are equivalent for an orbit $O$: \\
$(1)$ The orbit $O$ is proper.
\\
$(2)$ The orbit $O$ either is a closed orbit or is non-recurrent.
\\
$(3)$ $O \subseteq \Cv \sqcup \mathrm{P}(v) = X - \mathrm{R}(v)$.
\end{corollary}

\subsection{Characterizations of weak orbit classes and orbit classes}

We have the following property.

\begin{lemma}\label{lem:orbit_cl}
For any point $x \in \mathop{\mathrm{Cl}}(v) \sqcup \mathrm{P}(v)= X - \mathrm{R}(v)$, we have $O(x) = \check{O}(x) = \hat{O}(x)$.
\end{lemma}

\begin{proof}
By definition, we have $O(x) \subseteq \check{O}(x) \subseteq \hat{O}(x) \subseteq \overline{O(x)}$.
If $x \in \mathop{\mathrm{Cl}}(v)$, then Lemma~\ref{lem:kc_closed} implies $O(x) = \overline{O(x)}$ and so $O(x) = \check{O}(x) = \hat{O}(x)$.
Thus we may assume that $x \in \mathrm{P}(v)$.
Then $(\alpha (x) \cup \omega(x)) \cap O(x) = \emptyset$ and so $\alpha (x) \cup \omega(x) = \overline{O(x)} - O(x)$.
Assume that $\hat{O}(x) - O(x) \neq \emptyset$.
Fix a point $y \in \hat{O}(x) - O(x)$.
Then $O(y) \subseteq \overline{O(y)} = \overline{O(x)}$ and so $O(y) \subseteq \overline{O(x)} - O(x) = \alpha (x) \cup \omega(x)$.
The closedness of $\alpha(x)$ and $\omega(x)$ implies $x \in \overline{O(y)} \subseteq \alpha (x) \cup \omega(x)$, which contradicts that $(\alpha (x) \cup \omega(x)) \cap O(x) = \emptyset$.
Thus $O(x) = \check{O}(x) = \hat{O}(x)$.
\end{proof}

We have the following equivalences.

\begin{lemma}\label{lem:hat02}
The following properties hold for any point $x \in X$:
\\
$(1)$ $x \in \mathop{\mathrm{Cl}}(v)$ if and only if $\hat{O}(x) \subseteq \mathop{\mathrm{Cl}}(v)$.
\\
$(2)$ $x \in \mathrm{P}(v)$ if and only if $\hat{O}(x) \subseteq \mathrm{P}(v)$.
\\
$(3)$ $x \in \mathrm{R}(v)$ if and only if $\hat{O}(x) \subseteq \mathrm{R}(v)$.
\end{lemma}

\begin{proof}
By definition, we have $O(x) \subseteq \hat{O}(x) \subseteq \overline{O(x)}$.
Lemma~\ref{lem:orbit_cl} implies that assertions $(1)$ and $(2)$ hold.
If $\hat{O}(x) \subseteq \mathrm{R}(v)$, then $x \in \hat{O}(x) \subseteq \mathrm{R}(v)$.
Suppose that $x \in \mathrm{R}(v)$.
Assume that $\hat{O}(x) \not\subseteq \mathrm{R}(v)$.
Fix a point $y \in \hat{O}(x) \setminus \mathrm{R}(v)$.
Then $O(x) \neq O(y)$.
If $y \in \mathop{\mathrm{Cl}}(v)$, then $\overline{O(x)} = \overline{O(y)} = O(y) \subseteq \mathop{\mathrm{Cl}}(v)$, which contradicts $x \in \mathrm{R}(v)$.
Thus $y \in \mathrm{P}(v)$.
Since $\overline{O(x)} = \overline{O(y)}$, the difference $\overline{O(y)} - O(y) = \overline{O(x)} - O(y)$ is closed and contains $O(x)$, which contradicts that the closure $\overline{O(x)}$ is the minimal closed subset containing $O(x)$.
Thus  $\hat{O}(x) \subseteq \mathrm{R}(v)$.
\end{proof}

\begin{corollary}\label{cor:check_eq}
The following properties hold for any point $x \in X$:
\\
$(1)$ $x \in \mathop{\mathrm{Cl}}(v)$ if and only if $\check{O}(x) \subseteq \mathop{\mathrm{Cl}}(v)$.
\\
$(2)$ $x \in \mathrm{P}(v)$ if and only if $\check{O}(x) \subseteq \mathrm{P}(v)$.
\\
$(3)$ $x \in \mathrm{R}(v)$ if and only if $\check{O}(x) \subseteq \mathrm{R}(v)$.
\end{corollary}

\begin{proof}
Since $O(x) \subseteq \check{O}(x) \subseteq \hat{O}(x)$, Lemma~\ref{lem:hat02} implies assertions.
\end{proof}

We have the following statements.

\begin{lemma}\label{lem:aw2cl}
For any points $x,y \in \mathrm{R}(v)$ with $\alpha(x) = \alpha(y)$ and $\omega(x) = \omega(y)$, we obtain $\overline{O(x)} = \overline{O(y)}$.
%
\end{lemma}

\begin{proof}
By time reversion if necessary, we may assume that $x \in \alpha(x)$.
The invariance and closedness of $\alpha(x)$ imply that $\overline{O(x)} = \alpha(x)$.
If $y \in \alpha(y)$, then $\overline{O(y)} = \alpha(y)$ and so $\overline{O(x)} = \alpha(x) = \alpha(y) = \overline{O(y)}$.
Thus we may assume that $y \in \omega(y)$.
Then $\omega(x) \subseteq \overline{O(x)} = \alpha(x) = \alpha(y) \subseteq \overline{O(y)} = \omega(y) = \omega(x)$.
Therefore $\overline{O(x)} =\overline{O(y)}$.
\end{proof}

\begin{lemma}\label{lem:aw2cl02}
For any points $x \in \mathrm{R}(v)$, the abstract weak orbit $[x]$ is connected and  $[x] = \check{O} = \{ y \in \mathrm{R}(v) \mid  \alpha(x) = \alpha(y), \omega(x) = \omega(y) \}$.
\end{lemma}

\begin{proof}
Fix points $x, y \in \mathrm{R}(v)$.
Corollary~\ref{cor:check_eq} and Lemma~\ref{lem:aw2cl} imply $[x] = \check{O}(x) = \{ y \in \mathrm{R}(v) \mid \overline{O(x)} = \overline{O(y)}, \alpha(x) = \alpha(y), \omega(x) = \omega(y) \} = \{ y \in \mathrm{R}(v) \mid  \alpha(x) = \alpha(y), \omega(x) = \omega(y) \}$.
By $O(x) \subseteq [x] \subseteq \overline{O(x)}$, since any subset between a connected subset and its closure is connected, the abstract orbit class $[x]$ is connected.
\end{proof}

%

\subsection{Properties of abstract weak orbits and abstract orbits of flows}

By definitions of abstract weak orbit space and abstract orbit space, Lemma~\ref{lem:orbit_cl} implies the following reductions as in Figure~\ref{Fig:reductions}.

\begin{proposition}\label{prop:quotient}
Let $v$ be a flow on a Hausdorff space $X$.
Then the abstract weak orbit space $X/[v]$ is a quotient space of the weak orbit class space $X/\check{v}$, and the abstract orbit space $X/\langle v \rangle$ is a quotient space of the orbit class space $X/\hat{v}$ and of the abstract weak orbit space $X/[v]$.
\end{proposition}

\begin{figure}
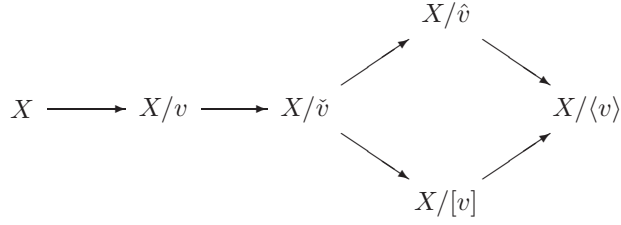

\[
\begin{diagram}
\node[4]{X/\hat{v}} \arrow[1]{se}\\
\node[1]{X} \arrow[1]{e} \node[1]{X/v} \arrow[1]{e} \node[1]{X/\check{v}} \arrow[1]{se} \arrow[1]{ne} \node[2]{X/\langle v \rangle}\\
 \node[4]{X/[v]} \arrow[1]{ne}
\end{diagram}
\]
\caption{Reductions of a Hausdorff space into the abstract orbit space.}
\label{Fig:reductions}
\end{figure}

We have the following statements.

\begin{lemma}\label{lem:awo}
The abstract orbit $\langle x \rangle = [x]$ of a singular {\rm (resp.} periodic {\rm)} point $x$ is the connected component of $\Sv$ {\rm (resp.} $\Pv$ {\rm)} containing $x$.
\end{lemma}

\begin{proof}
Fix a point $x \in \Cv$.
Then $\langle x \rangle = [x] = [x]''$ and $\alpha' (x) = \omega' (x) = \emptyset$.
If $x \in \Sv$, then the abstract orbit $\langle x \rangle$ is the connected component of the subset $\{ y \in \Sv \mid \alpha' (x) = \alpha' (y), \omega'(x) = \omega' (y) \} = \{ y \in \Sv \mid \alpha' (y) = \emptyset = \omega' (y) \} = \Sv$ containing $x$.
If $x \in \Pv$, then the abstract orbit $\langle x \rangle$ is the connected component of the subset $\{ y \in \Pv \mid \alpha' (x) = \alpha' (y), \omega'(x) = \omega' (y) \} = \{ y \in \Pv \mid \alpha' (y) = \emptyset = \omega' (y) \} = \Pv$ containing $x$.
\end{proof}

\begin{lemma}\label{lem:ch_01}
The following statement holds:
\\
$(1)$ For any point $x \in \mathrm{P}(v)$, the abstract orbit $\langle x \rangle = [x]$ is the connected component of $\{ y \in \mathrm{P}(v) \mid \alpha(x) = \alpha(y), \omega(x) = \omega(y) \}$ containing $x$, and we have that $\alpha(x) = \alpha'(x)$ and $\omega(x) = \omega'(x)$.
\\
$(2)$ For any point $x \in \mathrm{R}(v)$, we have $[x] \subseteq \langle x \rangle \subseteq (\alpha(x) \cup \omega(x)) \cap \mathrm{R}(v)$.
\end{lemma}

\begin{proof}
Suppose that $x \in \mathrm{P}(v)$.
Then $x \notin \alpha(x) \cup \omega(x)$.
The closedness and invariance of $\alpha(x)$ and $\omega(x)$ imply that $O(x) \cap (\alpha(x) \cup \omega(x)) = \emptyset$.
Therefore $\alpha(x) = \alpha'(x)$ and $\omega(x) = \omega'(x)$.
By definition of (abstract) weak orbit, we obtain $\langle x \rangle = [x] = [x]'' \subseteq \{ y \in \mathrm{P}(v) \mid \alpha' (x) = \alpha' (y), \omega'(x) = \omega' (y) \} = \{ y \in \mathrm{P}(v) \mid \alpha(x) = \alpha(y), \omega(x) = \omega(y) \}$.
Suppose that $x \in \mathrm{R}(v)$.
Then $x \in \alpha(x) \cup \omega(x)$.
The closedness and invariance of $\alpha(x)$ and $\omega(x)$ imply that $\check{O}(x) \subseteq \hat{O}(x) \subseteq \overline{O(x)} = \alpha(x) \cup \omega(x)$.
\end{proof}

The previous lemmas imply the following explicit characterizations of the abstract weak orbit and the abstract orbit of a flow on a Hausdorff space.

\begin{corollary}\label{cor:ch}
The following statement holds for a point $x$ in a Hausdorff space $X$:
\[
  [x] = \begin{cases}
    \text{the connected component of } \Sv \text{ containing } x & \text{if } x \in \mathop{\mathrm{Sing}}(v)  \\
    \text{the connected component of } \Pv \text{ containing } x & \text{if } x \in \mathop{\mathrm{Per}}(v)  \\
    \text{the connected component of } \\
     \{ y \in \mathrm{P}(v) \mid \alpha(x) = \alpha(y), \omega(x) = \omega(y) \}
     \text{ containing } x& \text{if } x \in \mathrm{P}(v) \\
    \{ y \in \mathrm{R}(v) \mid \alpha(x) = \alpha(y), \omega(x) = \omega(y) \} & \text{if } x \in \mathrm{R}(v)
  \end{cases}
\]
\end{corollary}

\begin{corollary}\label{cor:ch01}
The following statement holds for a point $x$ in a Hausdorff space $X$:
\[
  \langle x \rangle = \begin{cases}
    \text{the connected component of } \Sv \text{ containing } x & \text{if } x \in \mathop{\mathrm{Sing}}(v)  \\
    \text{the connected component of } \Pv \text{ containing } x & \text{if } x \in \mathop{\mathrm{Per}}(v)  \\
    \text{the connected component of } \\
     \{ y \in \mathrm{P}(v) \mid \alpha(x) = \alpha(y), \omega(x) = \omega(y) \}
     \text{ containing } x& \text{if } x \in \mathrm{P}(v) \\
    \{ y \in \mathrm{R}(v) \mid \overline{O(x)} = \overline{O(y)} \} & \text{if } x \in \mathrm{R}(v)
  \end{cases}
\]
\end{corollary}

We have the following properties of abstract (weak) orbits.

\begin{lemma}\label{lem:ch_02}
For any points $x, y \in X$, we have $[x] = [y]$ or $[x] \cap [y] = \emptyset$.
\end{lemma}

\begin{proof}
By definition of abstract weak orbit space, any abstract weak orbit space is contained in either $\Sv$, $\Pv$, $\mathrm{P}(v)$, or $\mathrm{R}(v)$.
Fix points $x, y \in X$ with $[x] \cap [y] \neq \emptyset$.
If $x \in \Sv$, then both $[x]$ and $[y]$ are connected component of $\Sv$ and so $[x] = [y]$.
If $x \in \Pv$, then both $[x]$ and $[y]$ are connected component of $\Pv$ and so $[x] = [y]$.
Suppose that $x \in \mathrm{P}(v)$.
Fix a point $z \in [x] \cap [y]$.
Lemma~\ref{lem:ch_01} implies that $\alpha(x) = \alpha(y) = \alpha(z)$ and $\omega(x) = \omega(y) = \omega(z)$.
This means that $[x]$ and $[y]$ are the connected component of $\{ w \in \mathrm{P}(v) \mid \alpha(w) = \alpha(z), \omega(w) = \omega(z) \}$ containing $z$ and so $[x] = [y]$.
Thus we may assume that $x \in \mathrm{R}(v)$.
Fix a point $z \in [x] \cap [y]$.
By definition of weak orbit class, we obtain that $\alpha(x) = \alpha(y) = \alpha(z)$  and $\omega(x) = \omega(y) = \omega(z)$.
Then $[x] = \check{O}(x) = \{ w \in X \mid \alpha(w) = \alpha(z), \omega(w) = \omega(z) \} = \check{O}(y) = [y]$.
\end{proof}

\begin{lemma}\label{lem:ch_03}
For any points $x, y \in X$, we have $\langle x \rangle = \langle y \rangle$ or $\langle x \rangle \cap \langle y \rangle = \emptyset$.
\end{lemma}

\begin{proof}
By definition of abstract orbit space, any abstract weak orbit space is contained in either $\Sv$, $\Pv$, $\mathrm{P}(v)$, or $\mathrm{R}(v)$.
Fix points $x, y \in X$ with $\langle x \rangle \cap \langle y \rangle \neq \emptyset$.
If $x \notin \mathrm{R}(v)$, then the assertion holds because $[x] = \langle x \rangle$ and $[y] = \langle y \rangle$.
Thus we may assume that $x \in \mathrm{R}(v)$.
Fix a point $z \in \langle x \rangle \cap \langle y \rangle$.
By definition of orbit class, we obtain that $\overline{O(x)} = \overline{O(y)} = \overline{O(z)}$.
Then $\langle x \rangle = \hat{O}(x) = \{ w \in X \mid \overline{O(w)} = \overline{O(z)} \} = \hat{O}(y) = \langle y \rangle$.
\end{proof}

\begin{lemma}\label{lem:closure_01}
The following properties hold for a point $x \in X$: \\
$(1)$ If $x \in \Cv$, then $\mathop{\downarrow}_{\leq_v} \langle x \rangle = \mathop{\downarrow}_{\leq_v} [x] = \bigcup_{y \in \langle x \rangle} \overline{O(y)} = \bigcup_{y \in [x]} \overline{O(y)} = [x] = \langle x \rangle$.
 \\
 $(2)$ If $x \in \mathrm{P}(v)$, then $\mathop{\downarrow}_{\leq_v} \langle x \rangle = \mathop{\downarrow}_{\leq_v} [x] = \bigcup_{y \in \langle x \rangle} \overline{O(y)} = \bigcup_{y \in [x]} \overline{O(y)} = (\alpha(x) \cup \omega(x)) \sqcup [x] = (\alpha(x) \cup \omega(x)) \sqcup \langle x \rangle = \overline{O(x)} \cup [x]  = \overline{O(x)} \cup \langle x \rangle$.
 \\
 $(3)$ If $x \in \mathrm{R}(v)$, then $\mathop{\downarrow}_{\leq_v} \langle x \rangle = \mathop{\downarrow}_{\leq_v} [x] = \bigcup_{y \in \langle x \rangle} \overline{O(y)} = \bigcup_{y \in [x]} \overline{O(y)} = \overline{O(x)} = \overline{[x]}  = \overline{\langle x \rangle} = \alpha(x) \cup \omega(x)$.
 \\
 In any case, we have $\mathop{\downarrow}_{\leq_v} \langle x \rangle = \mathop{\downarrow}_{\leq_v} [x] = \bigcup_{y \in \langle x \rangle} \overline{O(y)} = \bigcup_{y \in [x]} \overline{O(y)} = \alpha(x) \cup \omega(x) \cup [x] = \alpha(x) \cup \omega(x) \cup \langle x \rangle = \overline{O(x)} \cup [x] = \overline{O(x)} \cup \langle x \rangle$.
\end{lemma}

\begin{proof}
Suppose that $x \in \Sv$.
Lemma~\ref{lem:awo} implies that the abstract orbit $\langle x \rangle = [x]$ is the connected component of $\Sv$ containing $x$.
This means that $\mathop{\downarrow}_{\leq_v} [x] = \bigcup_{y \in [x]} \overline{O(y)} = \bigcup_{y \in [x]} O(y) = [x]$.
Suppose that $x \in \Pv$.
Lemma~\ref{lem:awo} implies that the abstract orbit $\langle x \rangle = [x]$ is the connected component of $\Pv$ containing $x$.
This means that $\mathop{\downarrow}_{\leq_v} [x] = \bigcup_{y \in [x]} \overline{O(y)} = \bigcup_{y \in [x]} O(y) = [x]$.
Suppose that $x \in \mathrm{P}(v)$.
Lemma~\ref{lem:ch_01} implies that $\mathop{\downarrow}_{\leq_v} [x] = \bigcup_{y \in [x]} \overline{O(y)} = \bigcup_{y \in [x]} ((\alpha(y) \cup \omega(y)) \sqcup O(y) )= (\bigcup_{y \in [x]} (\alpha(y) \cup \omega(y))) \sqcup [x] = (\alpha(x) \cup \omega(x)) \sqcup [x] = \overline{O(x)} \cup [x]$.
Suppose that $x \in \mathrm{R}(v)$.
Then $x \in \alpha(x) \cup \omega(x)$.
The closedness and invariance of $\alpha(x)$ and $\omega(x)$ imply that $\overline{O(x)} = \alpha(x) \cup \omega(x)$.
By definition of abstract (weak) orbit, we have
$\overline{[x]} = \overline{\check{O}(x)} = \overline{O(x)} =  \overline{\hat{O}(x)} = \overline{\langle x \rangle}$.
Therefore $\mathop{\downarrow}_{\leq_v} [x] = \bigcup_{y \in [x]} \overline{O(y)} = \overline{O(x)} = \bigcup_{y \in \langle x \rangle} \overline{O(y)} = \mathop{\downarrow}_{\leq_v} \langle x \rangle$.
This implies that the assertion $(3)$ holds.
\end{proof}

Recall that $\delta A =  A -  \mathrm{int} A$, $\partial_\perp A = \left( \bigcup_{z \in A} \overline{O(z)} \right) - A$, $\partial_\pitchfork A = \overline{A} - \bigcup_{z \in A} \overline{O(z)}$, and $\bp A = \partial_\perp A \sqcup \partial_\pitchfork A = \overline{A} - A$ for any subset $A$ of a Hausdorff space.
We have the following characterizations of coborders and vertical boundaries of abstract (weak) orbits of points.

\begin{corollary}\label{cor:closure_01}
The following properties hold for a point $x \in X$: \\
$(1)$ If $x \in \Sv$, then $\partial_+ \langle x \rangle = \partial_+ [x] = \emptyset$ {\rm(i.e.} $\partial \langle x \rangle = \partial [x] = \delta \langle x \rangle = \delta [x]$ {\rm)}.
 \\
$(2)$ If $x \in \Pv$, then $\partial_\perp \langle x \rangle = \partial_\perp [x] = \emptyset$ {\rm(i.e.} $\partial_+ \langle x \rangle = \partial_+ [x] = \partial_\pitchfork \langle x \rangle = \partial_\pitchfork [x]$ {\rm)}.
 \\
 $(3)$ If $x \in \mathrm{P}(v)$, then $\partial_\perp \langle x \rangle = \partial_\perp [x] = \alpha(x) \cup \omega(x) = \alpha'(x) \cup \omega'(x)$.
\\
$(4)$ If $x \in \mathrm{R}(v)$, then $\partial_\perp \langle x \rangle = \partial_+ \langle x \rangle = (\alpha(x) \cup \omega(x)) - \langle x \rangle$ and $\partial_\perp [x] = \partial_+ [x] = (\alpha(x) \cup \omega(x)) - [x]$.
In particular, $\partial_\pitchfork \langle x \rangle = \partial_\pitchfork [x] = \emptyset$ for any $x \in \mathrm{R}(v)$.
\end{corollary}

\begin{proof}
Suppose that $x \in \Sv$.
Then $[x] = \langle x \rangle$ is the connected component of $\Sv$ containing $x$ and so closed.
Therefore $\partial_+ [x] = \overline{[x]} - [x] = \emptyset$.
Suppose that $x \in \Pv$.
Then $[x] = \langle x \rangle$ is the connected component of $\Pv$ containing $x$.
Since any periodic orbits are closed, we have that $\partial_\perp [x] = \left( \bigcup_{z \in [x]} \overline{O(z)} \right) - [x] = [x] - [x] = \emptyset$.
Suppose that $x \in \mathrm{P}(v)$.
Lemma~\ref{lem:ch_01} implies that $\alpha(x) \cup \omega(x) = \alpha'(x) \cup \omega'(x)$, and that $[x] = \langle x \rangle$ is the connected component of $\{ y \in \mathrm{P}(v) \mid \alpha(x) = \alpha(y), \omega(x) = \omega(y) \}$ containing $x$.
We claim that $(\alpha(x) \cup \omega(x)) \cap [x] = \emptyset$.
Indeed, assume that $(\alpha(x) \cup \omega(x)) \cap [x] \neq \emptyset$.
Then there is a point $y \in (\alpha(x) \cup \omega(x)) \cap [x]$.
By definition of abstract weak orbit, Lemma~\ref{lem:ch_01} implies that $\alpha(y) = \alpha'(y) = \alpha'(x) = \alpha(x)$ and $\omega(y) = \omega'(y) = \omega'(x) = \omega(x)$.
Since $y \in \mathrm{P}(v)$, we obtain that $y \notin \alpha'(y) \cup \omega'(y) = \alpha'(x) \cup \omega'(x) = \alpha(x) \cup \omega(x)$, which contradicts $y \in \alpha(x) \cup \omega(x)$.
Therefore $\partial_\perp [x] = \left( \bigcup_{y \in [x]} \alpha(y) \cup \omega(y) \right) \setminus [x] = (\alpha(x) \cup \omega(x)) \setminus [x] = \alpha(x) \cup \omega(x) = \alpha'(x) \cup \omega'(x)$.
Suppose that $x \in \mathrm{R}(v)$.
Lemma~\ref{lem:closure_01} implies that $\overline{O(x)} = \overline{[x]} = \overline{\langle x \rangle} = \alpha(x) \cup \omega(x)$.
The definitions of abstract weak orbit and abstract orbit imply the assertion $(4)$.
\end{proof}

\subsection{Properties of binary relations}

We have the following equivalences.

\begin{lemma}\label{lem:equiv}
The following properties hold for any points $x, y \in X$:
\\
$(1)$ $[x] \leq_v [y]$ if and only if $[x] \cap \overline{O(y)} \neq \emptyset$.
\\
$(2)$ $[x] \leq_\alpha [y]$ if and only if $[x] = [y]$ or $[x] \cap \alpha(y) \neq \emptyset$.
\\
$(3)$ $[x] \leq_\omega [y]$ if and only if $[x] = [y]$ or $[x] \cap \omega(y) \neq \emptyset$.
\end{lemma}

\begin{proof}
If $[x] = [y]$, then $[x] \cap \overline{O(y)} \neq \emptyset$, $[x] \leq_v [y]$, $[x] \leq_\alpha [y]$, and $[x] \leq_\omega [y]$.
Thus we may assume that $[x] \neq [y]$.
If $[x] \cap \overline{O(y)} \neq \emptyset$ (resp. $[x] \cap \alpha(y) \neq \emptyset$, $[x] \cap \omega(y) \neq \emptyset$), then there is a point $x_1 \in [x]$ such that $\overline{O(x_1)} \subseteq \overline{O(y)}$ (resp. $x_1 \in \alpha (y)$, $x_1 \in \omega (y)$) and so that $[x] \leq_v [y]$ (resp. $[x] \leq_\alpha [y]$, $[x] \leq_\omega [y]$).
Conversely, suppose that $[x] \leq_v [y]$.
Then there are points  $x_1 \in [x]$ and $y_1 \in [y]$ such that $\overline{O(x_1)} \subseteq \overline{O(y_1)}$.
Assume that $y \in \mathop{\mathrm{Cl}}(v)$.
Then the abstract weak orbit $[y]$ is a connected component of $\Sv$ or $\Pv$ containing $y$ and so $[x] = [y]$, which contradicts $[x] \neq [y]$.
Thus $y \notin \mathop{\mathrm{Cl}}(v)$.
Suppose that $y \in \mathrm{P}(v)$.
Since $[y] \subseteq \mathrm{P}(v)$, we obtain $[y] \subseteq \{ z \in \mathrm{P}(v) \mid \alpha(z) = \alpha(y), \omega(z) = \omega(y) \}$.
Then $x_1 \in \bigcup_{z \in [y]} \alpha(z) \cup \omega(z) = \alpha(y) \cup \omega(y) \subseteq \overline{O(y)}$ and so $x_1 \in [x] \cap \overline{O(y)}$.
Suppose that $y \in \mathrm{R}(v)$.
Then $[y] = \check{O}(y) = \{ z \in X \mid \overline{O(z)} = \overline{O(y)}, \alpha(z) = \alpha(y), \omega(z) = \omega(y) \}$ and so $\overline{O(x_1)} \subseteq \overline{O(y_1)} = \overline{O(y)}$.
This means that $x_1 \in [x] \cap \overline{O(y)}$.
Suppose that $[x] \leq_{\alpha} [y]$.
Then there are points  $x_1 \in [x]$ and $y_1 \in [y]$ such that $x_1 \in \alpha(y_1)$.
As above, if $y \in \mathop{\mathrm{Cl}}(v)$, then the abstract weak orbit $[y]$ is a connected component of $\Sv$ or $\Pv$ containing $y$ and so $[x] = [y]$, which contradicts $[x] \neq [y]$.
Thus $y \notin \mathop{\mathrm{Cl}}(v)$.
Assume that $y \in \mathrm{P}(v)$.
Since $[y] \subseteq \mathrm{P}(v)$, we obtain $[y] \subseteq \{ z \in \mathrm{P}(v) \mid \alpha(z) = \alpha(y), \omega(z) = \omega(y) \}$.
Then $x_1 \in \bigcup_{z \in [y]} \alpha(z) = \alpha(y)$ and so $x_1 \in [x] \cap \alpha(y)$.
Suppose that $y \in \mathrm{R}(v)$.
Then $[y] = \check{O}(y) = \{ z \in \mathrm{R}(v) \mid \alpha(z) = \alpha(y), \omega(z) = \omega(y) \}$ and so $x_1 \in \bigcup_{z \in [y]} \alpha(z) = \alpha(y)$.
This means $x_1 \in [x] \cap \alpha(y)$.
Symmetry implies that the condition $[x] \leq_\omega [y]$ implies $[x] = [y]$ or $[x] \cap \omega (y) \neq \emptyset$.
\end{proof}

\begin{lemma}\label{lem:equiv02}
For any points $x, y \in X$, we have that $\langle x \rangle \leq_v \langle y \rangle$ if and only if $\langle x \rangle \cap \overline{O(y)} \neq \emptyset$.
\end{lemma}

\begin{proof}
If $\langle x \rangle = \langle y \rangle$, then $\langle x \rangle \cap \overline{O(y)} \neq \emptyset$ and $\langle x \rangle \leq_v \langle y \rangle$.
Thus we may assume that $\langle x \rangle \neq \langle y \rangle$.
If $\langle x \rangle \cap \overline{O(y)} \neq \emptyset$, then there is a point $x_1 \in \langle x \rangle$ such that $\overline{O(x_1)} \subseteq \overline{O(y)}$ and so that $\langle x \rangle \leq_v \langle y \rangle$.
Conversely, suppose that $\langle x \rangle \leq_v \langle y \rangle$.
Then $\overline{O(x_1)} \subseteq \overline{O(y_1)}$ for some $x_1 \in \langle x \rangle$ and $y_1 \in \langle y \rangle$.
Suppose that $x,y \notin \mathrm{R}(v)$.
Then $\langle x \rangle = [x]$ and $\langle y \rangle = [y]$.
Lemma~\ref{lem:equiv} implies that $\langle x \rangle \cap \overline{O(y)} \neq \emptyset$.
Suppose that $y \in \mathrm{R}(v)$.
Then $\langle y \rangle = \hat{O}(y) = \{ z \in X \mid \overline{O(z)} = \overline{O(y)} \}$ and so $x_1 \in \bigcup_{z \in \langle y \rangle} \overline{O(z)} = \overline{O(y)}$.
This implies that $x_1 \in \langle x \rangle \cap \overline{O(y)}$.
Thus we may assume that $y \notin \mathrm{R}(v)$ and $x \in \mathrm{R}(v)$.
Then $\langle y \rangle = [y]''$.
Lemma~\ref{lem:type} implies that either $[y]'' \subseteq \Cv$ or $[y]'' \subseteq \mathrm{P}(v)$.
Since $\overline{O(x_1)} \subseteq \overline{O(y_1)}$, we have $[y]'' \subseteq \mathrm{P}(v)$.
Then $\alpha(y_1) = \alpha'(y_1) = \alpha'(y) = \alpha(y)$ and $\omega(y_1) = \omega'(y_1) = \omega'(y) = \omega(y)$.
Therefore $x_1 \in \overline{O(y_1)} - O(y_1) = \alpha(y_1) \cup \omega(y_1) = \alpha(y) \cup \omega(y) = \overline{O(y)} - O(y)$ and so $x_1 \in \langle x \rangle \cap \overline{O(y)}$.
\end{proof}

We show that the pre-order relation $\leq_v$ on a Hausdorff space $X$ is the union of pre-orders $\leq_{\alpha}$ and $\leq_\omega$ on $X$ as a direct product.

\begin{lemma}\label{lem:hat}
The following statements are equivalent for any points $x, y \in X$:
\\
$(1)$ $x \leq_v y$.
\\
$(2)$ Either $x \leq_{\alpha} y$ or $x \leq_{\omega} y$.
\end{lemma}

\begin{proof}
Notice that $\overline{O} = \alpha(O) \cup O \cup \omega(O)$ for any orbit $O$.
If $O(x) = O(y)$, then $x \leq_v y$, $x \leq_{\alpha} y$, and $x \leq_{\omega} y$.
Thus we may assume that $O(x) \neq O(y)$.
Suppose that $x \leq_v y$.
Then $O(x) \subseteq \overline{O(y)}$.
Since $O(x) \neq O(y)$, we have $O(x) \subseteq \overline{O(y)} - O(y) \subseteq  \alpha(y) \cup \omega(y)$.
Then either $x \in \alpha(y)$ or $x \in \omega(y)$.
This means that either $x \leq_{\alpha} y$ or $x \leq_{\omega} y$.
Conversely, suppose that $x \leq_{\omega} y$.
Then $x \in \omega(y) \subseteq \overline{O(y)}$ and so $\overline{O(x)} \subseteq \overline{O(y)}$.
Then $x \leq_v y$.
By symmetry, the condition $x \leq_{\alpha} y$ implies $x \leq_v y$.
\end{proof}

We have the following equivalence.

\begin{lemma}\label{lem:v_perp}
The following statements are equivalent for any points $x, y \in X$:
\\
$(1)$ $[x] \leq_v [y]$.
\\
$(2)$ Either $[x] \leq_{\alpha} [y]$ or $[x] \leq_{\omega} [y]$.
\\
$(3)$ $[x] \leq_\perp [y]$.
\\
$(4)$ $[x] \cap \overline{O(y)} \neq \emptyset$.
\\
$(5)$
Either $[x] = [y]$ or $[x] \cap ((\alpha(y) \cup \omega(y)) \setminus [y]) \neq \emptyset$.
\end{lemma}

\begin{proof}
Lemma~\ref{lem:hat} implies that the conditions $(1)$ and $(2)$ are equivalent.
By Lemma~\ref{lem:equiv}, the conditions $(1)$ and $(4)$ are equivalent.
If $[x] = [y]$, then $[x] \leq_v [y]$ and $[x] \leq_\perp [y]$.
Thus we may assume that $[x] \neq [y]$.
Lemma~\ref{lem:ch_02} implies $[x] \cap [y] = \emptyset$.
Therefore $[x] \leq_v [y]$ if and only if $[x] \cap ((\bigcup_{z \in [y]} \overline{O(z)}) - [y]) \neq \emptyset$.
Moreover, $[x] \leq_\perp [y]$ if and only if $[x] \cap (\bigcup_{z \in [y]} (\alpha(z) \cup \omega(z)) \setminus [y]) \neq \emptyset$.
Lemma~\ref{lem:closure_01} implies that $ \bigcup_{z \in [y]} (\alpha(z) \cup \omega(z)) \setminus [y] = (\bigcup_{z \in [y]} \overline{O(z)}) - [y]  = \overline{O(y)} \setminus [y] = (\alpha(y) \cup \omega(y)) \setminus [y]$.
This means that conditions $(1)$, $(3)$, and $(5)$ are equivalent.
%
\end{proof}

Similarly, we have the following statement.

\begin{lemma}\label{lem:v_perp_02}
The following statements are equivalent for any points $x, y \in X$:
\\
$(1)$ $\langle x \rangle \leq_v \langle y \rangle$.
\\
$(2)$ Either $\langle x \rangle \leq_{\alpha} \langle y \rangle$ or $\langle x \rangle \leq_{\omega} \langle y \rangle$.
\\
$(3)$ $\langle x \rangle \leq_\perp \langle y \rangle$.
\\
$(4)$ $\langle x \rangle \cap \overline{O(y)} \neq \emptyset$.
\\
$(5)$
Either $\langle x \rangle = \langle y \rangle$ or $\langle x \rangle \cap ((\alpha(x) \cup \omega(x)) \setminus \langle y \rangle) \neq \emptyset$.
\end{lemma}

\begin{proof}
Lemma~\ref{lem:hat} implies that the conditions $(1)$ and $(2)$ are equivalent.
By Lemma~\ref{lem:equiv02}, the conditions $(1)$ and $(4)$ are equivalent.
If $\langle x \rangle  = \langle y \rangle $, then $\langle x \rangle  \leq_v \langle y \rangle $ and $\langle x \rangle  \leq_\perp \langle y \rangle $.
Thus we may assume that $\langle x \rangle  \neq \langle y \rangle $.
Then $\langle x \rangle  \cap \langle y \rangle  = \emptyset$.
Therefore $\langle x \rangle  \leq_v \langle y \rangle $ if and only if $\langle x \rangle  \cap ((\bigcup_{z \in \langle y \rangle } \overline{O(z)}) - \langle y \rangle ) \neq \emptyset$.
Moreover, $\langle x \rangle  \leq_\perp \langle y \rangle $ if and only if $\langle x \rangle  \cap (\bigcup_{z \in \langle y \rangle } (\alpha(z) \cup \omega(z)) \setminus \langle y \rangle ) \neq \emptyset$.
Lemma~\ref{lem:closure_01} implies that $ \bigcup_{z \in \langle y \rangle } (\alpha(z) \cup \omega(z)) \setminus \langle y \rangle  = (\bigcup_{z \in \langle y \rangle } \overline{O(z)}) - \langle y \rangle   = \overline{O(y)} \setminus \langle y \rangle  = (\alpha(y) \cup \omega(y)) \setminus \langle y \rangle $.
This means that conditions $(1)$, $(3)$, and $(5)$ are equivalent.
\end{proof}

By definitions, Lemma~\ref{lem:decomp_limit}, Lemma~\ref{lem:v_perp}, and Lemma~\ref{lem:v_perp_02} imply the following observations.

\begin{lemma}
The following properties hold for points $x,y \in X$: \\
$(1)$ $[x] <_\pitchfork [y]$ if and only if $[x] \cap (\overline{[y]} - \bigcup_{z \in [y]} \overline{O(z)}) \neq \emptyset$ \, {\rm (i.e.} $[x] \cap \partial_\pitchfork \overline{[y]} \neq \emptyset$ {\rm)}.
\\
$(2)$ $\langle x \rangle <_\pitchfork \langle y \rangle$ if and only if $\langle x \rangle \cap (\overline{\langle y \rangle} - \bigcup_{z \in \langle y \rangle} \overline{O(z)}) \neq \emptyset$ \, {\rm (i.e.} $\langle x \rangle \cap \partial_\pitchfork \overline{\langle y \rangle} \neq \emptyset$ {\rm)}.
\end{lemma}

\begin{proof}
Note that $\bigcup_{z \in [w]} (\alpha(z) \cup \omega(z)) \cup [w] = \bigcup_{z \in [w]} (\alpha(z) \cup O(z) \cup  \omega(z)) = \bigcup_{z \in [w]} \overline{O(z)}$ and $\bigcup_{z \in \langle w \rangle} (\alpha(z) \cup \omega(z)) \cup \langle w \rangle = \bigcup_{z \in \langle w \rangle} (\alpha(z) \cup O(z) \cup  \omega(z)) = \bigcup_{z \in \langle w \rangle} \overline{O(z)}$  for any $w \in X$.
Therefore $[x] <_\pitchfork [y]$ if and only if
there is a point $x_1 \in [x]$ such that $x_1 \in \overline{[y]} - \left( \bigcup_{z \in [y]} (\alpha(z) \cup \omega(z)) \cup [y] \right)$.
Such conditions are equivalent to $\emptyset \neq [x] \cap \left( \overline{[y]} - \left( \bigcup_{z \in [y]} (\alpha(z) \cup \omega(z)) \cup [y] \right) \right) = [x] \cap (\overline{[y]} - \bigcup_{z \in [y]} \overline{O(z)}) = [x] \cap \partial_\pitchfork \overline{[y]}$.
This implies the assertion $(1)$ holds.
Similarly, we have that $\langle x \rangle <_\pitchfork \langle y \rangle$ if and only if $\emptyset \neq \langle x \rangle \cap \left( \overline{\langle y \rangle} - \left( \bigcup_{z \in \langle y \rangle} (\alpha(z) \cup \omega(z)) \cup \langle y \rangle \right) \right) = \langle x \rangle \cap (\overline{\langle y \rangle} - \bigcup_{z \in \langle y \rangle} \overline{O(z)}) = \langle x \rangle \cap \partial_\pitchfork \overline{\langle y \rangle}$.
\end{proof}

\begin{lemma}\label{lem:equivalent}
The following conditions are equivalent for points $x,y \in X$: \\
$(1)$ $[x] \leq_\partial [y]$.
\\
$(2)$ $[x] \cap \overline{[y]} \neq \emptyset$.
\\
$(3)$ Either $[x] \leq_\perp [y]$ or $[x] \leq_\pitchfork [y]$.
\\
$(4)$ Either $[x] \leq_v [y]$ or $[x] \leq_\pitchfork [y]$.
\\
$(5)$ Either $[x] \leq_\alpha [y]$, $[x] \leq_\omega [y]$, or $[x] \leq_\pitchfork [y]$.
\end{lemma}

\begin{proof}
By definition of binary relation $\leq_\partial$, the conditions $(1)$ and $(2)$ are equivalent.
By definitions of binary relations $\leq_\partial$, $\leq_\perp$, and $\leq_\pitchfork$, the conditions $(1)$ and $(3)$ are equivalent.
Lemma~\ref{lem:v_perp} implies that conditions $(3)$ and $(4)$ are equivalent.
Lemma~\ref{lem:hat} implies that conditions $(4)$ and $(5)$ are equivalent.
\end{proof}

\begin{lemma}\label{lem:equivalent02}
The following conditions are equivalent for points $x,y \in X$: \\
$(1)$ $\langle x \rangle \leq_\partial \langle y \rangle$.
\\
$(2)$ $\langle x \rangle \cap \overline{\langle y \rangle} \neq \emptyset$.
\\
$(3)$ Either $\langle x \rangle \leq_\perp \langle y \rangle$ or $\langle x \rangle \leq_\pitchfork \langle y \rangle$.
\\
$(4)$ Either $\langle x \rangle \leq_v \langle y \rangle$ or $\langle x \rangle \leq_\pitchfork \langle y \rangle$.
\\
$(5)$ Either $\langle x \rangle \leq_\alpha \langle y \rangle$, $\langle x \rangle \leq_\omega \langle y \rangle$, or $\langle x \rangle \leq_\pitchfork \langle y \rangle$.
\end{lemma}

\begin{proof}
By definition of binary relation $\leq_\partial$, the conditions $(1)$ and $(2)$ are equivalent.
By definitions of binary relations $\leq_\partial$, $\leq_\perp$, and $\leq_\pitchfork$, the conditions $(1)$ and $(3)$ are equivalent.
Lemma~\ref{lem:v_perp_02} implies that conditions $(3)$ and $(4)$ are equivalent.
Lemma~\ref{lem:hat} implies that conditions $(4)$ and $(5)$ are equivalent.
\end{proof}

Note that each of the binary relations $\leq_\partial$ on the abstract weak orbit spaces and on the abstract orbit spaces for flows on Hausdorff spaces is the union of $\leq_\alpha$, $\leq_\omega$, and $\leq_\pitchfork$ as a direct product.
We have the following statement.

\begin{lemma}\label{lem:trans_order}
Let
$\leq_{\tau_{[v]}}$ be the specialization order of the quotient topology $\tau_{[v]}$ on $X/[v]$.
For any $x, y \in X$, the order relation $[x] \leq_\partial [y]$ implies both $\langle x \rangle \leq_\partial \langle y \rangle$ and $[x] \leq_{\tau_{[v]}} [y]$.
\end{lemma}

\begin{proof}
Let $\pi_{[v]} \colon X \to X/[v]$ be the quotient map.
Fix any classes $[x] \neq [y] \in X/[v]$.
Suppose that $[x] \leq_\partial [y]$.
Then there are points $x_1 \in [x]$ and $y_1 \in [y]$ such that $x_1 \leq_\partial y_1$.
Since $[x_1] \subseteq \langle x_1 \rangle$ and $[y_1] \subseteq \langle y_1 \rangle$, we have $\langle x \rangle \leq_\partial \langle y \rangle$.
By definition of $\leq_\partial $, we have $x_1 \in \overline{[y]}$.
This means that $[x] \cap \overline{[y]} \neq \emptyset$.
Since $\overline{[y]} \subseteq \pi_{[v]}^{-1}(\overline{[y]}^{\tau_{[v]}})$, we have $[x] \in \overline{[y]}^{\tau_{[v]}}$ in $X/[v]$ and so $[x] \leq_{\tau_{[v]}} [y]$.
\end{proof}

Notice that the order relation $[x] \leq_{\tau_{[v]}} [y]$ need not imply $[x] \leq_\partial [y]$ in general.
Indeed, let $v$ be a flow on a closed disk $\mathbb{D}$ as in the example in Figure~\ref{non_transitive}.
We show that $[d_3] \not\leq_\partial [d_2]$ and $[d_3] =_{\tau_{[v]}} [d_2]$.
Indeed, since $\emptyset = D_3 \cap \overline{D_2} = [d_3] \cap \overline{[d_2]}$, we have $[d_3] \not\leq_\partial [d_2]$.
On the other hand, by $\mathbb{D}/[v] = \overline{[d_2]}^{\tau_{[v]}} = \overline{[d_3]}^{\tau_{[v]}}$, we obtain $[d_3] =_{\tau_{[v]}} [d_2]$.

\subsubsection{Properties of pre-orders and orbit class}

The binary relation $\leq_v$ is a pre-order on the abstract weak orbit space $X/[v]$ and the abstract orbit space $X/\langle v \rangle$.

\begin{lemma}\label{lem:preorder}
The binary relation $\leq_v$ on the abstract weak orbit space $X/[v]$ of a flow $v$ on a Hausdorff space $X$ is a pre-order.
\end{lemma}

\begin{proof}
By definition of binary relation $\leq_v$, reflexivity holds for the relation $\leq_v$.
Therefore it suffices to show transitivity.
Fix points $x,y,z \in X$ with $[x] \leq_v [y]$ and $[y] \leq_v [x]$.
We may assume that $[x] \neq [y]$ and $[y] \neq [z]$.
Then $O(x) \neq O(y)$ and $O(y) \neq O(z)$, and there are points $x_1 \in [x]$, $y_1, y_2 \in [y]$ and $z_2 \in [z]$ such that $\overline{O(x_1)} \subseteq \overline{O(y_1)}$ and $\overline{O(y_2)} \subseteq \overline{O(z_2)}$.
By $O(x_1) \neq O(y_1)$ and $\overline{O(y_1)} = \alpha(y_1) \cup O(y_1) \cup \omega(y_1)$, we have $x_1 \in \overline{O(y_1)} - O(y_1) = \alpha'(y_1) \cup \omega'(y_1)$.
Since $\alpha'(y_1) = \alpha'(y_2)$ and $\omega'(y_1) = \omega'(y_2)$ if $y \notin \mathrm{R}(v)$, and since $\alpha(y_1) = \alpha(y_2)$ and $\omega(y_1) = \omega(y_2)$ if $y \in \mathrm{R}(v)$, we obtain $x_1 \in \alpha'(y_1) \cup \omega'(y_1) \subseteq \alpha(y_2) \cup \omega(y_2) \subseteq \overline{O(y_2)}$ and so $x_1 \in \overline{O(x_1)} \subseteq \overline{O(y_2)} \subseteq \overline{O(z_2)} \subseteq \overline{[z]}$.
By $x_1 \in [x]$, we have $[x] \cap \overline{[z]} \neq \emptyset$.
Lemma~\ref{lem:equivalent} implies that $[x] \leq_v [z]$.
\end{proof}

\begin{lemma}\label{lem:preorders02}
The binary relation $\leq_v$ on the abstract orbit space $X/\langle v \rangle$ of a  flow $v$ on a Hausdorff space $X$ is a partial order.
\end{lemma}

\begin{proof}
By definition of binary relation $\leq_v$, reflexivity holds for the relation $\leq_v$.
Suppose that $x,y \in X$ with $\langle x \rangle \leq_v \langle y \rangle$ and $\langle y \rangle \leq_v \langle x \rangle$.
Assume that $\langle x \rangle \neq \langle y \rangle$.
Then $O(x) \neq O(y)$.
Lemma~\ref{lem:equiv02} implies that $\langle x \rangle \cap \overline{O(y)} \neq \emptyset$ and $\langle y \rangle \cap \overline{O(x)} \neq \emptyset$.
Then $\overline{O(x)} = \overline{O(y)}$.
Since $O(x) \neq O(y)$, we have $\overline{O(x)} \subseteq \alpha(y) \cup \omega(y) \subseteq \overline{O(y)}$ and $\overline{O(y)} \subseteq \alpha(x) \cup \omega(x) \subseteq \overline{O(x)}$.
Therefore $\overline{O(x)} = \overline{O(y)} = \alpha(x) \cup \omega(x) = \alpha(y) \cup \omega(y)$.
This means that $O(x)$ and $O(y)$ are non-closed recurrent orbits (i.e. $O(x) \sqcup O(x) \subseteq \mathrm{R}(v)$).
Then $\langle x \rangle = \hat{O}(x) = \hat{O}(y) = \langle y \rangle$, which contradicts $\langle x \rangle \neq \langle y \rangle$.
Thus $\langle x \rangle = \langle y \rangle$ and so the relation $\leq_v$ is  antisymmetric.
Therefore it suffices to show transitivity.
Fix points $x,y,z \in X$ with $\langle x \rangle \leq_v \langle y \rangle$ and $\langle y \rangle \leq_v \langle x \rangle$.
We may assume that $\langle x \rangle \neq \langle y \rangle$ and $\langle y \rangle \neq \langle z \rangle$.
Then $O(x) \neq O(y)$ and $O(y) \neq O(z)$, and there are points $x_1 \in \langle x \rangle$ and $y_1 \in \langle y \rangle$ such that $\overline{O(x_1)} \subseteq \overline{O(y_1)}$ and there are points $y_2 \in \langle y \rangle$ and $z_2 \in \langle z \rangle$ such that $\overline{O(y_2)} \subseteq \overline{O(z_2)}$.
By $O(x_1) \neq O(y_1)$ and $\overline{O(y_1)} = \alpha(y_1) \cup O(y_1) \cup \omega(y_1)$, we have $x_1 \in \overline{O(y_1)} - O(y_1) = \alpha'(y_1) \cup \omega'(y_1)$.
Since $\alpha'(y_1) = \alpha'(y_2)$ and $\omega'(y_1) = \omega'(y_2)$ if $y \notin \mathrm{R}(v)$, and since $\overline{O(y_1)} = \overline{O(y_2)}$ if $y \in \mathrm{R}(v)$, we obtain $\alpha'(y_1) \cup \omega'(y_1) \subseteq \overline{O(y_2)}$ and so $x_1 \in \overline{O(x_1)} \subseteq \overline{O(y_2)} \subseteq \overline{O(z_2)} \subseteq \overline{\langle z \rangle}$.
By $x_1 \in \langle x \rangle$, we have $\langle x \rangle \cap \overline{\langle z \rangle} \neq \emptyset$.
Lemma~\ref{lem:equivalent02} implies that $\langle x \rangle \leq_v \langle z \rangle$.
\end{proof}

We have the following triviality of $\leq_\pitchfork$ on $X/\hat{v}$.

\begin{lemma}
The binary relation $\leq_\pitchfork$ on the orbit class space $X/\hat{v}$ is identical.
\end{lemma}

\begin{proof}
Fix any orbit classes $\hat{O}, \hat{O}'$ with $\hat{O} \leq_\pitchfork \hat{O}'$.
Suppose $\hat{O} \neq \hat{O}'$.
Then $\hat{O} \cap \hat{O}' = \emptyset$.
Since $\overline{O'} = \alpha(O') \cup O' \cup \omega(O')$, we have $\emptyset \neq \hat{O} \cap \left( \overline{\hat{O}'} - \left( \bigcup_{y \in \hat{O}'} (\alpha(y) \cup \omega(y)) \right) \right) = \hat{O} \cap \left( \overline{\hat{O}'} - \left( \bigcup_{y \in \hat{O}'} (\alpha(y) \cup O(y) \cup \omega(y)) \right) \right) = \hat{O} \cap \left( \overline{\hat{O}'} - \left( \bigcup_{y \in \hat{O}'} \overline{O(y)} \right) \right) = \hat{O} \cap \left( \overline{\hat{O}'} - \overline{O'} \right) = \hat{O} \cap \left( \overline{O'} - \overline{O'}  \right)= \emptyset$, which is a contradiction.
Thus $\hat{O} = \hat{O}'$ if and only if $\hat{O} \leq_\pitchfork \hat{O}'$.
\end{proof}

\section{Fundamental properties of flows on compact Hausdorff spaces}
From now on, we assume that the Hausdorff space $X$ is compact unless otherwise stated.

\subsection{Coincidence between closedness of orbits and one of the subsets}

We have the following equivalence of closedness.

\begin{lemma}\label{lem:closed}
An orbit of a flow on a compact Hausdorff space is a closed orbit if and only if it is a closed subset.
\end{lemma}

\begin{proof}
Let $O$ be an orbit of a flow $v$ on a compact Hausdorff space $X$ and $x \in O$ a point.
Then $X$ is a normal space.
Lemma~\ref{lem:kc_closed} implies that any closed orbit is a closed subset.
Conversely, suppose that $O$ is a closed subset.
Since $X$ is compact, the orbit $O$ is compact.
If $O$ is a closed orbit, then the assertion holds.
Thus we may assume that $O$ is not a closed orbit.
Suppose that $O$ is not recurrent.
Then $O \cap (\alpha(x) \cup \omega(x)) = \emptyset$.
Since $X$ is compact, the finite intersection property of $X$ implies that
the $\alpha$-limit set and the $\omega$-limit set of a point is nonempty.
This means that $O \subsetneq O \sqcup (\alpha(x) \cup \omega(x)) = \overline{O}$, which contradicts the closedness of $O$.
Thus $O$ is recurrent.
By time reversion if necessary, we may assume that $x \in \alpha(x)$.
The invariance of $\alpha(x)$ implies that $O \subseteq \alpha(x)$.
Put $T_1 := -1$.
The Hausdorff separation axiom implies that there are disjoint open \nbds $V_1$ and $U_1$ of $x$ and $v_{T_1}(x)$ respectively.
Since $x \in \alpha(x) \cap V_1$ and $v_1(x) \in \alpha(x) \cap U_1$, the connectivity of $v(\R_{<t}, x)$ for any $t \in \R$ implies that $v(\R_{<t}, x) \not\subseteq V_1 \sqcup U_1$ and so that there is an unbounded decreasing sequence $(t_{1,m})_{m \in \Z_{>0}}$ with $v_{t_{1,3m}}(x) \notin V_1 \sqcup U_1$, $v_{t_{1,3m+1}}(x) \in V_1$, and $v_{t_{1,3m+2}}(x) \in U_1$.
Then the difference $K_1 := O \setminus (V_1 \sqcup U_1)$ contains $\{ v_{t_{1,3m}}(x) \}_{m \in \Z_{>0}}$ and $x \notin K_1$.
The regularity of $X$ implies that there are disjoint open \nbds $V'_1$ and $W_1$ of  the point $x$ and the closed subset $K_1$ respectively.
Then $\overline{V_1'} \cap W_1 = \emptyset$.
The difference $L'_1 := O \setminus W_1$ is a closed subset such that $x \in V'_1 \cap O \subseteq L'_1$.
Since the closed interval $I_1 := [-1,1]$ is compact and so the image $v_x(I_1)$ is compact, Hausdorff separation axiom implies that $v_x(I_1)$ is closed.
The union $L_1 := L'_1 \cup v_x(I_1) \cup \overline{V_1'}$ is closed, $x \in L_1$, and there is a large integer $N>0$ such that $v_{t_{1,3m}}(x) \in K_1 \setminus v_x(I_1) \subseteq W_1 \setminus v_x(I_1) = W_1 \setminus L_1$ for any $m >N$.
Put $I_n := [-n,n]$.
By induction, we define an open cover $\{ V'_n \}$ of $O$ with $V'_n \subseteq V'_{n+1}$ for any $n \in \Z_{>0}$, and define $T_n, V_n, U_n, (t_{n,m})_{m \in \Z_{>0}}, K_n, W_n, L'_n$ , and $L_n$ for any $n \in \Z_{>0}$ as follows.
Suppose that $T_n, V_n, U_n, (t_{n,m})_{m \in \Z_{>0}}, K_n, W_n, V'_n, L'_n$, and  $L_n$ are defined.
Fix any $T_{n+1} \leq - (n+1)$ such that $v_{T_{n+1}}(x) \in W_n \setminus L_n = W_n \setminus v_x(I_n)$.
Since $v_{T_{n+1}}(x) \notin L_n$, the regularity of $X$ implies that  there are disjoint open \nbds $V_{n+1}$ and $U_{n+1}$ of $L_n$ and $v_{T_{n+1}}(x)$ respectively.
Since $x \in \alpha(x) \cap L_n \subset  \alpha(x) \cap V_{n+1}$ and $v_{T_{n+1}}(x) \in \alpha(x) \cap U_{n+1}$, the connectivity of $v(\R_{<t}, x)$ for any $t \in \R$   implies that $v(\R_{<t}, x) \not\subseteq V_{n+1} \sqcup U_{n+1}$ and so that there is an unbounded decreasing sequence $(t_{n+1,m})_{m \in \Z_{>0}}$ with $v_{t_{n+1,3m}}(x) \notin V_{n+1} \sqcup U_{n+1}$, $v_{t_{n+1,3m+1}}(x) \in V_{n+1}$, and $v_{t_{n+1,3m+2}}(x) \in U_{n+1}$.
Then the difference $K_{n+1} := O \setminus (V_{n+1} \sqcup U_{n+1})$ is closed and contains $\{ v_{t_{n+1,3m}}(x) \}_{m \in \Z_{>0}}$ and $L_n \cap K_{n+1} = \emptyset$.
The normality of $X$ implies that there are disjoint open \nbds $V'_{n+1}$ and $W_{n+1}$ of $L_n$ and $K_{n+1}$ respectively.
Then $\overline{V_{n+1}'} \cap W_{n+1} = \emptyset$ and $\overline{V_n'} \subseteq L_n \subset V_{n+1}'$.
The difference $L'_{n+1} := O \setminus W_{n+1}$ is a closed subset with $L_n' \subseteq O \cap L_n \subseteq L'_{n+1}$ such that $x \in O \cap V'_{n+1} \subseteq L'_{n+1}$.
Since the closed interval $I_{n+1} := [-(n+1),n+1]$ is compact and so the image $v_x(I_{n+1})$ is compact, Hausdorff separation axiom implies that $v_x(I_{n+1})$ is closed.
The union $L_{n+1} := L'_{n+1} \cup v_x(I_{n+1}) \cup \overline{V_{n+1}'}$ is closed, $x \in L_n \subseteq  L_{n+1}$, and there is a large integer $N>0$ such that $v_{t_{n+1,3m}}(x) \in W_{n+1} \setminus L_{n+1} = W_{n+1} \setminus v_x(I_{n+1})$ for any $m >N$.
This completes the inductive construction.
In particular, we have an open cover $\{ V'_n \}$ of $O$ with $V'_n \subseteq V'_{n+1}$ for any $n \in \Z_{>0}$.
Since $O$ is compact, there is a finite subcover $\{ V'_{n_1}, V'_{n_2}, \ldots , V'_{n_k} \}$.
Since $V'_n \subseteq V'_{n+1}$ for any $n \in \Z_{>0}$, there is an open subset $V'_{n_l}$ such that $O \subseteq \bigcup_{i=1}^k V'_{n_i} = V'_{n_l}$.
Since $V'_{n} \cap K_{n} = \emptyset$ for any $n \in \Z_{>0}$, we have $V'_{n_l} \cap K_{n_l} = \emptyset$.
On the other hand, we obtain $\emptyset \neq \{ v_{t_{n_l,3m}}(x) \mid m \in \Z_{>0} \} \cap K_{n_l} \subseteq O \subseteq V'_{n_l}$ and so $V'_{n_l} \cap K_{n_l} \neq \emptyset$, which contradicts $V'_{n_l} \cap K_{n_l} = \emptyset$.
This means that an orbit which is a closed subset is closed as an orbit.
\end{proof}

%

%
%
%

The compactness is necessary.
Indeed, a flow $v$ on the Euclidean space $\R$ defined by $v(t,x) = x+t$ consists of one non-closed orbit $\R$ which is a closed subset.
Moreover, the Hausdorff separation axiom is necessary.
In fact, there is a flow on a compact space with a non-closed orbit which is a closed subset.
Indeed, let $X$ be the indiscrete topological space $\R$ and $v$ a flow on $X$ defined by $v(t,x) := t+x$.
Then $X$ consists of one orbit $X$ which is a closed subset but neither singular nor periodic.
Moreover, there is a flow such that any orbit is a non-closed subset and is singular.
Indeed, let $X$ be an indiscrete topological space which has at least two points, and $v$ the identical flow on $X$.
Then $X = \Sv$ consists of proper, recurrent singular points which are non-closed subsets.
In addition, there is a flow on a compact space such that any orbit is a non-closed subset and is periodic.
Indeed, let $X$ be an indiscrete topological space which has at least two points, $ \R/\Z$ a circle with the induced topology from the Euclidean topology, $Y := X \times \R/\Z$ the product topological space, and $v$ an $\R$-action on $Y$  defined by $v(t,(x,[y])) = (x,[y+t])$.
Lemma~\ref{cor:indiscrete} implies that $v$ is a flow on $Y$ such that $Y$ consists of non-proper periodic orbits which are not closed subsets.

\subsection{Characterizations of types of points of flows on compact Hausdorff spaces}
Since any Hausdorff space is KC, Lemma~\ref{lem:connectivity} implies the following known fact.

\begin{corollary}\label{cor:connectivity}
Any $\alpha$-limit set and any $\omega$-limit set of a point of a compact Hausdorff space are nonempty, invariant, closed, and connected.
\end{corollary}

The compactness and Hausdorff separation axiom of $X$ implies the following statement.

\begin{lemma}\label{lem:char}
The following properties hold for any point $x \in X$:
\\
$(1)$ $x \in \mathop{\mathrm{Cl}}(v)$ if and only if $O(x) = \check{O}(x) = \hat{O}(x) = \overline{O(x)}$.
\\
$(2)$ If $x \in \mathrm{P}(v)$, then $O(x) = \check{O}(x) = \hat{O}(x)  \subsetneq \overline{O(x)}$.
\end{lemma}

\begin{proof}
If $x \in \mathop{\mathrm{Cl}}(v)$, then $O(x) = \check{O}(x) = \hat{O}(x)  = \overline{O(x)}$.
If $O(x) = \check{O}(x) = \hat{O}(x)  = \overline{O(x)}$, then $O(x)$ is closed and so is a closed subset because of Lemma~\ref{lem:closed}.
Suppose that $x \in \mathrm{P}(v)$.
Lemma~\ref{lem:orbit_cl} implies that $O(x) = \check{O}(x) = \hat{O}(x) $.
Since $\alpha(x) \subset \overline{O(x)}$ is nonempty and $O(x) \cap \alpha(x) = \emptyset$, we have $\emptyset \neq \alpha(x)  \subseteq \overline{O(x)} - O(x)$ and so $O(x) \subsetneq \overline{O(x)}$.
\end{proof}

We have the following non-minimality.

\begin{lemma}\label{lem:non_min}
A point in $\mathrm{P}(v)$ is non-minimal with respect to pre-orders $\leq_\alpha$ and $\leq_\omega$.
\end{lemma}

\begin{proof}
Fix a point $x \in \mathrm{P}(v)$.
Lemma~\ref{lem:kc_proper} implies that $\overline{O(x)} - O(x) = \alpha(x) \cup \omega(x)$ is closed.
Fix a point $y \in \alpha(x)$.
Since $y \in \alpha(x) \subseteq \overline{O(x)} - O(x)$, we obtain $O(y) \neq O(x)$.
The closedness and invariance of $\alpha(x)$ imply that $\alpha(y) \subseteq \overline{O(y)} \subseteq \alpha(x) \subseteq \overline{O(x)} - O(x)$.
Since $x \notin \alpha(x)$, we have $x \notin \alpha(y)$.
This means that $y <_\alpha x$.
By symmetry, we have $z <_\omega x$ for some $z \in \omega(x)$.
\end{proof}

By definition of $\alpha'$ and $\omega'$, Lemma~\ref{lem:decomp_limit} and Lemma~\ref{lem:closed} imply the following observations.

\begin{lemma}\label{lem:empty}
The following are equivalent for a point $x \in X$:
\\
$(1)$ $x \in \mathop{\mathrm{Cl}}(v)$.
\\
$(2)$ $\alpha'(x) = \emptyset$.
\\
$(3)$ $\omega'(x) = \emptyset$.
\\
$(4)$ The orbit $O(x)$ is a closed subset.
\end{lemma}


\begin{lemma}\label{lem:sing}
The following are equivalent for a point $x \in X$:
\\
$(1)$ $x \in \Sv$.
\\
$(2)$ $\{ x \} = \alpha(x)$ and $\alpha'(x) = \emptyset$.
\\
$(3)$ $\{ x \} = \omega(x)$ and $\omega'(x) = \emptyset$.
\end{lemma}

\begin{lemma}\label{lem:per}
The following are equivalent for a point $x \in X$:
\\
$(1)$ $x \in \Pv$.
\\
$(2)$ $\{ x \} \neq \alpha(x)$ and $\alpha'(x) = \emptyset$.
\\
$(3)$ $\{ x \} \neq \omega(x)$ and $\omega'(x) = \emptyset$.
\end{lemma}

\begin{lemma}\label{lem:proper}
The following are equivalent for a point $x \in X$:
\\
$(1)$ $x \in \mathrm{P}(v)$.
\\
$(2)$ $\alpha'(x) = \alpha(x)$ and $\omega'(x) = \omega(x)$.
\\
$(3)$ $O(x) \cap (\alpha(x) \cup \omega(x)) = \emptyset$.
\end{lemma}

\begin{lemma}\label{lem:rec}
The following are equivalent for a point $x \in X$:
\\
$(1)$ $x \in \mathrm{R}(v)$.
\\
$(2)$ $O(x) \subsetneq \alpha(x) \cup \omega(x)$.
\\
$(3)$ $\emptyset \neq \alpha' (x) \subsetneq \alpha(x)$ or $\emptyset \neq \omega' (x) \subsetneq \omega(x)$.
\\
$(4)$ $\emptyset \neq \alpha' (x) \cup \omega' (x) \subsetneq \alpha(x) \cup \omega(x)$.
\end{lemma}

\subsection{On limit sets and classes}

We have the following property.

\begin{lemma}\label{lem:alpha}
For any points $x, y \in X$ with $O(x) \neq O(y)$ and $\alpha' (x) = \alpha' (y) \neq \emptyset$, we have that either $x \notin \alpha(x)=\alpha'(x)$ or $\hat{O}(x) = O(x) \subseteq \mathrm{R}(v)$.
\end{lemma}

\begin{proof}
Suppose that $O(x) \neq O(y)$ and $\alpha' (x) = \alpha' (y) \neq \emptyset$.
Lemma~\ref{lem:empty} implies that $x, y \notin \mathop{\mathrm{Cl}}(v)$ and so $x,y \in \mathrm{P}(v) \sqcup \mathrm{R}(v) = X - \mathop{\mathrm{Cl}}(v)$.
Suppose that $x \notin \alpha(x)$.
Then the invariance of $\alpha(x)$ implies that $O(x) \cap \alpha(x) = \emptyset$ and so that $\alpha(x)=\alpha'(x)$.
Suppose that $x \in \alpha(x)$.
Then $x \in \mathrm{R}(v)$.
The closedness of $\alpha(x)$ implies $\overline{O(x)} = \alpha(x)$ and so $\overline{O(x)} - O(x) = \alpha' (x)$.
We claim that $\hat{O}(x) = O(x)$.
Indeed, assume that $\hat{O}(x) \neq O(x)$.
Fix a pont $z \in \hat{O}(x) - O(x)$.
Then $z \in \hat{O}(x) - O(x) \subseteq \overline{O(x)} - O(x) = \alpha' (x) = \alpha' (y) \subseteq \alpha (y)$.
The invariance and closedness of $\alpha(y)$ imply that $O(x) \subseteq \overline{O(x)} = \overline{O(z)} = \overline{\hat{O}(x) - O(x)} \subseteq \alpha(y)$.
Since $O(x) \neq O(y)$, we have $O(x) \subseteq \alpha(y) \setminus O(y) = \alpha'(y) = \alpha'(x) = \alpha(x) - O(x)$, which is a contradiction.
\end{proof}

By symmetry, we have the following property.

\begin{lemma}\label{lem:omega}
For any points $x, y \in X$ with $O(x) \neq O(y)$ and $\omega' (x) = \omega' (y) \neq \emptyset$, we have that either $x \notin \omega(x)=\omega'(x)$ or $\hat{O}(x) = O(x) \subseteq \mathrm{R}(v)$.
\end{lemma}

We obtain the following properties.

\begin{lemma}\label{lem:P}
For any points $x, y \in X$ with $O(x) \neq O(y), \alpha' (x) = \alpha' (y)$, and $\omega'(x) = \omega' (y)$, we have $\hat{O}(x) = O(x)$.
\end{lemma}

\begin{proof}
Suppose that $\alpha' (x) = \emptyset$.
By Lemma~\ref{lem:empty}, we obtain  $x \in \mathop{\mathrm{Cl}}(v)$.
Lemma~\ref{lem:closed} implies that $O(x) = \hat{O}(x) = \overline{O(x)}$.
Suppose that $\alpha' (x) \neq \emptyset$.
Lemma~\ref{lem:empty} implies that $\omega' (x) \neq \emptyset$ and that the orbit $O(x)$ is neither a closed subset nor a closed orbit.
Lemma~\ref{lem:alpha} and Lemma~\ref{lem:omega} imply that either $x \notin \alpha(x) \cup \omega(x)$ or $\hat{O}(x) = O(x) \subseteq \mathrm{R}(v)$.
Thus we may assume that $x \notin \alpha(x) \cup \omega(x)$.
Then $x \notin \mathrm{R}(v)$.
This means that $x \in X - (\Cv \sqcup \mathrm{R}(v)) = \mathrm{P}(v)$.
Lemma~\ref{lem:char} implies that $\hat{O}(x) = O(x)$.
\end{proof}


%

The previous lemma implies the following statement.


%

\begin{lemma}\label{lem:quotient_sp}
If $\hat{O} \neq O$ for any orbit $O \subseteq \mathrm{R}(v)$, then $O(x) = [x]'' = [x]' \subsetneq \hat{O}(x)$ for any point $x \in \mathrm{R}(v)$.
\end{lemma}

\begin{proof}
Fix a point $x \in \mathrm{R}(w)$.
Since $x \in \alpha(x) \cup \omega(x)$, by time reversion if necessary, we may assume that $x \in \alpha(x)$.
Then $\overline{O(x)} = \alpha(x)$.
Fix a point $y \in \hat{O}(x) - O(x)$.
Since $y \in \overline{O(y)} - O(x) = \overline{O(x)} - O(x) = \alpha(x) - O(x) = \alpha'(x)$, we have $x \in \overline{O(x)} = \overline{O(y)} = \overline{\alpha'(x)}$.
Fix a point $z \in [x]'$.
Since $\alpha'(x) = \alpha'(z)$, we obtain $\alpha(x) = \overline{O(x)} = \overline{\alpha'(x)} = \overline{\alpha'(z)} \subseteq \alpha(z)$ and so $\alpha(z) - O(z) = \alpha'(z) = \alpha'(x) = \alpha(x) - O(x) \subseteq \alpha(z) - O(x)$.
Therefore $\alpha(z) - O(z) \subseteq \alpha(z) - O(x)$.
Since $O(z) = O(x)$ or $O(z) \cap O(x) = \emptyset$, the inclusion $O(z) \cup O(x) \subseteq \alpha(z)$ implies that $O(z) = O(x)$.
This means that $O(x) = [x]'' = [x]'$.
\end{proof}

Recall that a topological space is paracompact if any open cover has a locally finite open refinement.
We have the follow statement.

\begin{corollary}\label{cor:quotient_sp}
Let $w$ be a flow on a compact manifold $M$.
The following properties hold for any point $x \in M$:
\\
$(1)$ $x \in \mathop{\mathrm{Cl}}(w)$ if and only if $O(x) = \check{O}(x) = \hat{O}(x) = \overline{O(x)}$.
\\
$(2)$ $x \in \mathrm{P}(w)$ if and only if $O(x) = \check{O}(x) = \hat{O}(x)  \subsetneq \overline{O(x)}$.
\\
$(3)$ $x \in \mathrm{R}(w)$ if and only if $O(x) = [x]'' = [x]' \subsetneq \hat{O}(x)$.
\end{corollary}

\begin{proof}
\cite[Lemma 3.1]{yokoyama2019properness} states that an orbit $O$ of a flow on a paracompact manifold is non-closed recurrent (i.e. $O \subseteq \mathrm{R}(w)$) if and only if $\hat{O} \neq O$.
Lemma~\ref{lem:char} and Lemma~\ref{lem:quotient_sp} imply the assertion.
\end{proof}

Compactness is necessary.
In fact, there is a flow on a metrizable space $X$ consisting of one orbit such that $X = \hat{O} = O = \mathrm{R}(v)$.
Indeed, consider an irrational rotation flow $w$ on a torus.
Fix an orbit $O$.
Then the restriction $w|_O$ is desired.
The author would like to know whether a non-closed orbit $O$ of a flow on a compact Hausdorff space is recurrent if and only if $\hat{O} \neq O$.
In other words, the author would like to know an answer to the following question.
\begin{question}
For a recurrent non-closed orbit $O$ of a flow on a compact Hausdorff space, does an equality $\hat{O} \neq O$ hold?
\end{question}

%
%
%

\subsection{Properties of classes}

%

We have the following equivalences.

\begin{lemma}\label{lem:018}
The following property holds for any point $x \in X$:
\\
$(1)$ $x \in \mathop{\mathrm{Cl}}(v)$ if and only if $[x]' \subseteq \mathop{\mathrm{Cl}}(v)$.
\\
Moreover, if $\hat{O} \neq O$ for any orbit $O \subseteq \mathrm{R}(v)$, then the following properties hold:
\\
$(2)$ $x \in \mathrm{P}(v)$ if and only if $[x]' \subseteq \mathrm{P}(v)$.
\\
$(3)$ $x \in \mathrm{R}(v)$ if and only if $[x]' \subseteq \mathrm{R}(v)$.
\end{lemma}

\begin{proof}
Suppose that $x \in \mathop{\mathrm{Cl}}(v)$.
Then $[x]'$ is the connected component of $\{ y \in X \mid \alpha' (x) = \alpha' (y), \omega'(x) = \omega' (y) \}$ containing $x$.
Lemma~\ref{lem:empty} implies that $\{ y \in X \mid \alpha' (x) = \alpha' (y), \omega'(x) = \omega' (y) \} = \{ y \in X \mid \emptyset = \alpha' (y) = \omega' (y) \} = \mathop{\mathrm{Cl}}(v)$ and so that the subset $[x]'$ is the connected component of $\mathop{\mathrm{Cl}}(v)$.
This implies that the assertion $(1)$ holds.
Assume that $\hat{O} \neq O$ for any orbit $O \subseteq \mathrm{R}(v)$.
Suppose that $x \in \mathrm{R}(v)$.
Corollary~\ref{lem:quotient_sp} implies that $[x]' = O(x) \subseteq \mathrm{R}(v)$.
Suppose that $x \in \mathrm{P}(v)$.
Assume that there is a point $y \in [x]' \cap (\mathop{\mathrm{Cl}}(v) \sqcup \mathrm{R}(v))$.
The statements $(1)$ and $(3)$ imply that $[x]' = [y]' \subseteq \mathop{\mathrm{Cl}}(v) \sqcup \mathrm{R}(v) = X - \mathrm{P}(v)$, which contradicts $[x]' \subseteq \mathrm{P}(v)$.
\end{proof}

The previous lemma and \cite[Lemma 3.1]{yokoyama2019properness} imply the following statement.

\begin{corollary}\label{cor:018}
Let $w$ be a flow on a compact manifold $M$.
The following properties hold for any point $x \in M$:
\\
$(1)$ $x \in \mathop{\mathrm{Cl}}(w)$ if and only if $[x]' \subseteq \mathop{\mathrm{Cl}}(v)$.
\\
$(2)$ $x \in \mathrm{P}(w)$ if and only if $[x]' \subseteq \mathrm{P}(v)$.
\\
$(3)$ $x \in \mathrm{R}(w)$ if and only if $[x]' \subseteq \mathrm{R}(v)$.
\end{corollary}

By definition of $[x]''$, we have the following statement.

\begin{lemma}\label{lem:type}
The following properties hold for any point $x \in X$:
\\
$(1)$ $x \in \Sv$ if and only if $[x]'' \subseteq \Sv$.
\\
$(2)$ $x \in \Pv$ if and only if $[x]'' \subseteq \Pv$.
\\
$(3)$ $x \in \mathrm{P}(v)$ if and only if $[x]''  \subseteq \mathrm{P}(v)$.
\\
$(4)$ $x \in \mathrm{R}(v)$ if and only if $[x]'' \subseteq \mathrm{R}(v)$.
\end{lemma}

\subsubsection{Reduction to the Morse graph}

The Morse graph of a flow can be obtained as a quotient space of the abstract  orbit space as follows.

\begin{theorem}\label{th:Morse_reduction}
Let $v$ be a flow on a compact metric space $X$.
Then the Morse graph $G_v$ of a flow is a quotient space of the abstract orbit space $X/\langle v \rangle$ with pre-orders $\leq_\alpha$ and $\leq_\omega$.
In particular, the underlying space of the Morse graph $G_v$ of a flow is a quotient space of the abstract orbit space $X/\langle v \rangle$.
\end{theorem}

\begin{proof}
Since the direction of edges are determined by pre-orders $\leq_\alpha$ and $\leq_\omega$, it suffices to show that the underlying space of the Morse graph  of a flow is a quotient space of the abstract orbit space.
By Corollary~\ref{cor:connectivity}, any $\omega$-limit set and any $\alpha$-limit set of a point in $X$ are nonempty connected closed invariant subsets.
Let $G_v = (V, D)$ be the Morse graph of $v$ with $V = \{ M_i \}$ and $D = \{ (M_j, M_k) \mid D_{j,k} \neq \emptyset \}$.
Here $D_{j,k} = \{ x \in M - \bigsqcup_i M_i \mid \alpha(x) \subseteq M_j, \omega(x) \subseteq M_k \}$.
Then $X = \bigsqcup_i M_i \sqcup \bigsqcup_{j,k}D_{j,k}$.
By definition of Morse graph, the union of Morse set $\bigsqcup V = \bigsqcup_i M_i$ is the chain recurrent set $\mathop{CR}(v)$ (i.e. $\bigsqcup_i M_i = \mathop{CR}(v)$).
By \cite[Theorem~3.3B]{Conley1988}, the chain recurrent set $\mathop{CR} (v)$ is closed and invariant and contains the non-wandering set $\Omega (v)$.
By the decomposition $X = \Cv \sqcup \mathrm{P}(v) \sqcup \mathrm{R}(v)$, we have $X - \mathrm{P}(v)= \mathcal{R}(v)$.
Since any recurrent point is non-wandering, we have that $\Cv \sqcup \mathrm{R}(v) = X - \mathrm{P}(v)= \mathcal{R}(v) \subseteq \Omega(v) \subseteq \mathop{CR}(v) = \bigsqcup_i M_i$.
Moreover \cite[Lemma~3.1B]{Conley1988} implies that connected components of $\mathop{CR} (v)$ are equivalence classes of the relation $\approx_{\mathop{CR}}$ on $\mathop{CR} (v)$, where $x \approx_{\mathop{CR}} y$ if $x \sim_{\mathop{CR}} y$ and $y \sim_{\mathop{CR}} x$.
Since Morse sets are connected components of $\mathop{CR} (v)$, Morse sets are equivalence classes of the relation $\approx_{\mathop{CR}}$ on $\mathop{CR} (v)$.
We show that any Morse set is a union of abstract orbits.
Indeed, fix a point $x \in X$.
If $x$ of $\Sv$ (resp. $\Pv$), then Lemma~\ref{lem:awo} implies that the abstract orbit $\langle x \rangle = [x]''$ is the connected component of $\Sv$ (resp. $\Pv$) containing $x$ and so is contained in some Morse set $M_i$, because Morse sets are connected components of $\mathop{CR} (v)$ and $\Cv \subseteq \mathop{CR} (v)$.
Suppose that $x \in \mathrm{R}(v)$.
Then $O(x) \subseteq \langle x \rangle = \hat{O}(x) \subseteq \overline{O(x)} \cap \mathrm{R}(v) \subseteq \mathcal{R}(v) \subseteq \mathop{CR}(v) = \bigsqcup_i M_i$.
The invariance and closedness of $\mathop{CR}(v)$ imply that $\overline{O(x)} \subseteq \mathop{CR}(v)$.
Since the closure of a connected subset is connected, the closure $\overline{O(x)}$ is connected and so is contained in some Morse set $M_i$ for some $i$.
Then $\langle x \rangle \subseteq \overline{O(x)} \subseteq M_i$.
Suppose that $x \in \mathrm{P}(v) \cap \bigsqcup_i M_i$.
Since $\bigsqcup_i M_i = \mathop{CR}(v)$, we have $x \in \mathrm{P}(v) \cap \mathop{CR}(v)$.
The invariance and closedness of $\mathop{CR}(v)$ imply that $\overline{O(x)} \subseteq \mathop{CR}(v)$.
Since the closure of a connected subset is connected, the closure $\overline{O(x)}$ is connected and so is contained in some Morse set $M_i$.
Then $\alpha'(x) \cup \omega'(x) = \alpha(x) \cup \omega(x) \subseteq \overline{O(x)} \subseteq M_i$.
Lemma~\ref{lem:type} implies that $\langle x \rangle = [x]'' \subseteq \{ y \in \mathrm{P}(v) \mid \alpha' (x) = \alpha' (y), \omega'(x) = \omega' (y) \} =
 \{ y \in \mathrm{P}(v) \mid \alpha (x) = \alpha (y), \omega(x) = \omega (y) \}$.
Fix a point $y \in \langle x \rangle$.
For any $\varepsilon > 0$ and $T>0$, there are an $(\varepsilon, T)$-chain from a point $\alpha$ in $\alpha (x) = \alpha (y) \subseteq M_i$ to $y$ and an $(\varepsilon, T)$-chain from $y$ to a point $\omega$ in $\omega (x) = \omega (y) \subseteq M_i$.
In other words, we obtain $\alpha \sim_{\mathop{CR}} y$ and $y \sim_{\mathop{CR}} \omega$.
By \cite[Theorem~3.3C]{Conley1988}, we have $w \sim_{\mathop{CR}} z$ and $z \sim_{\mathop{CR}} w$ for any points $w,z \in M_i$.
Since $\alpha, x,  \omega \in M_i$, we obtain $x \sim_{\mathop{CR}} \alpha$ and $\omega \sim_{\mathop{CR}} x$.
The transitivity of $\sim_{\mathop{CR}}$ implies that $x \sim_{\mathop{CR}} y$ and $y \sim_{\mathop{CR}} x$.
In particular, we obtain $\overline{O(y)} \subset \mathop{CR}(v)$.
Since $\alpha(y) = \alpha(x) \subset M_i$ is connected, and since $M_i$ is a connected component of $\mathop{CR}(v)$, we have $\overline{O(y)} \subset M_i$.
This means that $\langle x \rangle \subseteq M_i$.
Thus any Morse set is a union of abstract orbits.
Since the union $\bigsqcup_i M_i$ of Morse sets is a union of abstract orbits, so is the complement $\bigsqcup_{j,k}D_{j,k} = X - \bigsqcup_i M_i$.
By $X - \mathrm{P}(v) = \Cv \sqcup \mathrm{R}(v) = \mathcal{R}(v) \subseteq \mathop{CR}(v) = \bigsqcup_i M_i = X - \bigsqcup_{j,k}D_{j,k}$, we have $\mathrm{P}(v) \setminus \mathop{CR}(v) = \bigsqcup D_{j,k}$.
We show that any $D_{j,k}$ consists of abstract orbits.
Indeed, fix a point $x \in D_{j,k}$.
Then $\langle x \rangle \subseteq \mathrm{P}(v) \setminus \mathop{CR}(v)$, $\alpha(x) \subseteq M_j$, and $\omega(x) \subseteq M_k$.
Therefore $\langle x \rangle = [x]'' \subseteq \{ y \in \mathrm{P}(v) \setminus \mathop{CR}(v) \mid \alpha' (x) = \alpha' (y), \omega'(x) = \omega' (y) \} = \{ y \in \mathrm{P}(v)  \setminus \mathop{CR}(v) \mid \alpha (x) = \alpha (y) \subseteq M_j, \omega(x) = \omega (y) \subseteq M_k \} \subseteq D_{j,k}$.
This means that $D_{j,k}$ consists of abstract orbits.
Therefore the Morse graph of $v$ is a quotient space of the abstract orbit space $X/\langle v \rangle$.
\end{proof}

The previous theorem, Proposition~\ref{prop:quotient}, and Lemma~\ref{prop:quotient} imply the following reduction.

\begin{corollary}\label{cor:Morse_reduction_weak}
Let $v$ be a flow on a compact metric space $X$.
Then the Morse graph $G_v$ of a flow is a quotient space of the abstract weak orbit space $X/[ v ]$, of the weak orbit class space $X/\check{v}$, and of the orbit class space $X/\hat{v}$ with pre-orders $\leq_\alpha$ and $\leq_\omega$.
\end{corollary}

\section{Properties of binary relations of flows on compact Hausdorff spaces}

Let $v$ be a flow on a compact Hausdorff space $X$.
We have the following non-minimality.

\begin{lemma}\label{lem:non_min_abst}
The abstract weak orbit {\rm(resp.} abstract orbit {\rm)} of a point in $\mathrm{P}(v)$ is non-minimal with respect to pre-orders $\leq_\alpha$ and $\leq_\omega$ on $X/[v]$ {\rm(resp.} $X/\langle v \rangle$ {\rm)}.
\end{lemma}

\begin{proof}
Fix a point $x$ in $\mathrm{P}(v)$.
Then $[x] = \langle x \rangle \subseteq \mathrm{P}(v)$ is the connected component of the subset $\{ y \in X \mid \alpha' (x) = \alpha' (y), \omega'(x) = \omega' (y) \} = \{ y \in \mathrm{P}(v) \mid \alpha(x) = \alpha(y), \omega(x) = \omega(y) \}$ containing $x$.
By Lemma~\ref{lem:non_min}, there are point $y, z \in X$ with $y <_\alpha x$ and $z <_\omega x$.
Then $O(x) \neq O(y)$, $y \in \alpha(x)$, and $x \notin \alpha(y)$.
Since $y \in \alpha(x)$, we have $[y] \leq_\alpha [x]$ and $\langle y \rangle \leq_\alpha \langle x \rangle$.
We claim that $[x] \neq [y]$ and $\langle y \rangle \neq \langle x \rangle$.
Indeed, assume either $[x] = [y]$ or $\langle y \rangle = \langle x \rangle$.
Since $[x] = \langle x \rangle = [x]'' \subseteq \mathrm{P}(v)$, we have $[y] = \langle y \rangle = [y]'' \subseteq \mathrm{P}(v)$.
Then $[x] = [y]$ and so $\alpha(x) = \alpha(y)$.
On the other hand, since $y \in \mathrm{P}(v)$, we have $y \in \alpha(x) - \alpha(y) = \emptyset$, which is a contradiction.
Therefore $[y] <_\alpha [x]$ and so $\langle y \rangle <_\alpha \langle x \rangle$.
By symmetry, we have $[z] <_\omega [x]$ and $\langle y \rangle <_\omega \langle x \rangle$.
\end{proof}

We have the following characterization of a trivial abstract (weak) orbit space.

\begin{lemma}
The following statements are equivalent for a flow $v$ on a compact connected Hausdorff space $X$:
\\
$(1)$ The abstract weak orbit space $X/[v]$ is a singleton.
\\
$(2)$ The abstract orbit space $X/\langle v \rangle$ is a singleton.
\\
$(3)$ One of the following conditions holds:

$(i)$ The flow $v$ is identical $(\mathrm{i.e.} \,\, X = \Sv )$.

$(ii)$ The flow $v$ is pointwise periodic $(\mathrm{i.e.} \,\, X = \Pv )$.

$(iii)$ The flow $v$ is minimal.
\\
In the case $(iii)$, if $v$ is neither identical nor pointwise periodic, then $X = \mathrm{R}(v)$.
\end{lemma}

\begin{proof}
Since the abstract orbit space $X/\langle v \rangle$ is a quotient space of the abstract weak orbit space $X/[v]$, if $X/[v]$ is a singleton, then so is $X/\langle v \rangle$.
Suppose that the abstract orbit space $X/\langle v \rangle$ is a singleton.
Fix a point $x \in X$.
Lemma~\ref{lem:type} and Corollary~\ref{cor:check_eq} imply that $X = \langle x \rangle$ is either $\Sv$, $\Pv$, $\mathrm{P}(v)$, or $\mathrm{R}(v)$.
By Lemma~\ref{lem:non_min_abst}, the abstract weak orbit $\langle x \rangle$ is either $\Sv$, $\Pv$, or $\mathrm{R}(v)$.
If $\langle x \rangle \subseteq \Sv$ {(resp.} $\langle x \rangle \subseteq \Pv$ {\rm)}, then $X = \langle x \rangle =[x] = \Sv$ {(resp.} $X = \langle x \rangle =[x] = \Pv$ {\rm)}.
Thus we may assume that $X = \langle x \rangle = \mathrm{R}(v)$.
Then $\langle x \rangle = \hat{O}(x)$ and so $X = \overline{O(x)} = \overline{O(y)}$ for any point $y \in \langle x \rangle = X$.
This means that each minimal set is the whole space $X$ and so $X = \langle y \rangle = [y] = \overline{O(y)}$ for any $y \in X$.
Conversely, suppose that $X = \Sv$.
Then $[x] = \langle x \rangle$ for any point $x \in X$ is the connected component of $\Sv$ containing $x$ and so $X$.
Suppose that $X = \Pv$.
Then $[x] = \langle x \rangle$ for any point $x \in X$ is the connected component of $\Pv$ containing $x$ and so $X$.
Suppose that $v$ is minimal.
Then each nonempty invariant set is $X$ and so $X = \alpha(x) = \omega(x) = \overline{O(x)}$ for any point $x \in X$.
This means that $X = \check{O}(x) = \langle x \rangle = [x]$ for any point $x \in X$.
\end{proof}


We have the following triviality of Morse graphs and non-triviality of abstract (weak) orbit spaces.

\begin{proposition}\label{prop:refine}
Let $v$ be a flow on a compact connected Hausdorff space $X$.
Suppose that $X = \mathop{CR}(v)$ and that $v$ is neither identical, pointwise periodic, nor minimal.
Then the Morse graph of $v$ is a singleton but the abstract {\rm(}weak {\rm)} orbit space is not a singleton.
\end{proposition}

\section{Flows of finite type on compact manifolds}

\subsection{Topological characterization of types of orbits}

A point $x \in X$  is $T_D$ \cite{aull1962separation} if the derived set $\overline{\{ x \}} - \{ x \}$ of $\{ x\}$ is a closed subset.
A point $x$ in $X$ is $S_D$ if the point $\hat{x}$ in the $T_0$-tification $\hat{X}$ is $T_D$.
In other words, a point $x \in X$ is $S_{D}$ if and only if $\overline{\{ \hat{x} \}} - \{ \hat{x} \} \subseteq \hat{X}$ is closed in $\hat{X}$.
Lemma~\ref{lem:quotient_top} implies that a point $x \in X$ is $S_{D}$ if and only if $\overline{x} - \hat{x} \subseteq X$ is closed in $X$.
We have the following characterization of properness.

\begin{lemma}\label{lem:proper_orbit}
The following conditions are equivalent for an orbit $O$ of a flow $v$ on a paracompact manifold $M$:
\\
$(1)$ The orbit $O$ is proper.
\\
$(2)$ $O \subseteq \Cv \sqcup \mathrm{P}(v)$.
\\
$(3)$ The orbit $O$ is $T_D$ as a point in the orbit space $M/v$.
\\
$(4)$ The orbit $O$ has an open neighborhood in which $O$ is a closed subset.
\\
$(5)$ $O = \hat{O}$.
\\
$(6)$ The coborder $\bp O = \overline{O} - O$ is closed.
\end{lemma}

\begin{proof}
Fix an orbit $O$.
If $O$ is a closed orbit, then all conditions hold.
Thus we may assume $O \subseteq M - \Cv = \mathrm{P}(v) \sqcup \mathrm{R}(v)$.
Corollary~\ref{cor:proper} implies that the conditions $(1)$ and $(2)$ are equivalent.
\cite[Corollary 3.4]{yokoyama2019properness} implies that the conditions $(1)$ and $(5)$ are equivalent.
By Lemma~\ref{lem:kc_proper}, the condition $(1)$ implies the condition $(4)$.
Lemma~\ref{lem:quotient_top} implies that the conditions $(3)$ and $(6)$ are equivalent.
%
If $O$ has an open neighborhood $U$ in which $O$ is a closed subset, then $\overline{O} \cap (X - U) = \overline{O} - O$ is closed.
This means that the condition $(4)$ implies the condition $(6)$.
It suffices to show that the condition $(6)$ implies the condition $(2)$.
Suppose that the coborder $\overline{O} - O$ is closed.
Assume that $O \subseteq \mathrm{R}(v)$.
\cite[Lemma 3.1]{yokoyama2019properness}
implies $O \subsetneq \hat{O} \subseteq \overline{O}$ and so $\emptyset \neq \hat{O} - O \subseteq \overline{O} - O$.
Then $\overline{\hat{O} - O} = \overline{\overline{O} - O} = \overline{O}$.
Therefore the coborder $\overline{O} - O \subsetneq \overline{O} = \overline{\overline{O} - O}$ is not closed as a subset of $M$, which contradicts that the coborder $\overline{O} - O$ is closed as a subset of $M$.
Thus $x \in M -  \mathrm{R}(v) = \Cv \sqcup \mathrm{P}(v)$.
\end{proof}

Let $v$ be a flow on a compact manifold $M$, $\tau_v$ be the quotient topology on the orbit space $M/v$.
Recall that the orbit space $M/v$ is a decomposition $\{ O : \text{orbit} \}$ as a set.
Define the induced topology $\tau_v(M)$ on $M$ of $\tau_v$ by $\tau_v(M) := \{ \bigcup  \mathcal{U} \mid \mathcal{U} \in \tau_v \}$.
Note $\bigcup  \mathcal{U} = \bigcup_{O \in \mathcal{U}} O = \{ x \in M \mid x \in O \in \mathcal{U} \}$ for any open subset $\mathcal{U} \in \tau_v(M)$.
A point $x$ in $X$ is $S_1$ if the point $\hat{x}$ in the $T_0$-tification $\hat{X}$ is $T_1$.
In other words, a point $x \in X$ is $S_{1}$ if and only if $\{ \hat{x} \} \subseteq \hat{X}$ is closed in $\hat{X}$.
Lemma~\ref{lem:quotient_top} implies that a point $x \in X$ is $S_{1}$ if and only if $\hat{x} \subseteq X$ is closed in $X$.
We have the following topological characterization of types of orbits.

\begin{lemma}
Let $v$ be a flow on a compact manifold $M$.
The following statements hold with respect to the induced topology $\tau_v(M)$ on $M$ for any point $x \in M$:
\\
$(1)$ $x \in \Sv$ if and only if $x$ is $T_1$.
\\
$(2)$ $x \in \Cv$ if and only if $x$ is $S_1$.
\\
$(3)$ $x \in \Cv \sqcup \mathrm{P}(v)$ if and only if $x$ is $S_D$.
\end{lemma}

\begin{proof}
By definition of $S_D$, Lemma~\ref{lem:proper_orbit} implies that the assertion (3).
Definitions of $T_1$ and $S_1$ imply that a point $x \in M$ is $T_1$ if and only if $\{ x \}$ is closed with respect to $\tau_v(M)$, and that it is $S_1$ if and only if $O(x)$ is closed with respect to the quotient topology $\tau_v$ on the orbit space $M/v$.
This means that the assertions $(1)$ and $(2)$ hold.
\end{proof}

The previous lemma implies the following characterizations.

\begin{corollary}
Let $v$ be a flow on a compact manifold $M$.
The following statements hold for any orbit $O \subseteq M$:
\\
$(1)$ $O \subseteq \Cv$ if and only if $O$ is $T_1$ as a point in the orbit space $M/v$.
\\
$(2)$ $O \subseteq \mathrm{P}(v)$ if and only if $O$ is not $T_1$ but $T_D$ as a point in  the orbit space $M/v$.
\\
$(3)$ $O \subseteq \mathrm{R}(v)$ if and only if $O$ is non-$T_D$ as a point in  the orbit space $M/v$.
\end{corollary}

\subsection{Fundamental concepts and typical classes of dynamical systems}
Recall several concepts to state finiteness of flows.

\subsubsection{Topological equivalence and locally topological equivalence}

A flow $v$ on a topological space $X$ is topologically equivalent to a flow $w$ on a topological space $Y$ if there is a homeomorphism $h \colon X \to Y$ whose image of any orbit of $v$ is an orbit of $w$ and which preserves the direction of the orbits.
A point $x \in X$ is locally topologically equivalent to a point $y \in Y$ if
there are \nbds $U_x$ and $U_y$ of $x$ and $y$ respectively, and a homeomorphism $h \colon U_x \to U_y$ such that the images of connected components of the intersection of an orbit and $U_x$ are those of the intersection of an orbit and $U_y$ (i.e. $h(C_p) = C_{h(p)}$ for any point $p \in U_x$, where $C_p$ is the connected component of $O(p) \cap U_x$ containing $p$ and $C_{h(p)}$ is the connected component of $O(h(p)) \cap U_y$ containing $h(p)$), and that $h$ preserves the direction of the orbits.

\subsubsection{Limit cycles and limit circuits}
A cycle is a periodic orbit.
A limit cycle $O$ is a periodic orbit which is an $\omega$-limit set or $\alpha$-limit set (i.e. $O = \omega(x)$ or $O = \alpha(x)$) of a point $x \notin O$.
%
A trivial circuit is a singleton.
A non-trivial circuit is an immersed image of a circle.
A circuit is either trivial or non-trivial circuit.
A non-trivial circuit $\gamma$ of a flow $v$ is limit if it is either a limit cycle or a finite union of orbits in $\Sv \sqcup \mathrm{P}(v)$ which is an $\omega$-limit set or $\alpha$-limit set (i.e. $\gamma = \omega(x)$ or $\gamma = \alpha(x)$) of a point $x \notin \gamma$.
A limit circuit $\gamma$ is repelling (resp. attracting) if there is a \nbd $U$ of $\gamma$ such that $\alpha(x) = \gamma$ (resp. $\omega(x) = \gamma$) for any $x \in U - \gamma$.
A limit circuit $\gamma$ is semi-repelling (resp. semi-attracting) if there are a \nbd $U$ of $\gamma$ and a connected component $A$ of $U - \gamma$ such that $\alpha(x) = \gamma$ (resp. $\omega(x) = \gamma$) for any $x \in A$.

\subsubsection{Flow of finite type}

A singular point is quasi-nondegenerate if it has a neighborhood which intersects at most finitely many abstract weak orbits.
A flow $v$ on a topological space $X$ is a flow of finite type on the topological space if it satisfies the following conditions:
\\
$(1)$ Any singular point is quasi-nondegenerate.
\\
$(2)$ There are at most finitely many limit cycles.
\\
$(3)$ Any recurrent orbit is closed (i.e. $X = \mathop{\mathrm{Cl}}(v) \sqcup \mathrm{P}(v)$).
\\
When the whole space $X$ is a surface, we require quasi-regularity of singular points for definition of ``of finite type on a surface'' (see the details in \ref{def:2dflows}).
Lemma~\ref{lem:w_s} implies the following statement.

\begin{corollary}\label{cor:w_s}
Let $v$ be a flow of finite type on a topological space $X$.
Then the abstract orbit and the abstract weak orbit of a point of coincide with each other.
Moreover, we obtain $X/[v] = X/\langle v \rangle$.
\end{corollary}

\subsubsection{Chaotic flows in the sense of Devaney}
A flow is topologically transitive if, for any pair of nonempty open subsets $U$ and $V$, there is a number $t >1$ such that $v_t(U) \cap V \neq \emptyset$.
A topologically transitive flow $w$ on a compact metric space $M$ with $\overline{\mathop{\mathrm{Cl}}(w)} = M$ is chaotic in the sense of Devaney if it is sensitive to initial conditions (i.e. there is a positive number $\delta > 0$ such that for any $x \in M$ and any neighborhood $U$ of $x$, there are $y \in U$ and $T > 0$ such that ${\displaystyle d(w^T(x), w^T(y))> \delta}$).



\subsubsection{Unstable subsets and stable subsets}
For a subset $A$ of $X$, the unstable set $W^u(A)$ and the stable set $W^s(A)$) of $A$ are defined by $W^u(A) := \{ x \in X \mid \alpha(x) \subseteq A \}$ and $W^s(A) := \{ x \in X \mid \omega(x) \subseteq A \}$.
The unstable set $W^u(A)$ is the unstable manifold if it is an immersed manifold (i.e. the inclusion mapping $W^u(A) \to X$ is an injective immersion).
Here an immersion is a mapping between manifolds whose derivative at any point  is injective.
Similarly, the stable set $W^s(A)$ is the stable manifold if there is an immersion $W^s(A) \to X$ between manifolds.

\subsubsection{Structurally stability}
For a subset $\chi$ of the set $\chi^r(M)$ of $C^r$ vector fields for any $r \in \Z_{>0}$ on a manifold $M$, a vector field $X \in \chi$ is structurally stable with respect to $\chi$ if there is a $C^1$ \nbd $\mathcal{U} \subseteq \chi$ of $X$ such that any vector field $Y$ in $\mathcal{U}$ is topologically equivalent to $X$ (i.e. there is a homeomorphism $h \colon M \to M$ whose image of any orbit of $Y$ is an orbit of $X$ and which preserves the direction of the orbits).

\subsubsection{Gradient flows and Morse-Smale flows}
A $C^r$ vector field $X$ for any $r \in \Z_{\geq0}$ on a Riemannian manifold $(M,g)$ is gradient if there is a $C^{r+1}$ function $h \colon  M \to \R$ such that $X = - \mathrm{grad} (h)$, where the gradient of $h$ is defined by $dh = g(\mathrm{grad}(h), \cdot)$.
In other words, locally the gradient vector field $X$ is defined by $X :=  \sum_{i,j}g^{ij} (\partial_i h) \partial_j$ for any local coordinate system $(x_1, \ldots , x_n)$ of a point $p \in M$, where $\partial_i := \partial/\partial x_i$, $g_{ij} := g(\partial_i, \partial_j)$, and $(g^{ij}) := (g_{ij})^{-1}$.
A flow is gradient if it is topologically equivalent to a flow generated by a gradient  vector field.
A $C^r$ vector field $X$ for any $r \in \Z_{\geq0}$ on a closed manifold is Morse-Smale if $(1)$ the non-wandering set $\Omega(X)$ consists of finitely many hyperbolic closed orbits; $(2)$ any point in the intersection of the stable and unstable manifolds of closed orbits are transversal (i.e $W^s(O) \pitchfork W^u(O')$ for any orbits $O, O' \subset \Omega(X)$, where $W^s(O)$ is the stable manifold of $O$ and $W^u(O')$ is the unstable manifold of $O'$).
Here the transversality of submanifolds $A$ and $B$ on a manifold $M$ means that the submanifolds $A$ and $B$ span the tangent spaces for $M$ (i.e. $T_{A\cap B} M = T_{A\cap B} A + T_{A\cap B} B$).
Similarly, under two generic conditions for differentials to guarantee structural stability, Labarca and Pacifico defined a Morse-Smale vector field on a compact manifold as follows \cite{labarca1990stability}.
A $C^\infty$ vector field $X$ on a compact manifold is Morse-Smale  if $(\mathrm{MS}1)$ the non-wandering set $\Omega(X)$ consists of finitely many hyperbolic closed orbits; $(\mathrm{MS}2)$ the restriction $X|_{\partial M}$ is Morse-Smale; $(\mathrm{MS}3)$ for any orbits $O, O' \subset \Omega(X)$ and for any non-transversal point $x \in W^s(O) \cap W^u(O')$, we have $x \in \partial M$ and either $O$ or $O'$ is singular with respect to $X$ (i.e. $O \subseteq \mathop{\mathrm{Sing}}(X)$ or $O' \subseteq \mathop{\mathrm{Sing}}(X)$).
Therefore we say that a flow is Morse-Smale if it is topologically equivalent to a flow generated by a vector field satisfying conditions $(\mathrm{MS}1)$--$(\mathrm{MS}3)$.
A flow is Morse if it is a Morse-Smale flow without limit cycles.
Note that the two generic conditions for differentials form an open dense subset of the set of $C^\infty$ vector fields and that they are stated as follows: any closed orbit is $C^2$ linearizable, and the weakest contraction (resp. expansion) at any closed orbit is defined.
Here the weakest contraction at a singular (resp. periodic) point $p$ is defined if the contractive eigenvalue with biggest real part among the contractive eigenvalues of $DX(p)$ (resp. $DX_f(p)$, where $X_f$ is  the Poincar\'e map) is simple. Dually we can define that the weakest expansion at p is defined.
Note that a Morse-Smale flow on a compact surface satisfies the $C^2$ linearizable condition and the eigenvalue conditions up to topological equivalence.
Indeed, any hyperbolic closed orbit is either a sink, $\partial$-sink, a source, a $\partial$-source, a saddle, a $\partial$-saddle, and a limit cycle which is either attracting on each side or repelling on each side.
Since any limit cycle is either attracting on each side or repelling on each side, it can be perturbed into a $C^2$ linearizable cycle.
Any small \nbd of a hyperbolic singular point can be identified with one of a hyperbolic singular points for a gradient flow and so can be perturbed into a $C^2$ linearizable singular point satisfying the eigenvalue conditions.
Palis and Smale showed that a Morse-Smale $C^r$ vector field on a closed manifold is structurally stable with respect to the set of $C^r$ vector fields~\cite{palis1969morse,palis2000structural}.
Similarly, the assertion also holds for Morse-Smale vector fields on compact manifolds under the $C^2$ linearizable condition and the eigenvalue conditions \cite{labarca1990stability}.

\subsection{Refinements of CW decompositions of unstable manifolds of singular points of Morse flows}

It is known that the set of the unstable manifolds of singular points of a Morse flow on a closed manifold is a finite CW decomposition~\cite[Theorem~1]{abbondandolo2020stable}.
We show that the abstract weak orbit spaces of Morse flows on closed manifold are refinements of CW decompositions of unstable manifolds of singular points. 

\begin{theorem}
The CW decomposition of unstable manifolds of singular points of a Morse flow on a closed manifold is a quotient space of the abstract weak orbit space. 
Moreover, the set of connecting orbit sets of singular points is also a quotient space of the abstract weak orbit space. 
\end{theorem}

\begin{proof}
Let $v$ be a Morse flow on a closed manifold $M$. 
\cite[Theorem~1]{abbondandolo2020stable} implies that both the set $\{ W^u(x) \mid x \in \Sv \}$ of unstable manifolds of singular points and the set $\{ W^s(x) \mid x \in \Sv \}$ of stable manifolds of singular points are finite CW decompositions of $M$. 
Then $M = \bigsqcup_{x \in \Sv} W^u(x) = \bigsqcup_{x \in \Sv} W^s(x) = \bigsqcup_{x,y \in \Sv} W^u(x) \cap W^s(y)$. 
In particular, we obtain $W^u(x) = \bigsqcup_{y \in \Sv} W^u(x) \cap W^s(y)$ for any $x \in \Sv$. 
This implies that the $\omega$-limit sets and the $\alpha$-limit set of any points are singular points. 
Then $W^u(\alpha(x)) \cap W^s(\omega(x)) = \{ x \}$ for any $x \in \Sv$. 
This implies that $M = \Sv \sqcup \bigsqcup_{x \neq y \in \Sv} W^u(x) \cap W^s(y)$. 
For any $x \in \Sv$, we obtain $W^u(x) = \{x \} \sqcup \bigsqcup_{y \neq x \in \Sv} W^u(x) \cap W^s(y)$. 
By definition of Morse flow, we have $\Cv \subseteq M - \mathrm{P}(v) \subseteq \mathcal{R}(v) \subseteq \Omega(v) = \Cv$ and so $M = \Sv \sqcup \mathrm{P}(v)$. 
Therefore $\mathrm{P}(v) = \bigsqcup_{x \neq y \in \Sv} W^u(x) \cap W^s(y)$. 
Fix any pair $x \neq y \in \Sv$. 
Then $W^u(x) \cap W^s(y) = \{ z \in \mathrm{P}(v) \mid \alpha(z) = W^u(x), \omega(x) = W^s(y) \}$. 
For any $z  \in W^u(x) \cap W^s(y)$, the abstract weak orbit $[z]$ is the connected component of $\{ z \in \mathrm{P}(v) \mid \alpha(z) = W^u(x), \omega(x) = W^s(y) \} = W^u(x) \cap W^s(y)$ containing $z$. 
This means that the connecting orbit set $W^u(x) \cap W^s(y)$ is a union of abstract weak orbits.  
Since the unstable manifold of a singular point is a union of connecting obit sets, the unstable manifold of a singular point is a union of abstract weak orbits.  
\end{proof}

\subsection{Finiteness of Morse-Smale flows}

We have the following finiteness of Morse-Smale flows.

\begin{proposition}\label{prop:fin_MS}
The following properties hold for a Morse-Smale flow $v$ on a compact manifold $M$: \\
$(1)$ $M/[v] = M/\langle v \rangle = M/[v]_k = M/\langle v \rangle_k$ for any $k \in \Z_{>0}$.
\\
$(2)$ The abstract orbit space $M/\langle v \rangle$ with the partial order $\leq_v$ is an abstract multi-graph with finite vertices such that $M_0 = \Cv$ and $M_1 = \mathrm{P}(v)$, where $M_i$ is the set of point of height $i$ with respect to the partial order $\leq_v$.
\\
$(3)$ The abstract multi-graph $M/\langle v \rangle$ is finite if and only if the flow $v$ is of finite type.
\\
$(4)$ If the flow $v$ is $C^\infty$, then any abstract orbits are embedded submanifolds.
\\
$(5)$ If the flow $v$ is $C^\infty$ and is of finite type, the abstract orbit space $M/\langle v \rangle$ has a stratification $\emptyset = S_{-1} \subseteq S_0 \subseteq \cdots \subseteq S_n = M$ such that $S_i$ is the finite union of abstract orbits whose dimensions are less than $i+1$.
\end{proposition}

\begin{proof}
Since the $\alpha$-limit sets and $\omega$-limit set of any point are hyperbolic closed orbits, Corollary~\ref{cor:k_th} implies that $M/[v]  = M/[v]_k$ and $M/\langle v \rangle = M/\langle v \rangle_k$ for any $k \in \Z_{>0}$.
Since any recurrent point is non-wandering, we have that $M - \mathrm{P}(v)= \mathcal{R}(v) \subseteq \Omega(v)$.
Hyperbolicity implies that $\Omega(v) = \mathop{\mathrm{Cl}}(v)$ consists of finitely many orbits and so that $M = \Cv \sqcup \mathrm{P}(v)$.
By definitions of abstract weak orbit and abstract orbit, we obtain $M/[v] = M/\langle v \rangle$.
Notice that the $\alpha$-limit set and the $\omega$-limit set of any point are contained in the non-wandering set $\Omega(v) = \mathop{\mathrm{Cl}}(v)$ which is the finite union of hyperbolic orbits.
The connectivity of the $\alpha$-limit set and the $\omega$-limit set implies that the $\alpha$-limit set and the $\omega$-limit set of any point are hyperbolic closed orbits.
This implies that both the finite union of stable manifolds of closed orbits and the finite union of unstable manifolds of closed orbits are the whole manifold.
Moreover, there are at most finitely many $\omega$-limit sets and $\alpha$-limit sets.
Fix a point $x \in \mathrm{P}(v)$.
Then the abstract orbit $\langle x \rangle = [x]$ is the connected component of the connecting orbit set $W^u(\alpha(x)) \cap W^s(\omega(x))$ containing $x$.
By definitions, unstable manifolds and stable manifolds are immersed submanifolds.
Then the intersection $W^u(\alpha(x)) \cap W^s(\omega(x))$ is transverse at any point in the interior $M - \partial M$, and the restriction $(W^u(\alpha(x)) \cap W^s(\omega(x))) \cap \partial M$ is transverse in the boundary $\partial M$.
Therefore the intersection $W^u(\alpha(x)) \cap W^s(\omega(x))$ is also an immersed submanifold.
Since the intersection $W^u(\alpha(x)) \cap W^s(\omega(x))$ is the connecting orbit set from the Morse set $W^u(\alpha(x))$ to the Morse set $W^s(\omega(x))$, the abstract orbit $\langle x \rangle$ corresponds to a connected component of the connecting orbit set from the Morse set $W^u(\alpha(x))$ to the Morse set $W^s(\omega(x))$.
This means that the abstract orbit space $M/\langle v \rangle$ with the partial order $\leq_v$ is an abstract multi-graph with finite vertices such that $M_0 = \Cv$ and $M_1 = \mathrm{P}(v)$.
Since the finite union of stable manifolds of closed orbits is the whole manifold $M$, if $v$ is of finite type, then the saturation of the union of \nbds of closed orbits is $M$ and so the quasi-nondegeneracy implies that $M/\langle v \rangle$ is finite.
Since any recurrent orbit is closed, the non-existence of limit cycles implies the converse.
Suppose that $v$ is $C^\infty$.
\cite[Theorem~1]{meyer1968energy} implies the existence of a Lyapunov function of $v$ and so unstable manifolds and stable manifolds are embedded.
Since the intersections of unstable manifolds and stable manifolds are also immersed submanifolds, the intersections are embedded.
This means any abstract orbits are submanifolds.
Suppose that $v$ is of finite type.
Then $M$ consists of finitely many abstract orbits.
Let $S_i$ be the union of abstract orbits whose dimensions are less than $i+1$.
Then a filtration $\emptyset = S_{-1} \subseteq S_0 \subseteq \cdots \subseteq S_n = M$ is a stratification.
\end{proof}

Notice that the connected components of connecting orbit sets are not finite in general.
In fact, there is a smooth Morse-Smale flow $v$ on a closed three dimensional manifold whose abstract orbit space with the partial order $\leq_v$ is an abstract multi-graph with directed edges which have infinitely many connected components.
Indeed, consider a vector field $X_0 = (-x, -y, -z)$ on a unit closed ball $M_0 := \mathbb{D}^3$ in $\R^3$ and a vector field $Y_1 = (-x, -y, z)$ on a solid cylinder $C := [0,1] \times \mathbb{D}^2$ as in Figure~\ref{fig:ball_cyclinder}.
\begin{figure}
\begin{center}
\includegraphics[scale=0.3]{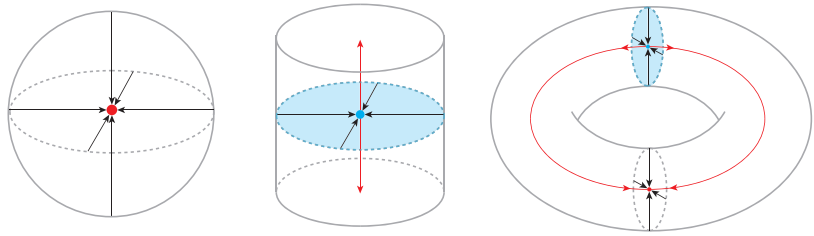}
\end{center}
\caption{Left, the vector field $X_0$; middle, the vector field $Y_1$; right, the vector field $X_1$.}
\label{fig:ball_cyclinder}
\end{figure}
Attaching the handle $C$ to the ball $\mathbb{D}^3$, the resulting space is a solid torus $M_1 := M_0 \cup C$.
Identify $M_1$ with $\R/\Z \times \mathbb{D}^2$.
Smoothing the resulting vector field on $M_1$, we can obtain the resulting vector field $X_1 = (\sin (x/2\pi), -y, -z)$ on $M_1$ as on the right of  Figure~\ref{fig:ball_cyclinder}.
Replacing the saddle with the index one of $Y_1$ by a pair of a source and a hyperbolic periodic orbit $\gamma$ with the index one, denote by $Y'_1$  the resulting vector field as on the left of Figure~\ref{fig:ball_cyclinder02}.
\begin{figure}
\begin{center}
\includegraphics[scale=0.3]{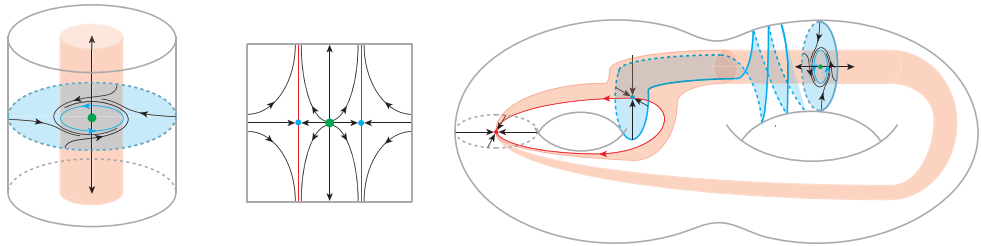}
\end{center}
\caption{Left, the vector field $Y'_1$; middle, the projection of the vector field $Y'_1$ into the $x$-$y$ plane; right, the vector field $X_2$.}
\label{fig:ball_cyclinder02}
\end{figure}
Attaching a copy $C'$ of the handle $C$ with the vector field $Y_1$ to the solid torus $M_1$, the resulting space is denote by $M_2 := M_0 \cup C \cup C'$ as in Figure~\ref{fig:ball_cyclinder02+}.
\begin{figure}
\begin{center}
\includegraphics[scale=0.3]{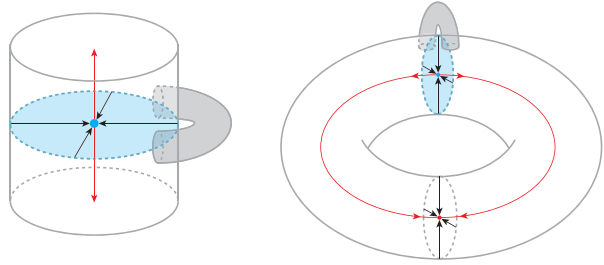}
\end{center}
\caption{Left, attaching a copy $C'$ of the handle $C$ with the vector field $Y_1$; right, attaching a copy $C'$ of the solid torus $M_0 \cup C$ with the vector field $X_1$.}
\label{fig:ball_cyclinder02+}
\end{figure}
Smoothing the resulting vector field on $M_2$, we can obtain the resulting vector field $X_2$ on $M_2$ as on the right of Figure~\ref{fig:ball_cyclinder02}.
Let $s_1$ be the saddle with index one of $X_2$ and $O$ the hyperbolic periodic orbit of $X_2$.
Then the stable manifold $W^s(s_1)$ and the unstable manifold $W^u(O)$ intersect transversally.
Here the transversality of a pair of differential submanifolds means that the sum of the tangent spaces of $W^s(s_1)$ and $W^u(O)$ at a point $x$ in the intersection $W^s(s_1) \cap W^u(O)$ is the tangent space of one of the whole space $M_2$ (i.e. $T_x W^s(s_1) + T_x W^u(O) = T_x M_2$).
On the other hand, attaching two copies of the handle $C$ with the vector field $-Y_1$ to a copy of the ball $\mathbb{D}^3$ with the vector field $-X_0$ as in Figure~\ref{fig:ball_cyclinder05+}, the resulting space is denoted by $M'_1$.
\begin{figure}
\begin{center}
\includegraphics[scale=0.275]{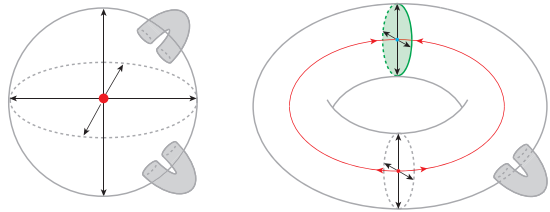}
\end{center}
\caption{Left, attaching two copies of $C$ to a copy of the ball $\mathbb{D}^3$ with the vector field $-X_0$; right, attaching a copy of $C$ to the union of a copy of the ball $\mathbb{D}^3$ with the vector field $-X_0$ and a copy of $C$ with the vector field $-Y_1$.}
\label{fig:ball_cyclinder05+}
\end{figure}
Smoothing the resulting vector field on $M'_1$, we can obtain the resulting vector field $X'_1$ on $M'_1$ as on the right of Figure~\ref{fig:ball_cyclinder05}.
Denote by $s_2$ a saddle of index two of $X'_1$ on $M'_1$.
\begin{figure}
\begin{center}
\includegraphics[scale=0.275]{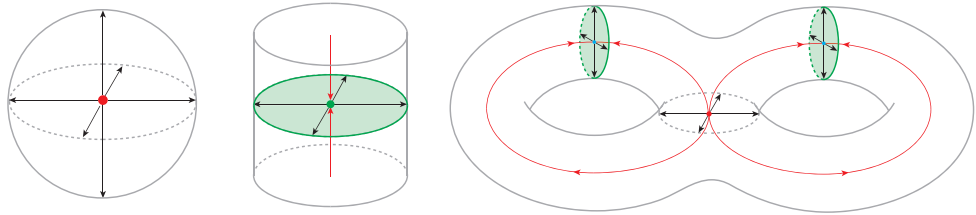}
\end{center}
\caption{Left, the vector field $-X_0$; middle, the vector field $-Y_1$; right,  the vector field $X'_1$.}
\label{fig:ball_cyclinder05}
\end{figure}
Let $\Sigma_2$ be an orientable closed surface of genus two and $f \colon \Sigma_2 \to \Sigma_2$ be a $C^\infty$ diffeomorphism which maps two simple closed curves as in Figure~\ref{fig:diffeo_s2s}.
\begin{figure}
\begin{center}
\includegraphics[scale=0.4]{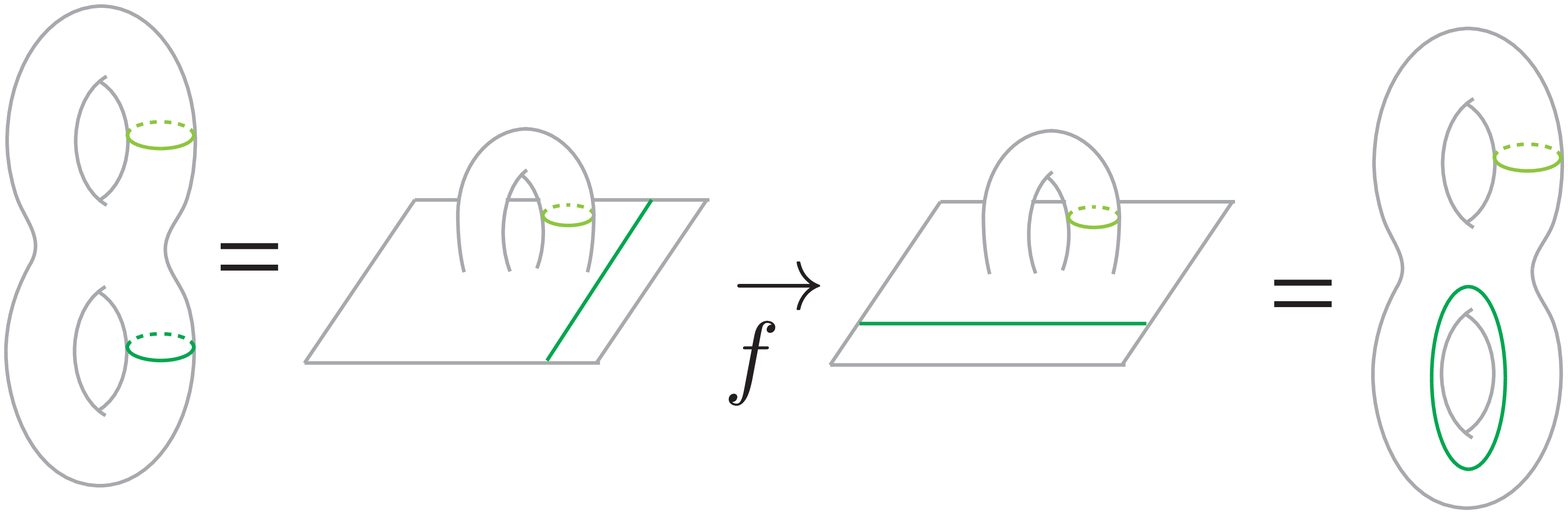}
\end{center}
\caption{A $C^\infty$ diffeomorphism on $\Sigma_2$.}
\label{fig:diffeo_s2s}
\end{figure}
Pasting $M_2$ and $M'_1$ by the diffeomorphism $f$ such that $W^u(s_2)$ and $W^s(s_1)$ intersect transversally and infinitely many times, and smoothing the resulting vector field on the resulting manifold $M := M_2 \cup M'_1$ as in Figure~\ref{fig:ball_cyclinder04}, the resulting manifold $M$ is a three-dimensional closed manifold and the resulting vector field $X$ is a $C^\infty$ Morse-Smale vector field with connecting orbit sets which have infinitely many connected components such that the $\Omega$-limit set is the set of closed orbit which consists of one sinks, three saddles, one limit cycle, and two source $\gamma$, because the unstable manifold $W^u(s_2)$ and the stable manifold $W^s(s_1)$ intersect transversally on the boundary $\partial M_2 = \partial M'_1$ and the vector field $X$ is transverse to $\partial M_2 = \partial M'_1$.
\begin{figure}
\begin{center}
\includegraphics[scale=0.3]{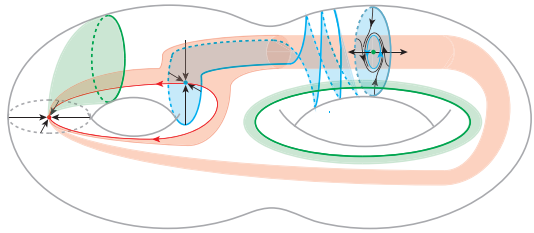}
\end{center}
\caption{A schematic picture of the intersection of the stable manifolds of the saddle $s_1$ and the limit cycle $\gamma$ and the unstable manifolds of the three saddles and the limit cycle $\gamma$ in $M_2$.}
\label{fig:ball_cyclinder04}
\end{figure}

On the other hand, as mentioned above, it is known that the set of the unstable manifolds of singular points of a Morse flow on a closed manifold is a finite CW decomposition~\cite[Theorem~1]{abbondandolo2020stable}.
Therefore the author would like to know whether any Morse flow on a closed manifold $M$ whose abstract orbit space $M/\langle v \rangle$ with the partial order $\leq_v$ is a finite abstract multi-graph.
In other words, does any connecting orbit set between distinct Morse sets of a Morse flow on a closed manifold consists of finitely many connected components?
Equivalently, the author would like to know an answer to the following question.
\begin{question}
Dose the intersection of the unstable manifold of any closed orbit and of the stable manifold of any closed orbit of a Morse flow on a closed manifold consists of finitely many connected components?
\end{question}

\subsection{Finiteness of gradient flows}

We show the following statement.

\begin{proposition}
The following properties hold for a gradient flow with nondegenerate singular points on a compact manifold $M$: \\
$(1)$ $M/[v] = M/\langle v \rangle$
\\
$(2)$ The abstract orbit space $M/\langle v \rangle$ with the partial order $\leq_v$ is an abstract multi-graph with finite vertices such that $M_0 = \Sv$, and $M_1 = \mathrm{P}(v)$, where $M_i$ is the set of point of height $i$ with respect to the partial order $\leq_v$.
\\
$(3)$
The abstract multi-graph $M/\langle v \rangle$ is finite if and only if the flow $v$ is of finite type.
\end{proposition}

\begin{proof}
Gradient property implies that each orbit is either singular or non-recurrent.
Then $\mathop{CR}(v) = \Omega(v) = \Sv$ and so $X = \Sv \sqcup \mathrm{P}(v)$.
By definitions of abstract weak orbit and abstract orbit, we obtain $M/[v] = M/\langle v \rangle$.
Non-degeneracy implies that $\Sv$ consists of finitely many orbits.
Since the $\alpha$-limit set and the $\omega$-limit set of any point are connected and is contained in the non-wandering set $\Omega(v) = \Sv$, the $\alpha$-limit set and the $\omega$-limit set are singular points.
The finiteness of singular points implies the finite existence of $\omega$-limit and $\alpha$-limit sets, and so $M/\langle v \rangle$ is finite if and only if $v$ is of finite type.
As the proof of Proposition~\ref{prop:fin_MS}, the abstract orbit of a non-singular point corresponds to a connected component of the connecting orbit set between distinct singular points.
This means that the abstract orbit space $M/\langle v \rangle$ with the partial order $\leq_v$ is an abstract multi-graph with finite vertices such that $M_0 = \Sv$ and $M_1 = \mathrm{P}(v)$.
\end{proof}

\section{Properties of surface flows}

\subsection{Notion of surface flows}
Recall several concepts to state properties of surface flows.

\subsubsection{Flow of finite type on a surface}\label{def:2dflows}
Let $w$ be a flow on a surface $S$.
A singular point is nondegenerate if it is either a saddle, a $\partial$-saddle, a sink, a $\partial$-sink, a source, a $\partial$-source, or a center.
In other words, a singular point is nondegenerate if and only if it is locally topologically equivalent to an isolated singular point $p$ of a flow generated by a $C^2$ vector field $X$ such that the determinant  of the Hesse matrix $(X_{ij})$ is non-zero (i.e.~$X_{11}X_{22} - X_{12}X_{21}\neq 0$), where $(x_1, x_2)$ is a local coordinate system and $X_{ij} := \partial^2 X(p)/\partial {x_i} \partial {x_j}$.
A flow is quasi-regular if any singular point either is locally topologically equivalent to a nondegenerate singular point or is a multi-saddle as in Figure~\ref{quasi_reg_sing}.
Here a multi-saddle is an isolated singular point as in Figure~\ref{quasi_reg_sing} (see \cite{yokoyama2017decompositions} for details of the definition of multi-saddle).
\begin{figure}
\begin{center}
\includegraphics[scale=0.3]{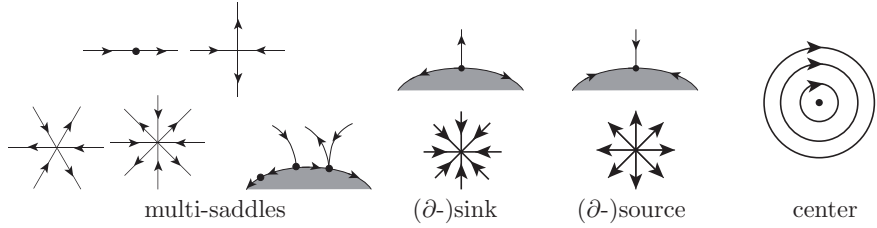}
\end{center}
\caption{The list of singular points appeared in quasi-regular flows.}
\label{quasi_reg_sing}
\end{figure}
A flow $v$ on a surface $S$ is a flow of finite type on a surface if it satisfies the following conditions:
\\
$(1)$ The flow $v$ is quasi-regular.
\\
$(2)$ There are at most finitely many limit cycles.
\\
$(3)$ Any recurrent orbit is closed
(i.e. $S = \mathop{\mathrm{Cl}}(w) \sqcup \mathrm{P}(w)$).
\\
Notice that concepts of ``of finite type'' on topological spaces and on surfaces are different from each other.
Gutierrez's smoothing theorem \cite{gutierrez1986smoothing} implies that each flow of finite type on a compact surface is topologically equivalent to a $C^{\infty}$-flow.
We have the following observation.
\begin{lemma}
A flow $v$ of finite type on a compact surface $S$ is also of finite type as a topological space {\rm (i.e.} $S$ consists of finitely many abstract weak orbits {\rm)}.
In particular, any singular point of $v$ is quasi-nondegenerate.
\end{lemma}

\begin{proof}
Let $v$ be a flow of finite type on a compact surface $S$.
\cite[Lemma 13.2]{yokoyama2017decompositions} implies that $S$ consists of finitely many abstract weak orbits such that each abstract weak orbit is either a  singular point, a limit cycle, a semi-multi-saddle separatrix, a trivial flow box, a periodic annulus, or an open transverse annulus.
\end{proof}



\subsubsection{Transversality for continuous flows on surfaces}

Recall that a curve (or arc) on a surface $S$ is a continuous mapping $C: I \to S$ where $I$ is a nondegenerate connected subset of a circle $\mathbb{S}^1$.
An orbit arc is an arc contained in an orbit.
A curve is simple if it is injective.
We also denote by $C$ the image of a curve $C$.
Denote by $\partial C := C(\partial I)$ the boundary of a curve $C$, where $\partial I$ is the boundary of $I \subset \mathbb{S}^1$. Put $\mathrm{int} C := C \setminus \partial C$.
A simple curve is a simple closed curve if its domain is $\mathbb{S}^1$ (i.e. $I = \mathbb{S}^1$).
A simple closed curve is also called a loop.
A curve $C$ is transverse to $v$ at a point $p \in \mathrm{int} C$ if there are a small neighborhood $U$ of $p$ and a homeomorphism $h:U \to [-1,1]^2$ with $h(p) = 0$ such that $h^{-1}([-1,1] \times \{t \})$ for any $t \in [-1, 1]$ is an orbit arc and $h^{-1}(\{0\} \times [-1,1]) = C \cap U$.
A curve $C$ is transverse to $v$ at a point $p \in \partial C \cap \partial S$ (resp. $p \in \partial C \setminus \partial S$) if there are a small neighborhood $U$ of $p$ and a homeomorphism $h:U \to [-1,1] \times [0,1]$ (resp. $h:U \to [-1,1]^2$) with $h(p) = 0$ such that $h^{-1}([-1,1] \times \{t \})$ for any $t \in [0, 1]$ (resp. $t \in [-1, 1]$) is an orbit arc and $h^{-1}(\{0\} \times [0,1]) = C \cap U$ (resp. $h^{-1}(\{0\} \times [-1,1]) = C \cap U$).
A simple curve $C$ is transverse to $v$ if so is it at any point in $C$.
A simple curve $C$ is transverse to $v$ is called a transverse arc.
Notice that we can also define transversality using tangential spaces, because each flow on a compact surface is topologically equivalent to a $C^1$-flow by the Gutierrez's smoothing theorem.

\subsubsection{Fundamental structures of flows on surfaces}
Let $w$ be a flow on a surface $S$.
An open saturated annulus is a transverse annulus if it consists of non-recurrent orbit which is topologically equivalent to a flow as on the right of Figure~\ref{flow-boxes}.
\begin{figure}
\begin{center}
\includegraphics[scale=0.4]{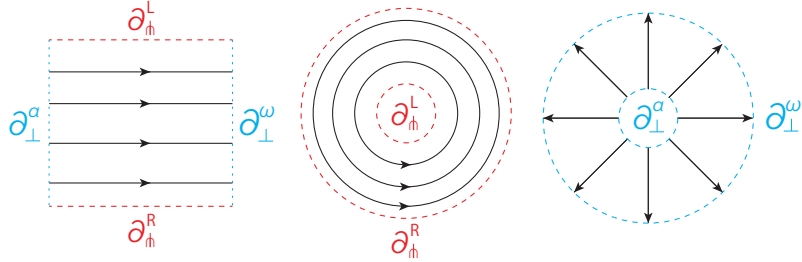}
\end{center}
\caption{A trivial flow box, an open periodic annulus, and an open transverse annulus}
\label{flow-boxes}
\end{figure}
A non-singular orbit of $w$ is a semi-multi-saddle separatrix if it is a separatrix from or to a multi-saddle.
A semi-multi-saddle separatrix of $w$ is a multi-saddle separatrix if it is a separatrix between multi-saddles.
A semi-multi-saddle separatrix of $w$ is an ss-separatrix if it is not between multi-saddles.
Notice that a semi-multi-saddle separatrix of $w$ is an ss-separatrix if and only if it is not a multi-saddle separatrix.
An orbit of the flow $w$ of finite type is an ss-component if it is either a sink, a $\partial$-sink, a source, a $\partial$-source, or a limit circuit.
In other words, an orbit is an ss-component if and only if it is either a semi-repelling $\alpha$-limit set or a semi-attracting $\omega$-limit set.
Here an $\omega$-limit set is semi-attracting if it is either a source, a $\partial$-source, or an attracting limit circuit.
An $\alpha$-limit set is semi-repelling if it is either a sink, a $\partial$-sink, or a  repelling limit circuit.
The union of multi-saddles, semi-multi-saddle separatrices, and ss-components of the flow $w$ of finite type is called the multi-saddle connection diagram and denote by $D(v)$.
A connected component of the multi-saddle connection diagram is called a multi-saddle connection.
Note that the multi-saddle connection diagram of a Hamiltonian vector field is also called the ss-multi-saddle connection diagram and denote by $D_{\mathrm{ss}}(v)$ in \cite{yokoyama2013word}.
A subset of $S$ which is either a torus, a Klein bottle, an open annulus, or an open M\"obius band is periodic if it consists of periodic orbits.
A flow box is homeomorphic to a rectangle $(0,1)^2$, $(0,1) \times [0, 1)$, or $(0,1) \times [0,1]$ such that any open orbit arc is of the form $(0,1) \times \{ y \}$.
By the flow box theorem for a continuous flow on a surface (cf. \cite[Theorem~1.1, p.45]{aranson1996introduction}), for any point, there is its open \nbd which is a flow box.
A flow box is a trivial flow box if it is a saturated disk to which the orbit space of the restriction of the flow is an interval as on the left of Figure~\ref{flow-boxes}.
In the same way, an $\omega'$-limit set is semi-attracting if it is either a source, a $\partial$-source, or a semi-repelling limit circuit.
An $\alpha'$-limit set is semi-repelling if it is either a sink, a $\partial$-sink, or a semi-attracting limit circuit.
Denote by $\partial_{\mathrm{P}(v)}$ the union of separatrices in $\partial S$ from $\partial$-sources to $\partial$-sinks.
Notice that, for a flow of finite type on a surface, a separatrix from a multi-saddle is an ss-separatrix if and only if it is to the semi-attracting $\omega$-limit set, and that a separatrix to a multi-saddle is an ss-separatrix if and only if it is from the semi-repelling $\alpha$-limit set.
A separatrix from or to a saddle is self-connected if each separatrix from and to the saddle.
A separatrix from or to a $\partial$-saddle is self-connected if each separatrix connecting two $\partial$-saddles on a boundary component of the surface (see Figure~\ref{self_conn}).
\begin{figure}
\begin{center}
\includegraphics[scale=0.15]{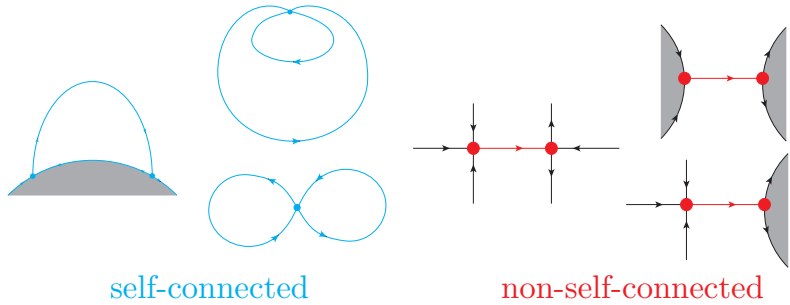}
\end{center}
\caption{Self-connected separatrices and non-self-connected separatrices.}
\label{self_conn}
\end{figure}

\subsubsection{Extended orbit spaces of flows on surfaces}
An extended orbit of a flow on a surface is an equivalence class of an equivalence relation $\sim_{w_{\mathrm{ex}}}$ defined by $x \sim_{w_{\mathrm{ex}}}y$ if they are contained in either an orbit or a multi-saddle connection, and that the extended orbit space of a flow $w$ on a surface $S$ is a quotient space $S/\sim_{w_{\mathrm{ex}}}$ and is denoted by $S/v_{\mathrm{ex}}$.
Note that an extended orbit is an analogous concept of  ``demi-caract\'eristique'' in the sense of Poincar\'e \cite{poincare1881memoire}.
In particular, an extended orbit for a Hamiltonian flow on a compact surface corresponds to a ``demi-caract\'eristique'' in the sense of Poincar\'e.
Moreover, the concept of ``extended positive orbit'' for a flow of finite type on a surface corresponds to one of  ``demi-caract\'eristique'' in the sense of Poincar\'e.

\subsubsection{Extended weak orbit spaces of flows on topological spaces}
We also define a generalization of extended orbits, called extended weak orbits as follows.
A closed connected invariant subset $\mathcal{S}$ consisting of finitely many abstract weak orbits is a quasi-saddle if $(1)$ there are points $x_\alpha, x_\omega \notin \mathcal{S}$ such that $\alpha(x_\alpha) \subseteq \mathcal{S}$ and $\omega(x_\omega) \subseteq \mathcal{S}$, and $(2)$ $|\{ [x] \mid \alpha(x) \subseteq \mathcal{S} \text{ or } \omega(x) \subseteq \mathcal{S} \}|< \infty$.
Then such an abstract weak orbit $[x]$ in the definition of quasi-saddle are called a quasi-saddle-separatrix, and the union of quasi-saddles and quasi-saddle-separatrices is called the quasi-saddle connection diagram.
The connected component of the quasi-saddle connection diagram is called a quasi-saddle connection.
Lemma~\ref{lem:same_limit_set} implies that the condition $(1)$ in the definition of quasi-saddle is equivalent to the following condition: $(1)'$ there are abstract weak orbits $[x_\alpha]$ and $[x_\omega]$ outside of $\mathcal{S}$ such that $\alpha([x_\alpha]) = \bigcup_{y \in [x_\alpha]} \alpha(y) \subseteq \mathcal{S}$ and $\omega([x_\omega]) = \bigcup_{z \in [x_\omega]}\omega(z) \subseteq \mathcal{S}$.
We define the extended weak orbit $[x]_{\mathrm{ex}}$ of a point $x$ by the equivalence class of $\sim_{[w]_{\mathrm{ex}}}$ containing $x$, where $\sim_{[w]_{\mathrm{ex}}}$ is an equivalence class defined by $y \sim_{[w]_{\mathrm{ex}}} z$ if either $[y] = [z]$ or there is a quasi-saddle connection which is closed and contains both $[y]$ and $[z]$.
Then the quotient space $X/\sim_{[w]_{\mathrm{ex}}}$ is denoted by $X/[w]_{\mathrm{ex}}$ and is called the extended weak orbit space of a flow $w$ on a topological space $X$.
Note that the extended weak orbit space of a flow on a topological space is a quotient space of the abstract weak orbit space $X/[w]$.
Moreover,  we will show that the extended weak orbit space of a Hamiltonian flow with finitely many singular points on a compact surface is a quotient space of the extended orbit space (see Lemma~\ref{lem:abs_ext}).


\subsection{Properties of gradient flows on a compact surface}


We have the following statement.

\begin{lemma}\label{lem:sep_multi}
Let $v$ be a flow with finitely many singular points on a compact surface $S$.
Then a point $x$ in a semi-multi-saddle separatrix is not recurrent if and only if $[x] = \langle x \rangle = O(x)$.
\end{lemma}

\begin{proof}
If $x \in \mathrm{R}(v)$, then \cite[Corollary 3.4]{yokoyama2019properness} implies that $O \neq \hat{O} = \langle x \rangle$.
This means that $x$ is not recurrent if $[x] = \langle x \rangle = O(x)$.
Conversely, suppose that $x \in S - \mathrm{R}(v)$.
The orbit $O(x)$ is not closed but $x \in \mathrm{P}(v)$.
By time reversion if necessary, we may assume that $\alpha(x)$ is a multi-saddle $y$.
Since there are at most finitely many singular points on a compact surface, there are at most finitely many separatrices from multi-saddles.
By Lemma~\ref{lem:type}, we have $\langle x \rangle = [x] = [x]'' \subseteq \mathrm{P}(v)$.
Since the subset $C := \{ z \in \mathrm{P}(v) \mid \alpha' (x) = \alpha' (z), \omega'(x) = \omega' (z) \} = \{ z \in \mathrm{P}(v) \mid y = \alpha(z), \omega(x) = \omega(z) \}$ is contained in the set of separatrices from the multi-saddle $y$, the subset $C$ is a finite union of orbits.
For any point $z \in C - O(x)$, we have that $\overline{O(x)} \cap O(z) = (\alpha'(x) \cup O(x) \cup  \omega'(x)) \cap O(z) = (\alpha'(x) \cup \omega'(x)) \cap O(z) = (\alpha'(z) \cup \omega'(z)) \cap O(z) = \emptyset$.
Therefore $\overline{O(x)} \cap C = O(x)$.
We claim that $O(x)$ is a connected component of $C$.
Indeed, fix a Riemannian metric and denote by $d$ the distance function and by $B_r(y) := \{ z \in S \mid d(y,z) \leq r \}$ for any $r > 0$ and $y \in X$.
For any $t \in \R$, there is a closed \nbd $U_t \subseteq B_{1/(|t|+1)}(v_t(x))$ of $v_t(x)$ such that $U_t \cap C \subseteq O(x)$.
Since $O(x)$ is $\sigma$-compact, there are an unbounded increasing sequence $(t_n)_{n \in \Z_{\geq 0}}$ and an unbounded decreasing sequence $(s_n)_{n \in \Z_{\geq 0}}$ with $t_0 = s_0 = 0$ such that the union $U := \bigcup_{n \in \Z_{\geq 0}} U_{t_n} \cup U_{s_n}$ is a \nbd of $O(x)$ with $U \cap C = O(x)$.
Assume that $\overline{U} \cap (C - O(x)) \neq \emptyset$.
Fix a point $y \in \overline{U} \cap (C - O(x))$.
Then $y \in \overline{\bigcup_{n \in \Z_{\geq 0}} U_{t_n}}$ or $y \in \overline{\bigcup_{n \in \Z_{\geq 0}} U_{s_n}}$.
By time reversion if necessary, we may assume that $y \in \overline{\bigcup_{n \in \Z_{\geq 0}} U_{t_n}}$.
The closedness of $U_t$ implies that $y \notin \bigcup_{n = 0}^N U_{t_n}= \overline{\bigcup_{n = 0}^N U_{t_n}}$ for any non-negative integer $N \in \Z_{\geq 0}$.
Then $y \in \bigcap_{N \in \Z_{\geq 0}} \overline{\bigcup_{n \in \Z_{\geq N}} U_{t_n}}$.
Since a metrizable space is Fr\'echet-Urysohn, there is a convergence sequence $(y_n)_{n \in \Z_{\geq 0}}$ of $y_n \in U_{t_n}$ whose limit is $y$.
Since $\lim_{n \to \infty} d(y_n, v_{t_n}(x)) = 0$, we have $y = \lim_{n \to \infty} v_{t_n}(x)$.
This means that $y \in \omega(x) \cap C \subseteq \overline{O(x)} \cap C = O(x)$, which contradicts $y \notin O(x)$.
Therefore $\overline{U} \cap (C - O(x)) = \emptyset$ and so $\overline{U} \cap C = O(x)$.
This implies that $O(x)$ is a connected component of $C$ and so that the connected component $[x]''$ of the subset $C$ containing $x$ is the separatrix $O(x)$ from the multi-saddle $y$.
\end{proof}

Let $v$ be a flow on a compact surface $S$.
Recall that $v$ is a flow of finite type on a surface if it satisfies the following conditions: $(1)$ $v$ is quasi-regular; $(2)$ There are at most finitely many limit cycles; $(3)$ Any recurrent orbit is closed (i.e. $S = \mathop{\mathrm{Cl}}(v) \sqcup \mathrm{P}(v)$).
Note that we require quasi-regularity for this definition because the topological space $X$ is a surface.
We have the following observation.

\begin{lemma}\label{lem:2dg}
Any gradient flow on a compact surface is quasi-regular if and only if it is of finite type.
\end{lemma}

\begin{proof}
Let $v$ be a gradient flow on a compact surface $S$.
By definition, any gradient flow consists of singular points and non-recurrent orbits and so $S = \Sv \sqcup \mathrm{P}(v)$.
This means that there are no limit cycles.
Thus the assertion holds.
\end{proof}

\begin{lemma}\label{lem:three_pts}
Any quasi-regular gradient flow on a closed surface has at least three abstract weak orbits. 
Moreover, the following conditions are equivalent for a quasi-regular gradient flow $v$ on a connected closed surface $S$: \\
$(1)$ The abstract weak orbit space $S/[v]$ consists of three elements. 
\\
$(2)$ The surface $S$ is a sphere which consists of a sink $s_+$, a source $s_-$ and non-recurrent orbits between them. 
\\ 
$(3)$ $D(v) = \emptyset$. 
\end{lemma}

\begin{proof}
Let $v$ be a quasi-regular gradient flow on a connected closed surface $S$.
The existence of a height function with maximal values, minimal values, and non-extremum values implies that there are a sink, a source, and non-singular orbits. 
This means that there are at least three abstract weak orbits. 
Suppose that $S/[v]$ consists of three elements. 
Since $v$ is a quasi-regular gradient flow on a closed surface, there are a sink $s_-$ and a source $s_+$.
Then the third element is the complement $\mathbb{A} := S - \{ s_-, s_+ \}$ of the union of the sink and source.
By Poincar\'e-Hopf theorem for continuous flows with finitely many singular points on compact surfaces (cf. \cite[Lemma 8.2]{yokoyama2017decompositions}), the Euler characteristic of the surface $S$ is two and so $S$ is a sphere.
Therefore the complement $\mathbb{A}$ is an open annulus which consists of non-recurrent orbits connecting $s_-$ and $s_+$. 
This means that the condition $(1)$ implies the condition $(2)$. 
If the surface $S$ is a sphere which consists of a sink $s_+$, a source $s_-$ and non-recurrent orbits between them, then $D(v) = \emptyset$. 
Suppose that $D(v) = \emptyset$. 
By Poincar\'e-Hopf theorem for continuous flows with finitely many singular points on compact surfaces (cf. \cite[Lemma 8.2]{yokoyama2017decompositions}), the non-existence of multi-saddles and existence of sinks and sources imply that the Euler characteristic of the connected closed surface $S$ is positive and so $S$ is a sphere.
Since the Euler characteristic of the sphere $S$ is two, the singular point set $\Sv$ consists of a sink $s_-$ and a source $s_+$. 
The complement $S - \Sv = \mathrm{P}(v)$ is a transverse annulus $\{ x \in S \mid \alpha (x) = s_+, \omega(x) = s_- \}$ which is an abstract weak orbit. 
This means that $S/[v]$ consists of three elements. 
\end{proof}

Recall that denote by $D(v)$ the multi-saddle connection diagram of a flow $v$.
We have the following characterization of non-existence of multi-saddle connection diagram of a quasi-regular gradient flow on a connected compact surface with nonempty boundary.

\begin{lemma}\label{lem:three_pts_bd}
The following conditions are equivalent for a quasi-regular gradient flow $v$ on a connected compact surface $S$ with nonempty boundary: \\
$(1)$ The abstract weak orbit space $S/[v]$ consists of three elements. 
\\
$(2)$ The surface $S$ is a closed disk $\mathbb{D}^2$ and the flow $v$ consists of a $\partial$-sink $s_-$, a $\partial$-source $s_+$ and non-recurrent orbits between them.
\\ 
$(3)$ $D(v) = \emptyset$. 
\end{lemma}

\begin{proof}
Let $v$ be a quasi-regular gradient flow on a compact surface $S$ with nonempty boundary. 
Suppose that the multi-saddle connection diagram $D(v)$ is empty. 
Since each boundary component contains at least two singular points, 
there are one $\partial$-sink and one $\partial$-source. 
By Poincar\'e-Hopf theorem for continuous flows with finitely many singular points on compact surfaces, the non-existence of multi-saddles implies that the Euler characteristic of $S$ is at least one. 
Since $\partial S \neq \emptyset$, the surface $S$ is a closed disk. 
The existence of a height function with maximal values, minimal values, and non-extremum values implies that there are a $\partial$-sink $s_-$, a $\partial$-source $s_+$, and non-singular orbits. 
Therefore $\Sv$ consists of the $\partial$-sink and the $\partial$-source. 
This means that the condition $(2)$ holds. 
Suppose that $S = \mathbb{D}^2$ and the flow $v$ consists of a $\partial$-sink $s_-$, a $\partial$-source $s_+$ and non-recurrent orbits between them.
Then the complement of the union of the $\partial$-sink and the $\partial$-source is 
$\{ x \in \mathrm{P}(v) \mid \alpha(x) = s_-, \omega(x) = s_+ \}$ which is an abstract weak orbit. 
This means that there are exactly three abstract weak orbits. 
Suppose that $S/[v]$ consists of three elements. 
Since $v$ is a gradient flow, there are either $\partial$-sinks or sinks, and there are $\partial$-sources or sources. 
Since each boundary component contains at least two singular points, there are one $\partial$-sink $s_-$ and one $\partial$-source $s_+$. 
Then the complement $S - \{s_-, s_+ \}$ is an abstract weak orbit. 
This means that $D(v) = \emptyset$. 
\end{proof}

We have the following stratifications.

\begin{proposition}\label{prop:01}
The following statements hold for a quasi-regular gradient flow $v$ on a compact connected surface $S$:
\\
$(1)$ The binary relation $\leq_\partial$ for the abstract weak orbit space $S/[v]$ is a partial order.
\\
$(2)$
The quotient map $S \to S/[v]$ is a finite poset-stratification with respect to $\leq_\partial$.
\\
$(3)$ The filtration $\emptyset \subset S_{\leq 0} \subseteq S_{\leq 1}  \subsetneq S_{\leq 2}$ is a stratification with respect to $\leq_\partial$, where $S_{\leq 0} = \Sv$, $S_{\leq 1} = \Sv \cup D(v)$, and $S_{\leq 2} - S_{\leq 1} =  \mathrm{P}(v) \setminus D(v)$.
\\
$(4)$ If there is a multi-saddle, then $S_{\leq i}$ is the set of elements of height at least $i$ with respect to $\leq_\partial$.
\end{proposition}

\begin{proof}
By definition of gradient flow, any gradient flow consists of singular points and non-recurrent orbits and so $S = \Sv \sqcup \mathrm{P}(v)$.
By Lemma~\ref{lem:three_pts}, if $v$ is a spherical flow consisting of one sink $s_\omega$ and one source $s_\alpha$ and one transverse annulus on a sphere, then the assertions hold. 
Lemma~\ref{lem:three_pts_bd} implies that if $S$ is a closed disk $\mathbb{D}^2$ and the flow $v$ consists of a $\partial$-sink $s_-$, a $\partial$-source $s_+$ and non-recurrent orbits between them, then the assertions hold. 
Thus we may assume that there is a multi-saddle.
Lemma~\ref{lem:equivalent} and Lemma~\ref{lem:non_min_abst} imply that $S_{\leq 0} = \Sv$ is the set of height zero points with respect to $\leq_\partial$.
Then the finite union $\Sv \cup D(v)$ of orbits is a closed subset and the union of orbits of height at most one with respect to $\leq_\partial$.
\cite[Corollary 11.3]{yokoyama2017decompositions} implies that the complement $S - (\Sv \cup D(v)) \subseteq \mathrm{P}(v)$ is a finite disjoint union of trivial flow boxes.
The boundary of such a trivial flow box $B$ contains a semi-multi-saddle separatrix $\gamma$ and is a finite union of separatrices and singular points.
Then $\gamma <_\pitchfork B$.
This implies that $S_{\leq 2} - S_{\leq 1} = \mathrm{P}(v) \setminus D(v)$ is the union of orbits of height two with respect to $\leq_\partial$.
Then the binary relation $\leq_\partial$ on $S/[v]$ is the partial order and also the specialization order of $S/[v]$.
The finiteness of $S/[v]$ implies that the quotient map $S \to S/[v]$ is a finite poset-stratification with respect to $\leq_\partial$.
\end{proof}


\subsection{Properties of Hamiltonian flows on surfaces}

Recall several concepts for Hamiltonian flows on surfaces to state a reduction of the abstract weak orbit space of a Hamiltonian flow into the Reeb graph.

\subsubsection{Fundamental notion of Hamiltonian flows on a compact surface}

A $C^r$ vector field $X$ for any $r \in \Z_{\geq0}$ on an orientable surface $S$ is Hamiltonian if there is a $C^{r+1}$ function $H \colon S \to \mathbb{R}$ such that $dH= \omega(X, \cdot )$ as a one-form, where $\omega$ is a volume form of $S$.
In other words, locally the Hamiltonian vector field $X$ is defined by $X = (\partial H/ \partial x_2, - \partial H/ \partial x_1)$ for any local coordinate system $(x_1,x_2)$ of a point $p \in S$.
Note that a volume form on an orientable surface is a symplectic form.
It is known that a $C^r$ ($r \geq 1$) Hamiltonian vector field on a compact surface is structurally stable with respect to the set of $C^r$ Hamiltonian vector fields if and only if both each singular point is nondegenerate and each separatrix is self-connected (see  \cite[Theorem 2.3.8, p. 74]{ma2005geometric}).
A flow is Hamiltonian if it is topologically equivalent to a flow generated by a Hamiltonian vector field.
Note that the multi-saddle connection diagram $D(v)$ of a Hamiltonian flow $v$ on a compact surface is the union of multi-saddles and separatrices between them.
We show the following statements.

\begin{lemma}\label{lem:ham_trivial}
The following statement holds for a structurally stable Hamiltonian vector field $v$ on a compact connected surface $S$:
\\
$(1)$ If the multi-saddle connection $D(v)$ is empty, then the abstract weak orbit space $S/[v]$ consists of at most three elements.
\\
$(2)$ $|S/[v]| = 1$ if and only if $S$ is a closed periodic annulus.
\\
$(3)$ $|S/[v]| = 2$ if and only if $S$ is a closed center disk.
\\
$(4)$ $|S/[v]| = 3$ if and only if $S$ is a rotating sphere {\rm(i.e.} consists of two centers and one open periodic annulus {\rm)}.
\end{lemma}

\begin{proof}
Let $v$ be a flow generated by a structurally stable Hamiltonian vector field on a compact surface $S$.
\cite[Theorem 2.3.8, p. 74]{ma2005geometric} implies that each singular point is nondegenerate and each separatrix is self-connected.
Since $v$ is generated by a Hamiltonian vector field, we have $S = \mathop{\mathrm{Cl}}(v) \sqcup \mathrm{P}(v)$ and it has neither sinks, $\partial$-sinks, sources, nor $\partial$-sources, and so each singular point is either a center, a saddle, or a $\partial$-saddle.
This implies that each boundary component which is not a periodic orbit contains at least two $\partial$-saddles and two multi-saddle separatrices, and so contains at least four abstract weak orbits.
Moreover, since $S$ is a surface, the flow $v$ contains at least one periodic annulus. 
Therefore if the multi-saddle connection $D(v)$ is empty, then the abstract weak orbit space $S/[v]$ consists of one periodic annulus and at most two centers.
If $S$ is a closed periodic annulus $U$ (resp. closed center disk, rotating sphere), then $|S/[v]| = 1$ (resp. 2, 3).
Conversely, suppose that $|S/[v]| \leq 3$.
Then each boundary component is a periodic orbit.
if $|S/[v]| = 1$, then $S = U$ is a closed periodic annulus.
If $|S/[v]| = 2$ (resp $3$), then the generalization of the Poincar\'e-Bendixson theorem for a flow with finitely many singular points implies that $\overline{U} - U$ is a singular point (resp. two singular points) and so one abstract weak orbit (resp. two abstract weak orbits).
\end{proof}

\begin{lemma}\label{lem:ham}
Any Hamiltonian flow $v$ on a compact connected surface $S$ has at most finitely many singular points if and only if it is of finite type.
In any case, the following statements hold:
\\
$(1)$ Each singular point is either a center or a multi-saddle.
\\
$(2)$
The multi-saddle connection diagram $D(v)$ is a finite union of non-recurrent orbits and of multi-saddles.
\\
$(3)$ The complement $S - (\Sv \cup D(v)) = \Pv$ consists of finitely many periodic annuli each of which is an abstract weak orbit.
\\
$(4)$ The binary relation $\leq_\partial$ for the abstract weak orbit space $S/[v]$ is a partial order.
\\
$(5)$
The quotient map $S \to S/[v]$ is a finite poset-stratification with respect to $\leq_\partial$.
\\
$(6)$ The filtration $\emptyset \subseteq S_{\leq 0} \subseteq S_{\leq 1}  \subsetneq S_{\leq 2}$ is a stratification with respect to $\leq_\partial$, where $S_{\leq 0} = \Sv$, $S_{\leq 1} = \Sv \cup D(v)$, and $S_{\leq 2} - S_{\leq 1}  =  \Pv$.
\\
$(7)$ If there is a multi-saddle, then $S_{\leq i}$ is the set of elements of height at least $i$ with respect to $\leq_\partial$.
\end{lemma}

\begin{proof}
By definition, a flow of finite type has at most finitely many singular points.
Suppose that a Hamiltonian flow on a compact surface has at most finitely many singular points.
Since the whole space is two dimensional and any connected closed $1$-manifold is a circle, the existence of Hamiltonian implies each orbit is proper.
Since a Hamiltonian flow on a compact surface is topologically equivalent to a divergence-free flow, divergence-free property implies the non-existence of limit cycles.
Divergence-free flows on a compact manifold have no non-wandering domains and so are non-wandering.
By  \cite[Theorem 3]{cobo2010flows}, each singular point of a non-wandering flow with finitely many singular points on a compact surface is either a center or a multi-saddles.
This implies that the flow is quasi-regular and so of finite type.
Moreover, the multi-saddle connection diagram $D(v)$ contains the union $\mathrm{P}(v)$ of non-recurrent orbits and consists of finitely many orbits.
Therefore the complement $S - (\Sv \cup D(v)) = \Pv$ consists of finitely many periodic annuli each of which is an abstract weak orbit (see Figure~\ref{flow-boxes}).
%
Then its boundary as a subset consists of centers, multi-saddles, and separatrices.
Suppose that the boundary $\partial \Pv$ consists of at most two centers $c_i$.
Lemma~\ref{lem:ham_trivial} implies that $D(v) = \emptyset$ and that $c_i <_\pitchfork \mathbb{A}$ if $c_i$ exists.
Suppose that the boundary $\partial \mathbb{A}$ contains multi-saddles or separatrices.
Since any separatrix is a multi-saddle separatrix, the boundary $\partial \mathbb{A}$ contains multi-saddles and separatrices.
Then $S_{\leq 0} = \Sv$ is the set of height zero points with respect to $\leq_\partial$, $S_{\leq 1} = \Sv \cup D(v)$ is the set of points of height at most one with respect to $\leq_\partial$, and $S_{\leq 2} - S_{\leq 1} = \Pv$ is the set of height two points with respect to $\leq_\partial$.
Moreover, the binary relation $\leq_\partial$ on $S/[v]$ is the partial order and also the specialization order of $S/[v]$.
The finiteness of $S/[v]$ implies that the quotient map $S \to S/[v]$ is a finite poset-stratification with respect to $\leq_\partial$.
\end{proof}

We have the following reduction by collapsing open periodic annuli into singletons

\begin{lemma}\label{lem:abs_ext}
The extended weak orbit space of a Hamiltonian flow with finitely many singular points on a compact surface is a quotient space of the extended orbit space.
Moreover, the extended weak orbit space corresponds to the abstract graph of the extended orbit space.
\end{lemma}

\begin{proof}
Let $v$ be a Hamiltonian flow with finitely many singular points on a compact surface $S$.
Lemma~\ref{lem:ham} implies that each singular point is either a center or a multi-saddle, and that the complement of the multi-saddle connection diagram $D(v)$ consists of periodic annuli.
Then such periodic annuli are abstract weak orbits.
Therefore it suffices to show that any multi-saddle connection is both an extended orbit and an extended weak orbit.
Then the multi-saddle connection diagram consists of finitely many closed multi-saddle connections.
Therefore any quasi-saddle is a multi-saddle and so the quasi-saddle connection diagram corresponds to the multi-saddle connection diagram.
Since any multi-saddle connections are closed, any extended weak orbit corresponds to an extended orbit.
Then the assertion holds, by collapsing connected components of $S - D(v)$, which are periodic annuli, into singletons.
\end{proof}

We show that the Reeb graph of a Hamiltonian flow with finitely many singular points on a compact surface is a reduction of the abstract weak orbit space as an abstract graph.

\begin{theorem}
The Reeb graph of a Hamiltonian flow with finitely many singular points on a compact surface corresponds to its extended weak orbit space as an abstract graph.
\end{theorem}

\begin{proof}
Let $v$ be a Hamiltonian flow with finitely many singular points on a compact surface $S$ and $H$ the Hamiltonian generating $v$.
We may assume that each critical point has distinct critical value.
Lemma~\ref{lem:ham} implies that each singular point is either a center or a multi-saddle.
Then the multi-saddle connection diagram consists of finitely many orbits and any multi-saddle connections are closed.
This means that a connected component of the inverse image $H^{-1}(c)$ for any $c \in \R$ is either a periodic orbit or a closed multi-saddle connection.
Therefore the Reeb graph of $v$ corresponds to the extended orbit space $S/v_{\mathrm{ex}}$ as a graph.
On the other hand, Lemma~\ref{lem:abs_ext} implies that the extended weak orbit space $S/[v]_{\mathrm{ex}}$
is the abstract graph of the extended orbit space $S/v_{\mathrm{ex}}$.
This means that the abstract graph of the Reeb graph of $v$ corresponds to the extended weak orbit space $S/[v]_{\mathrm{ex}}$.
\end{proof}


\section{Flows of finite type on compact surfaces}
In this section, assume that $v$ is a flow of finite type on a compact surface $S$.
Recall that $\partial_{\mathrm{P}(v)}$ is the union of separatrices in $\partial S$ from $\partial$-sources to $\partial$-sinks.
By \cite[Corollary 10.2]{yokoyama2017decompositions}, the strict border point set $\mathop{\mathrm{Bd}}(v)$ of a flow $v$ of finite type on a compact surface is the union of the multi-saddle connection diagram $D(v)$, centers, periodic orbits on the boundaries, and $\partial_{\mathrm{P}(v)}$.
We have the following statements.

\begin{proposition}\label{prop:five}
The abstract weak orbit space of a flow of finite type on a compact surface is a finite topological space such that any element is one of the five following subsets:
\\
$(1)$ A singular point.
\\
$(2)$ A semi-multi-saddle separatrix.
\\
$(3)$ A connected component $C$ of $\Pv$ such that the transverse boundary $\partial_\pitchfork C = \overline{C} - C$ consists of centers and non-trivial circuits in the multi-saddle connection diagram $D(v)$.
\\
$(4)$ A trivial flow box $\mathbb{D}$ which is a connected component of $S - D(v)$ 
such that $\omega(x) = \omega(y)$ and $\alpha(x) = \alpha(y)$ for any points $x, y \in \mathbb{D}$ and that the $\omega$-limit sets are semi-repelling and the $\alpha$-limit sets are semi-attracting.
Moreover, the transverse boundary $\partial_\pitchfork \mathbb{D} = \overline{\mathbb{D}} - \mathop{\downarrow}_{\leq_v}\mathbb{D} = \overline{\mathbb{D}} - (\bigcup_{x \in \mathbb{D}} \overline{O(x)})$ consists of multi-saddles and semi-multi-saddle separatrices.
\\
$(5)$ A transverse annulus $\mathbb{A}$ which is a connected component of $S - D(v)$ such that $\omega(x) = \omega(y)$ and $\alpha(x) = \alpha(y)$ for any points $x, y \in \mathbb{A}$ and that the $\omega$-limit sets are semi-repelling and the $\alpha$-limit sets are semi-attracting. Moreover, $\overline{\mathbb{A}} = \mathbb{A} \sqcup(\alpha(x) \cup \omega(x))$ for any $x \in \mathbb{A}$.
\end{proposition}

\begin{proof}
Let $v$ be a flow of finite type on a compact surface $S$ and $\mathop{\mathrm{Sing}}(v)_c$ the set of centers. 
Then $S = \Sv \sqcup \Pv \sqcup \mathrm{P}(v)$.
Since $\partial_{\mathrm{P}(v)}$ is the union of separatrices in $\partial S$ from $\partial$-sources to $\partial$-sinks, we obtain $\partial_{\mathrm{P}(v)} \subset \mathrm{int} \mathrm{P}(v)$ and so that each orbit in $\partial_{\mathrm{P}(v)}$ is contained in the boundary of some trivial flow box. 
By \cite[Corollary 10.2]{yokoyama2017decompositions}, we have  $\mathop{\mathrm{Bd}}(v) = \mathop{\mathrm{Sing}}(v)_c \sqcup D(v) \sqcup (\mathrm{int} \Pv \cap \partial S) \sqcup \partial_{\mathrm{P}(v)}$.
By \cite[Lemma 8.28]{yokoyama2017decompositions}, the strict border point set $\mathop{\mathrm{Bd}}(v)$ of the flow $v$ of finite type is closed and becomes the finite disjoint union of orbits.
Therefore the complement $S - \mathop{\mathrm{Bd}}(v)$ has finitely many connected components.
\cite[Theorem 11.2]{yokoyama2017decompositions} implies that the complement $S - \mathop{\mathrm{Bd}}(v)$ consists of periodic tori, periodic Klein bottles, periodic open annuli, periodic open M\"obius bands, trivial flow boxes, and transverse annuli. 
Moreover, $\omega$-limit $(\mathrm{resp.}$ $\alpha$-limit$)$ sets of points in one of trivial flow boxes and transverse annuli are semi-attracting $(\mathrm{resp.}$ semi-repelling) and correspond to each other.
In addition, the boundary of any connected component of the complement $S - \mathop{\mathrm{Bd}}(v)$ is a finite union of closed orbits and separatrices.
Since any singular point in $\partial \Pv$ is either a center or a multi-saddle, the boundary of any connected component of $(S - \mathop{\mathrm{Bd}}(v))$ which is contained in $\partial \Pv$ consists of centers and non-trivial circuits in $D(v)$.
On the other hand, the difference $\mathop{\mathrm{Bd}}(v) - D(v)$ is a finite union of centers, periodic orbits on the boundaries, and $\partial_{\mathrm{P}(v)}$ (i.e. $\mathop{\mathrm{Bd}}(v) - D(v) = \mathop{\mathrm{Sing}}(v)_c \sqcup (\mathrm{int} \Pv \cap \partial S) \sqcup \partial_{\mathrm{P}(v)}$). 
Then the complement $S - D(v)$ is a finite union of periodic tori, periodic Klein bottles, center disks, periodic annuli, periodic M\"obius bands, trivial flow boxes, and transverse annuli. 
Fix a point $x \in X$.
Suppose that $x$ is singular.
Lemma~\ref{lem:awo} implies that the saturated subset $[x]$ is the connected component of $\Sv$ containing $x$.
Since $v$ is of finite type, the abstract weak orbit $[x]$ is the singleton $\{ x \}$ and so is in the case $(1)$.
Suppose that $x$ is periodic.
Lemma~\ref{lem:awo} implies that the saturated subset $[x]$ is the connected component of $\Pv$ containing $x$.
\cite[Corollary 10.3]{yokoyama2017decompositions} implies that
the transverse boundary $\partial_\pitchfork [x] = [x] - [x] \subseteq \mathop{\mathrm{Bd}}(v) \setminus \Pv = \mathop{\mathrm{Sing}}(v)_c \sqcup D(v) \sqcup \partial_{\mathrm{P}(v)}$ consists of centers and non-trivial circuits in $D(v)$.
Therefore the abstract weak orbit $[x]$ is in the case $(3)$.
Thus we may assume that $x \in S - \mathop{\mathrm{Cl}}(v) = \mathrm{P}(v)$.
By definition, Lemma~\ref{lem:type} implies that $O(x) \subseteq [x] = [x]''$.
Suppose that $O(x)$ is a separatrix from a multi-saddle $y$.
Then $\alpha(x) = \alpha'(x) = y$.
Since $y$ is a multi-saddle, the subset $\{ z \in \mathrm{P}(v) \mid \alpha' (x) = \alpha' (z), \omega'(x) = \omega' (z) \} = \{ z \in \mathrm{P}(v) \mid y = \alpha(z), \omega(x) = \omega (z) \}$ is contained in the finite disjoint union of separatrices  each of whose connected components is one separatrix.
This means that $[x]'' = O(x)$ is a separatrix from $y$.
Therefore the abstract weak orbit $[x]$ is a semi-multi-saddle separatrix and so is in the case $(2)$.
By symmetry, the abstract weak orbit of a separatrix to a multi-saddle is the orbit.
This means that the abstract weak orbit $[x]$ is in the case $(2)$.
Thus we may assume that $x$ is not a semi-multi-saddle separatrix and so $x \in (S - (\mathop{\mathrm{Bd}}(v) - \partial_{\mathrm{P}(v)})) \cap \mathrm{P}(v) = S - D(v)$.
Then there is a connected component $U \subseteq \mathrm{P}(v)$ of $S - (\mathop{\mathrm{Bd}}(v) - \partial_{\mathrm{P}(v)}) = S - (\mathop{\mathrm{Sing}}(v)_c \sqcup D(v) \sqcup (\mathrm{int} \Pv \cap \partial S))$ which contains $x$ and is an open subset which is either a trivial flow box or a transverse annulus such that, for any $y \in U$, the $\omega$-limit sets $\omega(x) = \omega(y)$ are semi-attracting and the $\alpha$-limit sets $\alpha(x) = \alpha(y)$ are semi-repelling. 
Since centers are contained in $\mathrm{init} \Cv$, the open subset $U \subseteq \mathrm{P}(v)$ is a connected component of $S - D(v)$. 
Then for any $y \in U$ there is no point $z \in S - O(y)$ with $y \in \overline{O(z)}$.
By definition of $U$, if $U \cap \partial_{\mathrm{P}(v)}$, then $\alpha(U) = \bigcup_{y \in U} \alpha(y)$ is a $\partial$-source and $\omega(U) = \bigcup_{y \in U} \omega(y)$ is a $\partial$-sink.
This means that $\partial_\pitchfork U \cap \partial_{\mathrm{P}(v)} = \emptyset$.
Suppose that $U$ is a transverse annulus.
Then $\omega(x) = \omega'(x)$ is either a sink or a limit circuit and $\alpha(x) = \alpha'(x)$ is either a source or a limit circuit.
In other words, the $\omega$-limit sets $\omega(x)= \omega'(x)$ are semi-attracting and the $\alpha$-limit sets $\alpha(x) = \alpha'(x)$ are semi-repelling.
This implies that the open subset $U$ is a connected component of $S - D(v)$ and that $\partial_\perp U = \partial U = \alpha(x) \cup \omega(x) \subseteq D(v)$ and so $\overline{U} = U \sqcup (\alpha(x) \cup \omega(x))$.
Since the intersection $\alpha(D(v) \cap \mathrm{P}(v))$ consists of multi-saddles, we obtain $U \cap \partial U = \emptyset$. 
Therefore $U \subseteq \{ y \in \mathrm{P}(v) \mid \omega'(x) = \omega'(y), \alpha'(x) = \alpha'(y) \} = \{ y \in \mathrm{P}(v) \mid \omega(x) = \omega(y), \alpha(x) = \alpha(y) \}$ and $\{ y \in \mathrm{P}(v) \mid \omega'(x) = \omega'(y), \alpha'(x) = \alpha'(y) \} \cap \partial U = \emptyset$. 
This means that the open subset $U$ is a connected component of $\{ y \in \mathrm{P}(v) \mid \omega'(x) = \omega'(y), \alpha'(x) = \alpha'(y) \}$ containing $x$ and so $U = [x]'' = [x]$.
Therefore the abstract weak orbit $[x]$ is in the case $(5)$.
Thus we may assume that $U$ is a trivial flow box.
Then $\omega(x) = \omega'(x)$ is either a sink, a $\partial$-sink, or a limit circuit and $\alpha(x) = \alpha'(x)$ is either a source, $\partial$-source, or a limit circuit.
In other words, the $\omega'$-limit sets $\omega'(x)= \omega(x)$ are semi-attracting and the $\alpha'$-limit sets $\alpha'(x) = \alpha(x)$ are semi-repelling.
In particular, each of the $\omega'$-limit sets $\omega(x)= \omega'(x)$ and the $\alpha'$-limit sets $\alpha(x) = \alpha'(x)$ neither is empty nor is a multi-saddle.
We claim that $\{ y \in \mathrm{P}(v) \mid \omega'(x) = \omega'(y), \alpha'(x) = \alpha'(y) \} \cap \partial_\perp U = \emptyset$.
Indeed, assume that there is a point $z \in \{ y \in \mathrm{P}(v) \mid \omega'(x) = \omega'(y), \alpha'(x) = \alpha'(y) \} \cap \partial_\perp U = \{ y \in \mathrm{P}(v) \mid \omega(x) = \omega(y), \alpha(x) = \alpha(y) \} \cap (\alpha(x) \cup \omega(x))$.
A generalization of the Poincar\'e-Bendixson theorem for a flow with finitely many singular points (cf.  \cite[Theorem 2.6.1]{nikolaev1999flows} and  \cite[Corollary 6.6]{yokoyama2017decompositions}) implies that $\alpha(x)$ and $\omega(x)$ are either singular points or limit circuits.
If $z \in \Cv$, then $z \in \alpha(x) \cup \omega(x) = \alpha'(x) \cup \omega'(x) =  \alpha'(z) \cup \omega'(z) = \emptyset$, which is a contradiction.
Thus the non-singular point $z$ is contained in a limit circuit, and so $z \in \mathrm{P}(v)$ and either $\alpha(x)$ or $\omega(x)$ is a limit circuit.
Then $\alpha(x) \cup \omega(x) = \alpha'(x) \cup \omega'(x) = \alpha'(z) \cup \omega'(z) \subset \Sv$ contains no limit circuits, which contradicts the existence of limit circuits contained in $\alpha(x) \cup \omega(x)$.
Suppose that $\overline{U} = U \sqcup (\alpha(x) \cup \omega(x)) = U \sqcup \partial_\perp U$.  
The claim implies that $U = \overline{U} \cap \{ y \in \mathrm{P}(v) \mid \omega'(x) = \omega'(y), \alpha'(x) = \alpha'(y) \}$.
Then the open subset $U$ is a connected component of $\{ y \in \mathrm{P}(v) \mid \omega'(x) = \omega'(y), \alpha'(x) = \alpha'(y) \}$ containing $x$ and so $U = [x]'' = [x]$.
Then the abstract weak orbit $[x]$ is in the case $(4)$.
Thus we may assume that $\overline{U} - (U \sqcup \partial_\perp U) \neq \emptyset$.
By $\partial_\pitchfork U \sqcup \partial_\perp U = \partial U \subseteq D(v)$, the transverse boundary $\partial_\pitchfork U \subseteq D(v)$ consists of multi-saddles and semi-saddle separatrices.
Since the $\omega'$-limit set and the $\alpha'$-limit set of a multi-saddle are  empty, and since either of the $\omega'$-limit set and the $\alpha'$-limit set of any separatrix in $\partial_\pitchfork U$ is a multi-saddle in $\partial U - (\alpha(x) \cup \omega(x))$, the transverse boundary $\partial_\pitchfork U$ does not intersect $\{ y \in \mathrm{P}(v) \mid \omega'(x) = \omega'(y), \alpha'(x) = \alpha'(y) \} = \{ y \in \mathrm{P}(v) \mid \omega(x) = \omega(y), \alpha(x) = \alpha(y) \}$.
By $\partial_\pitchfork U \cap \{ y \in \mathrm{P}(v) \mid \omega(x) = \omega(y), \alpha(x) = \alpha(y) \} = \emptyset$, we have $U = \overline{U} \cap \{ y \in \mathrm{P}(v) \mid \omega(x) = \omega(y), \alpha(x) = \alpha(y) \}$.
Thus the open subset $U$ is a connected component of $\{ y \in \mathrm{P}(v) \mid \omega(x) = \omega(y), \alpha(x) = \alpha(y) \}$ containing $x$ and so $U = [x]'' = [x]$.
Therefore the abstract weak orbit $[x]$ is in the case $(4)$.
\end{proof}

\begin{proposition}
The binary relation $\leq_\pitchfork$ for a flow of finite type on a compact surface is a partial order at most height one.
\end{proposition}

\begin{proof}
Let $v$ be a flow of finite type on a compact surface $S$.
Proposition~\ref{prop:five} implies that each abstract weak orbit is either a singular point, a semi-multi-saddle separatrix, a connected component of $\Pv$, a transverse annulus, or a trivial flow box.
Since $\Sv$ is finite, each singular point is minimal with respect to $\leq_\pitchfork$.
Fix a point $x \in S$.
If $x$ is singular, then it is minimal with respect to $\leq_\pitchfork$.
Therefore $\mathrm{ht}_{\leq_\pitchfork}(\Sv) \leq 0$.
Suppose that $O(x)$ is a semi-multi-saddle separatrix.
Lemma~\ref{lem:sep_multi} implies that the abstract weak orbit of a separatrix is itself and so $[x] = O(x)$.
Therefore $\partial_\pitchfork [x] = \overline{O(x)} - \overline{O(x)} = \emptyset$ and so $x$ is minimal with respect to $\leq_\pitchfork$.
Suppose that $x \in \Pv$.
Then $[x]$ is the connected component $C$ of $\Pv$ containing $x$ such that the transverse boundary $\partial_\pitchfork [x] = \overline{C} - C$ consists of centers and circuits in $D(v)$.
Since $D(v)$ consists of multi-saddles and separatrices between multi-saddles, it consists of minimal points with respect to $\leq_\pitchfork$.
Therefore the transverse boundary $\partial_\pitchfork [x] = \overline{C} - C$ consists of points whose heights are zero.
This means that the height of $x$ is at most one.
Therefore $\mathrm{ht}_{\leq_\pitchfork}(\Pv) \leq 1$.
Suppose that $x$ is contained in a transverse annulus $\mathbb{A} \subseteq \mathrm{P}(v)$ which is a connected component of $S - \mathop{\mathrm{Bd}}(v)$.
Then $\overline{\mathbb{A}} = \mathbb{A} \sqcup (\alpha(x) \cup \omega(x))$.
Therefore the transverse boundary $\partial_\pitchfork [x] = \overline{\mathbb{A}} - \bigcup_{y \in \mathbb{A}} \overline{O(y)} = \overline{\mathbb{A}} - \overline{\mathbb{A}} = \emptyset$.  This means that $x$ is minimal with respect to $\leq_\pitchfork$.
Suppose that $x$ is contained in a trivial flow box $\mathbb{D}$.
Then the transverse boundary $\partial_\pitchfork \mathbb{D} = \overline{[x]} - \mathop{\downarrow}_{\leq_v} [x] = \overline{\mathbb{D}} - \mathop{\downarrow}_{\leq_v}\mathbb{D}$ consists of multi-saddles and semi-multi-saddle separatrices.
Since the abstract weak orbits of saddles and semi-multi-saddle separatrices are themselves, the abstract weak orbit $[x]$ does not intersect the transverse boundary $\partial_\pitchfork \mathbb{D}$
 and so $[x] = \mathbb{D}$.
Therefore the transverse boundary $\partial_\pitchfork [x] = \overline{[x]} - \mathop{\downarrow}_{\leq_v}[x] = \overline{\mathbb{D}} - \mathop{\downarrow}_{\leq_v}\mathbb{D}$ consists of multi-saddles and separatrices whose heights are zero.
This means that the height of $x$ is at most one.
Therefore $\mathrm{ht}_{\leq_\pitchfork}(\mathrm{P}(v)) \leq 1$.
Since $S = \Sv \sqcup \Pv \sqcup \mathrm{P}(v)$, we have $\mathrm{ht}_{\leq_\pitchfork}(S) \leq 1$.
\end{proof}

\begin{lemma}
The binary relation $\leq_\perp$ for a flow of finite type on a compact surface is a partial order at most height two.
\end{lemma}

\begin{proof}
Let $v$ be a flow of finite type on a compact surface $S$.
Proposition~\ref{prop:five} implies that each abstract weak orbit is either a singular point, a semi-multi-saddle separatrix, a connected component of $\Pv$, a transverse annulus, or a trivial flow box.
Since $\Sv$ is finite, each singular point is minimal with respect to $\leq_\perp$.
Therefore $\mathrm{ht}_{\leq_\perp}(\Sv) \leq 0$.
By definition of $\leq_\perp$, the height of a periodic point is zero.
Therefore $\mathrm{ht}_{\leq_\perp}(\Cv) \leq 0$.
Fix a point $x \in S - \Cv = \mathrm{P}(v)$.
By definition of $\leq_\perp$, the height of any point of a limit circuit is at most one.
Suppose that $O(x)$ is a semi-multi-saddle separatrix.
The generalization of the Poincar\'e-Bendixson theorem for a flow with finitely many singular points implies that any $\omega$-limit sets and $\alpha$-limit sets are either singular points or limit circuits.
This means that the height of any point is at most two.
Therefore $\mathrm{ht}_{\leq_\perp}(\mathrm{P}(v)) \leq 2$.
Since $S = \Sv \sqcup \Pv \sqcup \mathrm{P}(v)$, we have $\mathrm{ht}_{\leq_\perp}(S) \leq 2$.
\end{proof}

\begin{proposition}\label{prop:height}
The following properties hold for a flow $v$ of finite type on a compact surface $S$: \\
$(1)$ $S/[v] = S/\langle v \rangle = S/[v]_k = S/\langle v \rangle_k$ for any $k \in \Z_{>0}$.
\\
$(2)$ The binary relation $\leq_\partial$ is a partial order at most height three such that $S_{\leq 0} = \Cv$ and $D(v) \sqcup \Pv \subseteq S_{\leq 2}$, where $S_{\leq i}$ is the set of point of height at most $i$ with respect to the partial order $\leq_\partial$.
\\
$(3)$ The quotient map $S \to S/[v]$ is a finite poset-stratification with respect to $\leq_\partial$.
\\
$(4)$ If there are no limit circuits, then the abstract weak orbit space $S/[v]$ is a stratification.
\end{proposition}

\begin{proof}
Let $v$ be a flow of finite type on a compact surface $S$.
Corollary~\ref{cor:w_s} implies that $S/[v] = S/\langle v \rangle$.
The generalization of the Poincar\'e-Bendixson theorem for a flow with finitely many singular points implies that the $\omega$-limit set and $\alpha$-limit set of a point are either singular points or limit circuits.
Proposition~\ref{prop:five} implies that each abstract weak orbit is either a singular point, a semi-multi-saddle separatrix, a connected component of $\Pv$, a transverse annulus, or a trivial flow box.
This implies that, for any point $x$, we have $[\alpha(x)] = \alpha(x)$ and $[\omega(x)] = \omega(x)$.
By Lemma~\ref{lem:k_th_weak}, we obtain $S/[v] =S/[v]_k$ and $S/\langle v \rangle = S/\langle v \rangle_k$ for any $k \in \Z_{>0}$.
Since $\Sv$ is finite, each singular point is minimal with respect to $\leq_\partial $.
Fix a point $x \in S$.
If $x$ is singular, then it is minimal with respect to $\leq_\partial $.
Therefore $\mathrm{ht}_{\leq_\partial }(\Sv) \leq 0$.
Suppose that $x$ is a semi-multi-saddle separatrix.
Lemma~\ref{lem:sep_multi} implies that the abstract orbit of a semi-multi-saddle separatrix is itself and so $[x] = O$.
Assume that $O(x)$ is a separatrix in a limit circuit.
Then the coborder $\bp [x] = \overline{[x]} - [x] = \overline{O(x)} - O(x) = \alpha'(x) \cup \omega'(x)$ consists of multi-saddles.
This implies that $x \notin \overline{\overline{[x]} - [x]}$ and so that $\mathrm{ht}_{\leq_\partial }(x) = 1$.
Assume that $O(x)$ is a semi-multi-saddle separatrix but not contained in a limit circuit.
Then the coborder $\bp [x] = \overline{[x]} - [x] = \overline{O(x)} - O(x) = \alpha'(x) \cup \omega'(x)$ consist of a multi-saddle and either a singular point or a limit circuit.
Therefore $x \notin \overline{\overline{[x]} - [x]}$ and so $\mathrm{ht}_{\leq_\partial }(x) = 1$ or $2$.
%
Suppose that $x \in \Pv$.
Then $[x]$ is the connected component $C$ of $\Pv$ containing $x$.
Since the intersection $D(v) \cap \partial \Pv$ consists of multi-saddles and separatrices between multi-saddles, it consists of points of height at most one with respect to $\leq_\partial $.
Therefore the coborder $\bp [x] = \overline{[x]} - [x] = \overline{C} - C$ consists of centers and circuits in $D(v)$ and so the difference consists of points whose heights are at most one.
This implies that $\mathrm{ht}_{\leq_\partial }(x) \leq 2$.
Therefore $\mathrm{ht}_{\leq_\partial }(\Pv) \leq 2$.
Suppose that $x$ is contained in a transverse annulus $\mathbb{A} \subseteq \mathrm{P}(v)$ which is a connected component of $S - \mathop{\mathrm{Bd}}(v)$.
Then $\overline{\mathbb{A}} = \mathbb{A} \sqcup (\alpha(x) \cup \omega(x))$.
Therefore the coborder $\bp [x] = \overline{[x]} - [x] = \overline{\mathbb{A}} - \mathbb{A} = \alpha(x) \cup \omega(x)$ consists of singular points and limit circuits.  This means that $\mathrm{ht}_{\leq_\partial }(x) = 1$ or $2$.
%
Suppose that $x$ is contained in a trivial flow box $\mathbb{D}$.
Then the difference $\overline{\mathbb{D}} - \mathbb{D}$ consists of singular points and semi-multi-saddle separatrices.
Since the abstract orbits of saddles and semi-multi-saddle separatrices are themselves, the abstract orbit $[x]$ does not intersect the coborder $\bp \mathbb{D} = \overline{\mathbb{D}} - \mathbb{D}$ and so $[x] = \mathbb{D}$.
Therefore the coborder $\bp [x] = \overline{[x]} - [x] = \overline{\mathbb{D}} - \mathbb{D}$ consists of singular points and semi-multi-saddle separatrices.
Since the $\omega$-limit and $\alpha$-limit sets of semi-multi-saddle separatrices are either singular points or limit circuits, we have $x \notin \overline{\overline{[x]} - [x]}$.
Since the coborder $\bp [x]$
consists of points of heights at most two, we obtain $\mathrm{ht}_{\leq_\partial }(x) \leq 3$.
Therefore $\mathrm{ht}_{\leq_\partial }(\mathrm{P}(v)) \leq 3$.
The finiteness of the abstract weak orbit space $S/[v]$ implies that the partial order $\leq_\partial$ on $S/[v]$ is also the specialization order of $S/[v]$, and that if there are no limit circuits then the quotient map $S \to S/[v]$ is a finite poset-stratification with respect to $\leq_\partial$.
\end{proof}

Notice that the resulting space from the abstract weak orbit space by collapsing limit circuits into singletons is also a  stratification of the surface as in the previous proposition.

\section{Completeness for flows on compact surfaces}
In this section, we show that the abstract weak orbit space is a complete invariant of Morse flows on closed surfaces and of flows generated by structurally stable Hamiltonian vector fields on compact surfaces.

\subsection{Completeness for gradient flows on compact surfaces}
We show the following completeness of an abstract weak orbit space of a regular gradient flow with the partial order $\leq_\partial$.

\begin{proposition}\label{prop:comp_qr}
The abstract weak orbit space equipped with the partial order $\leq_\partial$ is a complete invariant of regular gradient flows on closed surfaces.
In particular, the abstract weak orbit space is a complete invariant of Morse flows on closed surfaces.
\end{proposition}

\begin{proof}
It is known that Morse flows on closed surfaces are gradient flows with nondegenerate singular points but without separatrices from a saddle to a saddle~\cite[Theorem B]{smale1961gradient}.
This implies that Morse flows on closed surfaces are quasi-regular gradient flows.
Let $v$ be a regular gradient flow on a closed surface $S$.
Then $S = \Sv \sqcup \mathrm{P}(v)$.
Since a closed surface has at most finitely many connected components, we may assume that $S$ is connected.
By non-existence of limit circuits, a generalization of the Poincar\'e-Bendixson theorem for a flow with finitely many singular points implies that any $\omega$-limit sets and $\alpha$-limit sets are either sinks, sources, or saddles.
By Lemma~\ref{lem:three_pts}, we may assume that $v$ is not a flow on a sphere $\mathbb{S}^2$ which consists of a sink $s_+$, a source $s_-$ and non-recurrent orbits between them.
Poincar\'e-Hopf theorem for continuous flows with finitely many singular points on compact surfaces implies that there are saddles and so there are no transverse annuli.
Proposition~\ref{prop:five} implies that the abstract weak orbit space is a finite set of sinks, sources, saddles, separatrices, and trivial flow boxes.
This means that the gluing information of boundaries of the trivial flow boxes can reconstruct the flow $v$.
Proposition~\ref{prop:height} implies that the transverse order $\leq_\partial$ is the union of these orders $\leq_\pitchfork, \leq_\alpha, \leq_\omega$ and is a partial order.
We show that the abstract weak orbit space with the transverse order $\leq_\partial$
have the gluing information.
Indeed, denote by $S_{= k}$ (resp. $S_{\leq k}$) the set of points whose abstract weak orbits have height $k$ (resp. at most height $k$) with respect to the transverse order $\leq_\partial$.
By construction, the union of abstract weak orbits whose heights are zero with respect to $\leq_\partial$ is $\Sv$ (i.e. $S_{= 0} = \Sv$), the union $S_{= 1}$ of abstract weak orbits whose heights are one with respect to $\leq_\partial$ is the union of semi-multi-saddle separatrices (i.e. $\Sv \cup D(v) = S_{\leq 1}$), and the union  $S_{= 2}$ of abstract weak orbits whose heights are two with respect to $\leq_\partial$ is the union of open trivial flow boxes which are connected components of the complement $S - D(v)$ of the multi-saddle connection diagram $D(v)$ (i.e. $S_{= 2} = S - (\Sv \cup D(v))$).
Since $S_{= 0} = \Sv$ and $S_{= 1} = D(v) \setminus \Sv$, the abstract directed graph structure of $\Sv \cup D(v)$ can be reconstructed by the pre-order $\leq_\partial$.
For any abstract weak orbit $U$ which is an open trivial flow box and which is a connected component of $S - D(v)$, we have that $\mathrm{ht}_{\leq_\partial }(U) \leq 2$ and $\partial U \subseteq \Sv \cup D(v) = S_{\leq 1}$.
For any point $x \in U$, we have that $[x] = U$ and that the transverse boundary  $\partial_\pitchfork [x] = \{ y \in S \mid y <_\pitchfork x \} = \mathop{\downarrow}_{\leq_\pitchfork} [x]  - [x] = \mathop{\downarrow}_{\leq_\partial} U - U$ is the union of abstract weak orbits which is less than $U$ with respect to the pre-order $\leq_\pitchfork$.
Therefore the partial order $\leq_\partial$ have the gluing information.
\end{proof}

\subsection{Completeness for generic Hamiltonian flows on compact surfaces}
Recall that the set of structurally stable Hamiltonian flows on compact surfaces is open dense in the set of Hamiltonian flows.
Generic Hamiltonian flows on compact surfaces are classified by the abstract weak orbit spaces as follows.

\begin{proposition}
The abstract weak orbit space equipped with the partial order $\leq_\partial$ is a complete invariant of flows generated by structurally stable Hamiltonian vector fields on compact surfaces.
\end{proposition}

\begin{proof}
Let $v$ be a flow generated by a structurally stable Hamiltonian vector field on a compact surface $S$.
Since the abstract weak orbit space of a connected surface is connected, we may assume that $S$ is connected.
By Lemma~\ref{lem:ham_trivial}, we may assume that the multi-saddle connection diagram $D(v)$ is not empty.
\cite[Theorem 2.3.8, p. 74]{ma2005geometric} implies that each singular point is nondegenerate and each separatrix is self-connected.
Since $v$ is generated by a Hamiltonian vector field, we have $S = \mathop{\mathrm{Cl}}(v) \sqcup \mathrm{P}(v)$ and it has neither sinks, $\partial$-sinks, sources, nor $\partial$-sources, and so each singular point is either a center, a saddle, or a $\partial$-saddle.
Denote by $\mathop{\mathrm{Sing}}(v)_c$ the set of centers.
By non-existence of limit circuits, the generalization of the Poincar\'e-Bendixson theorem for a flow with finitely many singular points implies that any $\omega$-limit sets and $\alpha$-limit sets of non-periodic orbits are singular points.
This implies that $D(v) = (\Sv - \mathop{\mathrm{Sing}}(v)_c) \sqcup \mathrm{P}(v)$ and so the complement $ S- (\mathop{\mathrm{Sing}}(v)_c \sqcup D(v))$ of the union of centers and the multi-saddle connection diagram $D(v)$ is the periodic point set $\Pv$ which is a finite disjoint union of periodic annuli.
Proposition~\ref{prop:five} implies that the abstract weak orbit space is a finite set of centers, saddles, $\partial$-saddles, separatrices, and periodic annuli.
This means that the gluing information of boundaries of the periodic annuli can reconstruct the flow $v$.
Proposition~\ref{prop:height} implies that the transverse order $\leq_\partial$ is the union of these orders $\leq_\pitchfork, \leq_\alpha, \leq_\omega$ and is a partial order.
We show that the abstract weak orbit space with the transverse order $\leq_\partial$ have the gluing information.
Indeed, denote by $S_{= k}$ (resp. $S_{\leq k}$) the set of points whose abstract weak orbits have height $k$ (resp. at most height $k$) with respect to the transverse order $\leq_\partial$.
By construction, we have that $S_{= 0} = \Sv$.
Since the multi-saddle connection diagram $D(v)$ of a flow generated by a structurally stable Hamiltonian vector field on a compact surface is the union of saddles, $\partial$-saddles, and separatrices, the abstract directed graph of the multi-saddle connection diagram $D(v)$ is the finite union of elements of height less than two.
Since the boundary of a connected component of $\Pv = S - (\Sv \sqcup \mathrm{P}(v))$ contains separatrices, we have $S_{= 1} = D(v) \setminus \Sv = \mathrm{P}(v)$ and $S_{= 2} = \Pv = S - (\Sv \sqcup \mathrm{P}(v))$.
The self-connectivity of $D(v)$ implies that each multi-saddle connection out of the boundary $\partial S$ has trivial cyclic orders (see Figure~\ref{tr_order} on the left and middle) and that the boundary component of such an open periodic annulus intersecting $\partial S$ is an embedded circuit (see Figure~\ref{embed}).
\begin{figure}
\begin{center}
\includegraphics[scale=0.4]{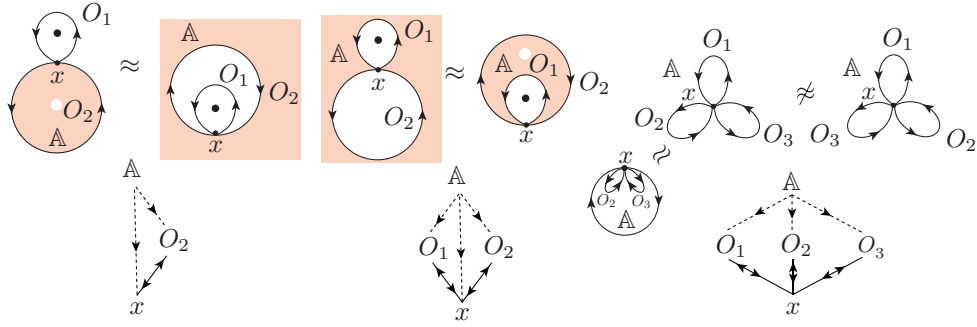}
\end{center}
\caption{Left, the abstract weak orbit space restricted to the boundary component of a periodic annulus $\mathbb{A}$ which consists of a saddle $x$ and a separatrix $O_1$ has exactly one realization, which is a circuit $[x] \leq_{\alpha} O_2 \geq_{\omega} [x]$; middle,  the abstract weak orbit space restricted to the boundary component of a periodic annulus $\mathbb{A}$ which consists of a saddle $x$ and two separatrices $O_1, O_2$ has exactly one realization, which is a circuit $[x] \leq_{\alpha} O_1 \geq_{\omega} [x] \leq_{\alpha} O_2 \geq_{\omega} [x]$; right, the abstract weak orbit space restricted to the boundary component of a periodic annulus $\mathbb{A}$ which consists of a saddle $x$ and separatrices $O_1, O_2, O_3$ has exactly two realizations. Moreover, the binary relation $\leq_\partial$ is a pre-order which corresponds to the union of the pre-order $\leq_v$ and the binary relation $\leq_\pitchfork$ as a direct product.}
\label{tr_order}
\end{figure}
\begin{figure}
\begin{center}
\includegraphics[scale=0.5]{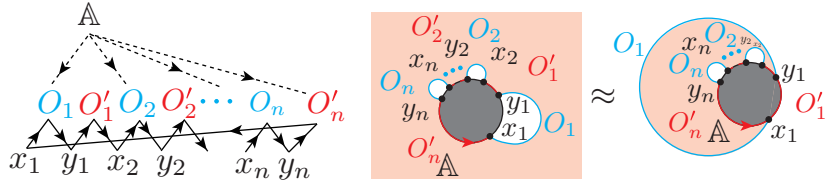}
\end{center}
\caption{Middle and right, the boundary component of an open periodic annulus $\mathbb{A}$ intersecting $\partial S$ and $O'_1$ is an embedded circuit $\{ x_1 \} \sqcup O_1\sqcup \{ y_1 \}  \sqcup O'_1\sqcup \{ x_2 \} \sqcup O_2\sqcup \{ y_2 \}  \sqcup O'_2 \sqcup \cdots \sqcup O'_n$; left, solid lines form the Hessian diagram of the order $\leq_v$ whose down (resp. up) arrows correspond to $\leq_{\omega}$ (resp. $\leq_{\alpha}$), and dotted directed lines correspond to $\leq_\pitchfork$. Moreover, the binary relation $\leq_\partial$ is a pre-order which corresponds to the union of the pre-order $\leq_v$ and the binary relation $\leq_\pitchfork$ as a direct product.}
\label{embed}
\end{figure}
This implies that the abstract directed graph structure of $D(v) = S_{\leq 1}$ can be reconstructed by the pre-order $\leq_\partial$.
For any abstract weak orbit $U$ which is an open periodic annulus and which is a connected component of $S - D(v)$, we have $\mathrm{ht}_{\leq_\partial }(U) = 2$.
For any point $x \in U$, we have that $[x] = U$ and that the boundary $\partial [x] = \partial_\pitchfork [x] = \{ y \in S \mid y <_\pitchfork x \} = \mathop{\downarrow}_{\leq_\pitchfork} [x]  - [x] = \mathop{\downarrow}_{\leq_\pitchfork} U - U$ contains non-trivial circuits in $D(v)$ and consists of centers and circuits in $D(v)$.
Therefore the boundary $\partial U = \partial_\pitchfork U =  \mathop{\downarrow}_{\leq_\pitchfork} U - U$ is the union of abstract weak orbits which is less than $U$ with respect to the pre-order $\leq_\partial$.
Therefore the partial order $\leq_\partial$ has the gluing information.
\end{proof}

\section{Reconstructions of orbit spaces and abstract weak orbit spaces of flows}


\subsection{Detection of original flows from the time-one maps}

A function on a one-dimensional manifold is piecewise non-constant linear if there is a cover of nondegenerate closed intervals to which the restriction of the function is linear but not constant.
We have the following property to reconstruct orbit spaces.

\begin{lemma}\label{lem:pl}
Let $v$ be a Hamiltonian flow with finitely many singular points on a compact surface $S$ and $T$ a union of transverse arcs such that any periodic orbit intersects $T$ exactly once.
If the period $p \colon \Pv/v \to \R$ is piecewise non-constant linear, then the orbit space of $v$ is homeomorphic to the abstract weak orbit space of {\rm (}the suspension flow of {\rm )} the time-one map of $v$.
\end{lemma}

\begin{proof}
Lemma~\ref{lem:ham} implies that $v$ is of finite type and that $S = \Sv \sqcup \Pv \sqcup \mathrm{P}(v)$.
Moreover, the surface $S$ is a finite union of abstract weak orbits which are either centers, semi-multi-saddle separatrices, or periodic annuli.
The flow box theorem (cf. \cite[Theorem 1.1, p.45]{aranson1996introduction}) implies that the period $p \colon \Pv/v \to \R$ is continuous.
Suppose that $p \colon \Pv/v \to \R$ is piecewise non-constant linear.
Then the subset of periodic orbits with irrational periods and one with rational period are dense subsets of $\Pv$.
Let $f$ be the time-one map of $v$ and $w$ the suspension flow of $f$.
Then $\mathrm{P}(v) = \mathrm{P}(f)$, $\Sv \subset \mathop{\mathrm{Fix}}(f)$, and $\overline{\Pv \cap \mathop{\mathrm{Per}}_{>1}(f)} = \overline{\Pv \cap \mathrm{R}(f)} = S$.
For a point $x \in \Pv$ with irrational period with respect to $v$, the abstract weak orbit of $w$ is a torus and so the abstract weak orbit of $f$ is $O_v(x)$.
For a point $x \in \Pv$ with rational period with respect to $v$, the abstract weak orbit of $w$ is the connected component of $\mathop{\mathrm{Per}}(w)$ containing $x$ which is a torus, and so the abstract weak orbit of $f$ is $O_v(x)$.
For a point $x \in \Sv$, the abstract weak orbit of $w$ is itself which is a periodic orbit, and so the abstract weak orbit of $f$ is the fixed point.
For a point $x \in \mathrm{P}(v)$, the abstract weak orbit of $f$ is the connected component of $\mathrm{P}(f)$ containing $x$ which is the saturation $w(O_v(x))$ of $O_v(x)$ by $w$ and so the abstract weak orbit of $f$ is $O_v(x)$.
This means that then the abstract weak orbit spaces of $v$ and of (the suspension flow of) the time-one map of $v$ are homeomorphic.
\end{proof}

We show that the orbit space of a Hamiltonian flows with finitely many singular points on a compact surface is homeomorphic to the abstract weak orbit space of the time-one map by taking a reparametrization.
Precisely, we have the following reconstructions of topologies of flows.

\begin{theorem}\label{th:Ham_equiv}
For any Hamiltonian flow $v$ with finitely many singular points on a compact surface $S$, there is an arbitrarily small reparametrization $w$ of $v$ with the compact-open topology such that the orbit space of $v$ is homeomorphic to the abstract weak orbit space of the time-one map of the reparametrization $w$.
\end{theorem}

\begin{proof}
Since $S$ consists of finitely many connected components, we may assume that $S$ is connected.
Let $\mathcal{H}(S)$ be the set of Hamiltonian flows with finitely many singular points on a compact surface $S$ with the compact-open topology.
Fix a  Riemannian metric $d$ on the compact surface $S$.
Since $\mathcal{H}(S)$ is a complete metrizable space (cf. \cite[Theorem~2.4]{montgomery1973cohomology}), fix a metric $d_{\mathcal{H}(S)}$ on $\mathcal{H}(S)$ defined by $d_{\mathcal{H}(S)}(v,w) := \sup \{ d(v(t,x), w(t,x)) \mid (t,x) \in [0,1] \times S \}$ for any flows $v, w \in \mathcal{H}(S)$.
Fix a Hamiltonian flow $v$ in $\mathcal{H}(S)$ and a positive number $\varepsilon \in  (0, 1)$.
Lemma~\ref{lem:ham} implies that $v$ is of finite type.
Moreover, the surface $S$ is a finite union of abstract weak orbits of $v$ which are either centers, semi-multi-saddle separatrices, or periodic annuli.
Note that the orbit space of such a periodic annulus is an interval.
Lemma~\ref{lem:ham} implies that there is a filtration $\emptyset \subset S_{\leq 0} \subset S_{\leq 1}  \subset S_{\leq 2}$ is a stratification with respect to $\leq_v$ such that $S_{\leq 0} = \Sv$, $S_{\leq 1} = \Sv \cup D(v)$, and $S_{\leq 2} - S_{\leq 1}  =  \Pv$ unless $D(v)$ is empty.
We claim that there is a reparametrization $w$ of $v$ with $d_{\mathcal{H}(S)}(v, w)< \varepsilon$ such that, for any periodic annulus $\mathbb{A}$ which is an abstract weak orbit of $w$, the periodic annulus $\mathbb{A}$ is also a union of abstract weak orbits of $w_1$ and $\mathbb{A}/w = \mathbb{A}_{w_1}/[v_{w_1}]$, where $w_1$ is the time-one map of $w$ and $v_{w_1}$ is the suspension flow of $w_1$.
Indeed, fix a periodic annulus $\mathbb{A}$ which is an abstract weak orbit of $v$.
Therefore any boundary component of $\mathbb{A}$ is either a center or a non-trivial circuit.
Suppose that $\mathbb{A} \cap \partial S  = \emptyset$.
Then there is a transverse open arc $T$ such that the two end points in $\overline{T} - T$ are singular points.
The flow box theorem implies that for any point $x \in T$ there is an open flow box $V$ containing $x$.
Identify $T$ with $\{ 0 \} \times \R \subset \R^2$.
Since any bounded closed interval is compact, we may assume that, for any $n \in \Z$, there are positive numbers $\varepsilon_n >0$ such that $U_n := [- \varepsilon_n, \varepsilon_n ] \times [n, n+1] \subset V$ is a closed flow box whose diameter is less than one and in which any closed orbit arc is of the form $[- \varepsilon_n, \varepsilon_n ] \times \{ y \}$.
Here the diameter of a subset $A$ of a metric space $(X,d)$ is the number $\sup \{ d(a,b) \in \R_{\geq 0} \mid a, b \in A \}$.
 Then the union $\bigcup_{n \in \Z} U_n \subset V$ is a \nbd of $T$.
 By a small perturbation of the parameter of $v$, we can obtain a reparametrization $v'$ of $v$ such that there are positive numbers $\varepsilon'_n \in (0, \varepsilon_n)$ and a continuous function $\delta \colon T \to \R_{>0}$ such that $d_{\mathcal{H}(S)}(v, v') < \varepsilon/2$ and that $U'_n := [- \varepsilon'_n, \varepsilon'_n ] \times [n, n+1]$ is a closed flow box with $d v'/dt |_{U'_n}(x,y) = (\delta ((0,y)), 0)$ for any $(x,y) \in U'_n$.
Put $U' := \bigcup_{n \in \Z} U'_n$.
Then the restriction $v'|_{U'}$ is a horizontal flow with constant speeds on all horizontal arc of forms $[- \varepsilon'_n, \varepsilon'_n ] \times \{ y \}$.
Let $p \colon T \to \R$ be the period of their orbits with respect to $v'$.
Define a continuous function $p' \colon T \to (0,1)$ by $p'((0,y)) := t_+(y) - t_-(y)$, where $t_-(y)$ is the negative hitting time from $y$ to the vertical boundary $\{ - \varepsilon'_n \} \times [n, n+1] \subset \partial U_n$ of $U_n$ with respect to $v'$ (i.e. $t_-(y)$ is the negative largest number with $v'_{t_-(y)}(y) = (-\varepsilon'_n, y)$), and $t_+(y)$ is the positive hitting time from $y$ to the vertical boundary $\{ \varepsilon'_n \} \times [n, n+1] \subset \partial U_n$ of $U_n$ with respect to $v'$  (i.e. $t_+(y)$ is the positive smallest number with $v'_{t_+(y)}(y) = (\varepsilon'_n, y)$).
Put $p'_{n,\min} := \min_{q \in T \cap U'_n} p'(q)$ and $\delta_{n,\max} := \max_{q \in T \cap U'_n} \delta(q)$.
Since the saturation $v'(T)$ is the periodic annulus $\mathbb{A}$, the functions $p$ and $p'$ are continuous.
Since the restriction $p|_{T \cap U'_n}$ is uniformly continuous, by a small perturbation of the parameter of $p|_{T \cap U'_n}$, for any positive integer $m \in  \Z_{>0}$, there is a piecewise non-constant linear continuous function $p_m \colon T \to \R$  such that $0 < p(q) - p_m(q) < \min \{ 1/ m, p'_{n,\min}/m, p'_{n,\min}/(m \delta_{n,\max}) \}$ for any $n \in \Z$ and for any $q \in \{ 0 \} \times [n, n+1] = T \cap U'_n$.
Then $\sup \{ p(q) - p_m(q) \mid q \in T \} \leq 1/ m$ and so $\lim_{m \to \infty} \sup \{ p(q) - p_m(q) \mid q \in T \} = 0$.
On the other hand, for any $q \in T$, replacing $v'|_{U'}$ such that  the vector field $dv'/dt|_{U'}(x,y) = (\delta (q), 0)$ becomes $(\delta (q) + a, 0)$ for any $a > 0$, the periods of the orbit of $q$ with respect to the resulting flows for any numbers $a>0$ form $(p(q) - p'(q), p(q)) \subset \R_{>0}$, and so there are numbers $a_m(q) > 0$ with $\lim_{m \to \infty} a_m(q) = 0$ such that the periods of the orbits of $q$ with respect to the resulting flows are $p_m(q)$ because $p_m(q) \in (p(q) - p'_{n,\min}, p(q)) \subset (p(q) - p'(q), p(q))$ for any $p \in \{ 0 \} \times [n, n+1] = T \cap U'_n$.
We show that $a_m$ is a continuous positive function.
Indeed, since $p'(q)\delta(q)$ is the length of the interval on which the speed $\delta(q)$ is replaced by $\delta(q) +a_m(q)$, we have $p(q) - p_m(q) = p'(q) - p'(q)\delta(q)/(\delta(q) +a_m(q)) = p'(q) a_m(q)/(\delta(q) +a_m(q))$ and so $(p(q) - p_m(q))/p'(q) = a_m(q)/(\delta(q) +a_m(q))$.
Since functions $p - p_m$ and $p'(q)$ are positive and continuous, by $0 < p(q) - p_m(q) < p'(q)$, the function $C_m \colon T \to \R_{>0}$ defined by $C_m(q) :=(p(q) - p_m(q))/p'(q)$ is continuous and $0 < C_m(q)< \min \{ 1, 1/m, 1/(m \delta_{n,\max}) \}$.
Then $C_m(q) (\delta(q) +a_m(q)) = a_m(q)$ and so $a_m(q) = \delta(q) C_m(q)/(1-C_m(q)) \in (0, (1-m)^{-1}m^{-1})$.
This implies that the function $a_m \colon T \to \R$ is continuous.
The resulting flow $v'_m$ from $v'$ by replacing the vector field $dv'/dt|_{U'}(x,y) = (\delta (q), 0)$ by $(\delta (q) + a_m(q), 0)$ is continuous such that the piecewise non-constant linear continuous function $p_m \colon T \to \R$ is the period with respect to $v'_m$.
Since $0 < a_m(q) < (1-m)^{-1}m^{-1}$ and $\lim_{m \to \infty} a_m(q) = 0$, there is a large integer $M >0$ such that $d_{\mathcal{H}(S)}(v', v'_M)< \varepsilon/2$.
Then $d_{\mathcal{H}(S)}(v, v'_M)< \varepsilon$.
Suppose that $\mathbb{A} \cap \partial S \neq \emptyset$.
Identifying $T$ with either $\{ 0 \} \times \R_{\geq 0} \subset \R^2$ or $\{ 0 \} \times [0,1] \subset \R^2$, the same argument implies the assertion.
Applying the above operations to all periodic annuli which are the connected components of the complement $S - D(v)$, we obtain the resulting vector field $w$ with $d_{\mathcal{H}(S)}(v, w)< \varepsilon$ and the period $p \colon \Pv/v \to \R$ is piecewise non-constant linear.
This means that, for any periodic annulus $\mathbb{A}$ which is an abstract weak orbit of $w$, since the union of periodic orbits in $\mathbb{A}$ by $w$ with rational periods and the union of periodic orbits in $\mathbb{A}$ by $w$ with irrational periods are dense in $\mathbb{A}$, the abstract weak orbit in a point of $\mathbb{A}$ by $w_1$ is a circle which is a periodic orbit of $w$.
This completes the claim.
Lemma~\ref{lem:pl} implies that the orbit space $S/v$ is homeomorphic to the abstract weak orbit space $S/[w_1]$ of the time-one map of the reparametrization  $w$.
\end{proof}

\subsection{Coincidence of abstract weak orbit spaces of flows and the time-one maps}

We show coincidences of the (second) abstract weak orbit spaces of a Morse-Smale flow on a compact manifold and of the time-one map.

\begin{theorem}\label{th:MS_equiv}
The {\rm(}second {\rm)} abstract weak orbit space of a Morse-Smale flow on a compact manifold is homeomorphic to the second abstract weak orbit space of the suspension flow of the time-one map.
Moreover, all periodic orbits of the flow has an irrational period if and only if the abstract weak orbit spaces of the flow and the time-one map are homeomorphic.
\end{theorem}

\begin{proof}
Let $v$ be a Morse-Smale flow on a compact manifold $M$, $f := v_1 \colon M \to M$ the time-one map of $v$, and $v_f$ the suspension flow of $f$ on the mapping torus $M_f$.
Identify $M$ with the slice $(M \times \{0 \} )/\sim_{\mathrm{susp}} \subset M_f$.
Proposition~\ref{prop:fin_MS} implies that  $M = \Cv \sqcup \mathrm{P}(v)$ and each abstract weak orbit $[x]_v$ by $v$ is either a closed orbit $O(x)$ or the  connected component containing $x$ of the connecting orbit set $W^u_v(\alpha_v(x)) \cap W^s_v(\omega_v(x))$ from the hyperbolic closed orbit $W^u_v(\alpha_v(x))$ to the hyperbolic closed orbit $W^s_v(\omega_v(x))$, where $W^u_v(A)$ (resp. $W^s_v(A)$, $\alpha_v(x)$,  $\omega_v(x)$) is the unstable manifold (resp. stable manifold, $\alpha$-limit set, $\omega$-limit set) with respect to $v$.
Then $\mathop{\mathrm{Cl}}(v) = \mathop{\mathrm{Per}}(f) \sqcup \mathrm{R}(f)$ and $\mathrm{P}(v) = \mathrm{P}(f)$.
Fix a point $x$.
Suppose that $x \in \Sv$.
By isolatedness of closed orbits of $v$, the orbit $O_{v_f}(x)$ is a connected component of $\mathop{\mathrm{Per}}(v_f)$ and so $O_{v_f}(x) = [x]_{v_f}$, where $[x]_{v_f}$ is the abstract weak orbit of $x = (x, 0)/\sim_{\mathrm{susp}} \in M_f$.
This implies that the abstract weak orbit $[x]_f$ is the singleton $\{ x \} = [x]_v$.
Suppose that $x \in \Pv$.
Assume that $x \in \mathrm{R}(f)$.
Then $\overline{O_f(x)} = O_v(x)$ is a limit cycle of $f$ such that $\overline{O_f(x)}$ is a connected component of $\mathrm{R}(f) \subseteq \Pv$.
Therefore the closure $\overline{O_{v_f}(x)}$ is a minimal torus which is a connected component of $\mathrm{R}(v_f) \subseteq v_f(\Pv)$, and so $\overline{O_{v_f}(x)} = \hat{O}_{v_f}(x)= \check{O}_{v_f}(x) \subset \mathrm{R}(v_f)$.
This means that $[x]_f = \check{O}_{f}(x) = O_v(x) = [x]_v$.
Thus we may assume that $x \in \mathop{\mathrm{Per}}(f)$.
Then $O_v(x)$ is the connected component of $\mathop{\mathrm{Per}}(f)$ containing $x$.
Therefore $v_f(O_v(x))$ is an invariant torus which is a connected component of $\mathop{\mathrm{Per}}(v_f)$ and so is the abstract weak orbit $[x]_{v_f}$.
This implies that $[x]_f = O_v(x) = [x]_v$.
Thus we may assume that $x \in \mathrm{P}(v)$.
Then $x \in \mathrm{P}(f)$ and so  $x = (x, 0)/\sim_{\mathrm{susp}} \in  \mathrm{P}(v_f)$.
By definition of time-one map, the $\alpha$-limit set $\alpha_f(x)$ by $f$ is an invariant subset of the closed orbit $\alpha_v(x)$ and the $\omega$-limit set $\omega_f(x)$ by $f$ is an invariant subset of the closed orbit $\omega_v(x)$.
In particular, if $\alpha_f(x) \subseteq \mathrm{R}(f)$ (i.e the period of $\alpha_v(x)$ for $v$ is irrational) (resp. $\omega_f(x) \subseteq \mathrm{R}(f)$ (i.e the period of $\omega_v(x)$ for $v$ is irrational)), then $\alpha_f(x) = \alpha_v(x)$ (resp. $\omega_f(x)  = \omega_v(x))$ is a closed orbit of $v$.
Then if the periods of all periodic orbits of $v$ is irrational, then the orbit space $M/v$ of $v$ is homeomorphic to the abstract weak orbit space $M_f/[v_f]$ and so to $M/[v]$.
If $\alpha_f(x) \subseteq \mathop{\mathrm{Per}}(f)$ (resp. $\omega_f(x) \subseteq \mathop{\mathrm{Per}}(f)$), then $\alpha_f(x) \subsetneq \alpha_v(x)$ (resp. $\omega_f(x)  \subsetneq \omega_v(x))$ is a finite subset contained in the closed orbit $\alpha_v(x)$ (resp. $\omega_v(x)$) of $v$, and so $[x]_f = [x]_{f,1} \subsetneq [x]_v$.
This means that $[x]_f \subsetneq [x]_v$, and that $\alpha_f(x)$ and $\omega_f(x)$ are closed orbits if and only if $[x]_f = [x]_v$.
Notice that $\alpha_f(x)$ and $\omega_f(x)$ are closed orbits if and only if the periods of closed orbits $\alpha_f(x)$ and $\omega_f(x)$ of $v$ are irrational.
This means that all periodic orbits of the flow has an irrational period if and only if the abstract weak orbit spaces of the flow and the time-one map are homeomorphic.
Since the hyperbolic closed orbit of $v$ is the abstract weak orbit of $v$ and $f$,  we have, for any $y \in [x]_v$, $[\alpha_f(y)]_f = \bigcup_{z \in \alpha_f(y)} [z]_f = \bigcup_{z \in \alpha_f(y)} \alpha_v(y) = \alpha_v(y) =\alpha_v(x) = [\alpha_v(x)]_v$ and $[\omega_f(y)]_f = [\omega_v(x)]_v$ by symmetry.
Since $\alpha_v(x)$ and $\omega_v(x)$ are closed orbits of $v$, we obtain
\begin{eqnarray*}
  [x]_{f,2} & =& \left\{ y \in \mathrm{P}(v) \mid [\alpha_f(y)]_{f,1} = [\alpha_f(x)]_{f,1}, [\omega_f(y)]_{f,1} = [\omega_f(x)]_{f,1} \right\} \\
&  = & \left\{ y \in \mathrm{P}(v) \mid [\alpha_f(y)]_{f} = [\alpha_f(x)]_{f}, [\omega_f(y)]_{f} = [\omega_f(x)]_{f} \right\}  \\
&  = & \left\{ y \in \mathrm{P}(v) \mid [\alpha_f(y)]_{f} = \alpha_v(x), [\omega_f(y)]_{f} = \omega_v(x) \right\}  \\
&  = & \left\{ y \in \mathrm{P}(v) \mid \alpha_f(y) \subseteq \alpha_v(x), \omega_f(y) \subseteq \omega_v(x) \right\}  \\
&  = & \left\{ y \in \mathrm{P}(v) \mid \alpha_v(y) = \alpha_v(x), \omega_v(y) = \omega_v(x) \right\}  \\
&  = & \left\{ y \in \mathrm{P}(v) \mid \alpha'_v(y) = \alpha'_v(x), \omega'_v(y) = \omega'_v(x) \right\}  \\
& = & [x]_v
\end{eqnarray*}
Therefore the abstract weak orbit space $M/[v]$ of $v$ is homeomorphic to the second abstract weak orbit space $M_f/[v_f]_2$.
\end{proof}

\begin{corollary}\label{cor:MS_equiv}
The abstract weak orbit space of a Morse flow on a compact manifold is homeomorphic to the abstract weak orbit space of the time-one map.
\end{corollary}

\section{Examples}

%

We state an example of a periodic homeomorphism to state triviality of abstract weak orbit space for virtually trivial dynamics.

\begin{example}
Let $g$ be a periodic homeomorphism on a compact orientable surface $S$ and $v_g$ the suspension flow of $g$ on the mapping torus 3-manifold $M$.
Then the orbit space $M/v_g \cong M/\hat{v}_g \cong S/g$ is an orbifold but the abstract {\rm(}weak {\rm)} orbit space $M/[v_g] = M/\langle v_g \rangle$ and the Morse graph is a singleton.
\end{example}

\subsection{Recurrent dynamics}

The Morse graphs of all examples in this subsection are singletons but the abstract weak orbit spaces are not singletons.

\begin{example}\label{ex:rot}
Let $\mathbb{D}$ be a unit disk with the polar coordinate system $(r, \theta)$ and $w$ a rotation defined by $w(t,(r,\theta)) = (r, 2 \pi rt + \theta)$.
Denote by $f_w$ the time-one map and by $v_{f_w}$ the suspension flow on the resulting solid torus $M$.
Then each orbit of $v_{f_w}$ is periodic or non-closed recurrent, and each minimal set of $v_{f_w}$ is a periodic orbit or a torus.
Therefore the orbit space $M/v_{f_w}$ is not $T_0$, the orbit class space $M/{\hat{v}_{f_w}}$ is $T_1$ but not $T_2$, and the abstract {\rm(}weak {\rm)} orbit space $M/[v_{f_w}] = M/\langle v_{f_w} \rangle  = \mathbb{D}/v$ is a closed interval and so $T_2$.
\end{example}

\begin{example}
Let $v$ be a Hamiltonian flow on a closed disk $\mathbb{D}_1$ with periodic boundary $O$ with two centers $c_1, c_2$, one homoclinic saddle connection $\{ s \} \sqcup O_1 \sqcup O_2$ and three open periodic annuli $\mathbb{A}_1, \mathbb{A}_2, \mathbb{A}_3$ $($see Figure~\ref{ex:01}$)$.
\begin{figure}
\begin{center}
\includegraphics[scale=0.15]{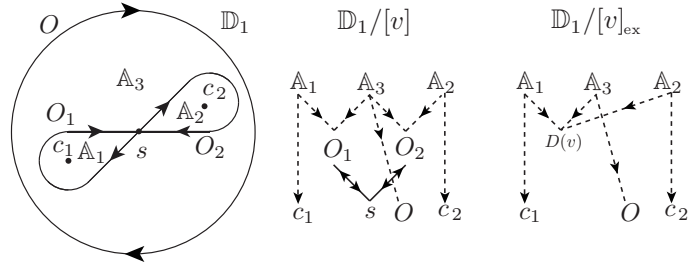}
\end{center}
\caption{The picture on the left is a Hamiltonian flow on a closed disk $\mathbb{D}_1$.
In the diagram on the right, dotted directed lines correspond to $\leq_\pitchfork$. Moreover, the binary relation $\leq_\partial$ is a pre-order which corresponds to the union of the pre-order $\leq_v$ and the binary relation $\leq_\pitchfork$ as a direct product.}
\label{ex:01}
\end{figure}
Then $\mathbb{D}_1 = \{ c_1, c_2, s \} \sqcup O_1 \sqcup O_2 \sqcup \mathbb{A}_1 \sqcup \mathbb{A}_2 \sqcup \mathbb{A}_3$ and $\mathbb{D}_1/[v] = \{ \{ c_1 \}, \{c_2\}, \{s \}, O_1, O_2, \mathbb{A}_1, \mathbb{A}_2,  \mathbb{A}_3 \}$.
Moreover, the abstract graph of the Reeb graph is the extended weak orbit space $\mathbb{D}_1/[v]_{\mathrm{ex}}$ which is the quotient space of the abstract weak orbit space $\mathbb{D}_1/[v]$ by collapsing the homoclinic saddle connection $\{ s \} \sqcup O_1 \sqcup O_2$ into a singleton.
\end{example}

\begin{example}
Let $v$ be a volume-preserving flow on a connected compact manifold which is neither identical, pointwise periodic, nor minimal.
The abstract orbit space of $v$ is not a singleton.
Indeed, since any volume-preserving flows on a compact manifold have no non-wandering domains and so are non-wandering and the $\Omega$-limit set is contained in the chain recurrent set $\mathrm{CR}(v)$, we have $X = \Omega(v) = \mathrm{CR}(v)$.
Corollary~\ref{prop:refine} implies the assertion.
\end{example}

Recall that Nielsen-Thurston theorem~\cite{thurston1988geometry} as follows:
Let $S$ be a compact connected orientable surface and $f:S \to S$ a homeomorphism.
Then there is a map $g$ isotopic to $f$ such that at least one of the following holds:
$(1)$ $g$ is periodic (i.e. $g^k = 1_S$ for some positive integer $k$);
$(2)$ $g$ preserves some finite union of disjoint simple closed curves on $S$, called $g$ reducible; or
$(3)$ $g$ is pseudo-Anosov.
We state an example of a homeomorphism which is isotopic to a periodic homeomorphism.

\begin{example}
Let $v_g$ be the suspension flow on a closed 3-manifold $\mathbb{T}^2_g$ of a homeomorphism $g$ on a torus $\mathbb{T}^2$ whose minimal sets consist of a semi-attracting limit cycle $\gamma_1$, one semi-attracting and semi-repelling limit cycle $\gamma_2$, one semi-repelling limit cycle, and cycles as on the left of Figure~\ref{ex:02}.
\begin{figure}
\begin{center}
\includegraphics[scale=0.195]{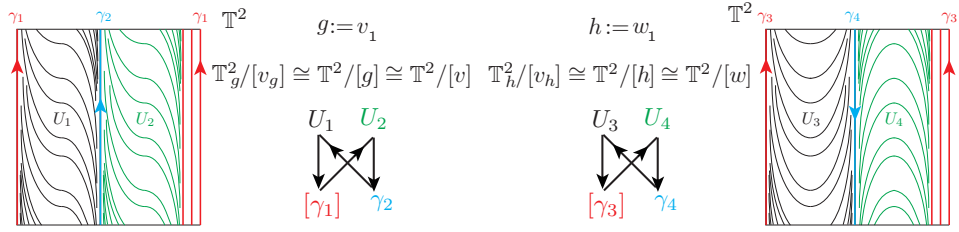}
\end{center}
\caption{The picture on the left and right are homeomorphisms $g$ and $h$ on torus $\mathbb{T}^2$ which are time-one map of flows, and whose suspension flows $v_g$ and $v_h$ is flows on closed 3-manifolds.
}
\label{ex:02}
\end{figure}
Then $\mathbb{T}^2 = [\gamma_1] \sqcup \gamma_2 \sqcup U_1 \sqcup U_2$ and the abstract weak orbit space $\mathbb{T}^2_g/[v_g] \cong \mathbb{T}^2/[g] = \{ [\gamma_1], \gamma_2, U_1, U_2 \}$ with the pre-order $\leq_v$ is a poset with height one which consists of four points with $U_1 <_{\alpha} [\gamma_1] >_{\omega} U_2$ and $U_1 <_{\omega} \gamma_2 >_{\alpha} U_2$, where $[\gamma_1]$ is a closed annulus consisting of cycles, the open annuli $U_1$ and $U_2$ are connected components of $\mathbb{T}^2 - ([\gamma_1] \sqcup \gamma_2)$.
Note that the poset $\mathbb{T}^2_g/[v_g]$ as an Alexandroff space is weakly homotopy equivalent to a circle.
\end{example}

We also state an example of a homeomorphism which is isotopic to a periodic homeomorphism and which is not topologically equivalent to the flow $g$ in the previous example but whose abstract weak orbit space is homeomorphic to one of $g$.

\begin{example}
Let $v_h$ be the suspension flow on a closed 3-manifold $\mathbb{T}^2_h$ of a homeomorphism $h$ on a torus $\mathbb{T}^2$ whose minimal sets consist of a semi-attracting limit cycle $\gamma_3$, one semi-attracting and semi-repelling limit cycle $\gamma_4$, one semi-repelling limit cycle, and cycles as on the right of Figure~\ref{ex:02}.
Then the abstract weak orbit space $\mathbb{T}^2_h/[v_h] \cong \mathbb{T}^2/[h] = \{ [\gamma_3], \gamma_4, U_3, U_4 \}$ is homeomorphic to the abstract weak orbit space $\mathbb{T}^2_g/[v_h] \cong \mathbb{T}^2/[g] = \{ [\gamma_1], \gamma_2, U_1, U_2 \}$ of the flow $g$ in the previous example.
\end{example}

We state examples of Pseudo-Anosov homeomorphisms.

\begin{example}
Let $v_g$ be the suspension flow on a closed 3-manifold $M$ of a pseudo-Anosov homeomorphism $g$ on a compact connected orientable surface $S$.
Then the abstract weak orbit space $M/[v_g] \cong S/[g]$ contains abstract weak orbits contained in $\mathop{\mathrm{Per}}(v_g)$ and abstract weak orbits contained in dense orbits.
In particular, the abstract weak orbit space $M/[v_g] \cong S/[g]$ is not a singleton.
\end{example}

We state a flow $v$ whose abstract weak orbit space with infinite height with respect to the binary relation $\leq_v$.

\begin{example}\label{ex:anosov}
Let $v_f$ be the suspension flow on a closed 3-manifold $M$ of an Anosov diffeomorphism $f$ on a torus $\mathbb{T}^2$.
Then the binary relation $\leq_{v_f}$ on the abstract weak orbit space $M/[v_f]  \cong \mathbb{T}^2 /[f]$ has infinite height. Indeed, the Anosov diffeomorphism $f$ is topological conjugate to a toral automorphism and so is semi-conjugate to a shift map on a shift space given by finite symbols.
\end{example}

Recall that a topologically transitive flow $v$ on a compact metric space $X$ with $\overline{\mathop{\mathrm{Cl}}(v)} = X$ is chaotic in the sense of Devaney if it is sensitive to initial conditions.

\begin{example}\label{ex:chaotic}
Let $w$ be a flow on a compact metric space $M$ which is chaotic in the sense of Devaney.
Then the abstract weak orbit space $M/[w]$ contains abstract weak orbits contained in $\mathop{\mathrm{Cl}}(w)$ and abstract weak orbits contained in dense orbits.
In particular, the abstract weak orbit space $M/[w]$ is not a singleton.
%
\end{example}

\subsection{Non-completeness of abstract weak orbit spaces and of Reeb graphs}

Roughly speaking, the abstract weak orbit space of a flow has no information of spiral directions  around limit circuits.
Therefore abstract weak orbit space is not a complete invariant of Morse-Smale flows with limit cycles.
In fact, there are two Morse-Smale flows on a torus which are not topologically equivalent to each other, but the abstract weak orbit spaces are isomorphic as in  Figure~\ref{non_complete01}.
\begin{figure}
\begin{center}
\includegraphics[scale=0.3]{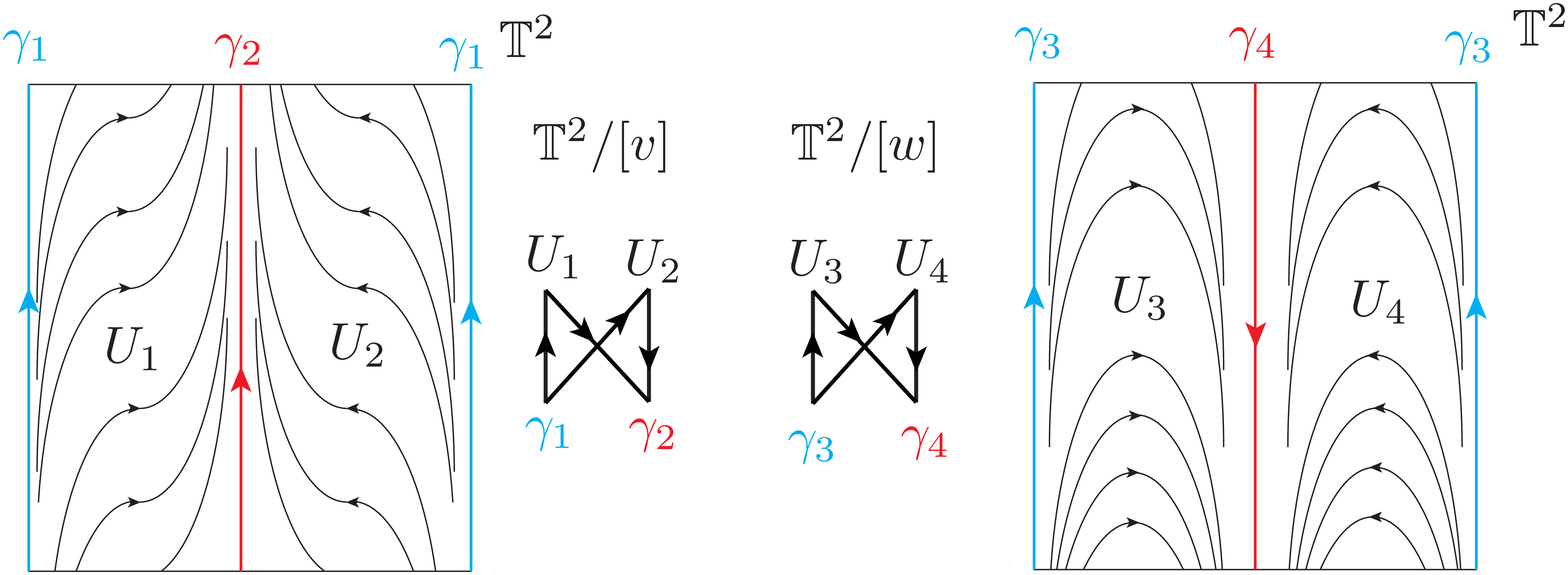}
\end{center}
\caption{Two flows generated by Morse-Smale vector fields on a torus are not topologically equivalent to each other but the abstract weak orbit spaces are isomorphic.}
\label{non_complete01}
\end{figure}
Moreover, an abstract weak orbit space has no information of cyclic orders around multi-saddles out of the boundary of the surface.
Therefore the abstract weak orbit space is not a complete invariant of Hamiltonian flows with finitely many singular points.
Indeed, there are two Hamiltonian flows $v$ and $w$ on a disk $\mathbb{D}^2$ as in Figure~\ref{fig05} that are not topologically equivalent to each other but whose abstract weak orbit spaces are isomorphic to each other.
\begin{figure}
\begin{center}
\includegraphics[scale=0.17]{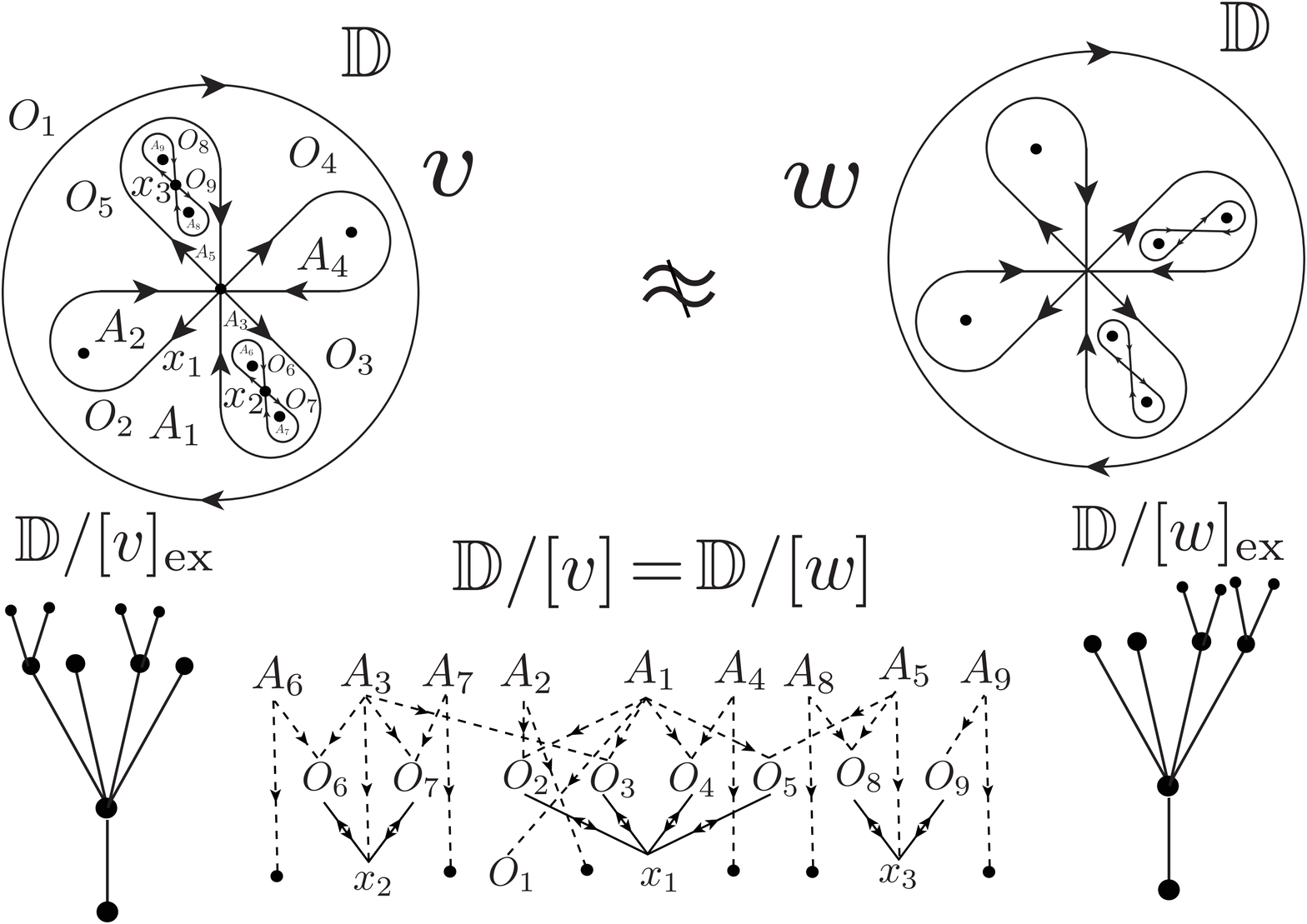}
\end{center}
\caption{Two flows $v$ and $w$ on a disk $\mathbb{D}$ that are not topologically equivalent to each other and whose multi-saddle connection diagrams are not isomorphic as plane graphs but isomorphic as abstract multi-graphs. In particular, the extended weak orbit spaces are isomorphic and are pre-ordered sets with the pre-orders $\leq_\partial$ each of which corresponds to the union of the pre-order $\leq_v$ and the binary relation $\leq_\pitchfork$ as a direct product.}
\label{fig05}
\end{figure}

Moreover, the Reeb graphs of the Hamiltonian flows are also not complete.
Indeed, let $v$ and $w$ be the flows as above in Figure~\ref{fig05}.
Then the abstract weak orbit spaces $\mathbb{D}/[v]$ and $\mathbb{D}/[w]$ are isomorphic.
Since Reeb graphs of Hamiltonian flows on compact surfaces as abstract graphs corresponds to the extended weak orbit spaces $\mathbb{D}/[v]_{\mathrm{ex}}$ and $\mathbb{D}/[w]_{\mathrm{ex}}$ which are quotient spaces of abstract weak orbit spaces $\mathbb{D}/[v]$ and $\mathbb{D}/[w]$, the non-completeness of the abstract weak orbit spaces implies that one of the Reeb graphs.

\bibliographystyle{my-amsplain-nodash-abrv-lastnamefirst-nodot}
\bibliography{../../../bib/yokoyama}
\end{document}